\documentclass[]{amsart}
\usepackage[utf8]{inputenc}  
\usepackage{csquotes}
\usepackage[T1]{fontenc}
\usepackage{graphicx}
\usepackage{tikz}
\usetikzlibrary{positioning}
\usepackage{float}
\usepackage{algorithm}
\usepackage{algorithmic}
\usepackage[utf8]{inputenc}
\usepackage[export]{adjustbox}
\usepackage{wrapfig}
\usepackage{amsfonts}
\usepackage{amsthm}
\usepackage{array}
\usepackage[top=2cm, bottom=2cm, left=2cm, right=2cm]{geometry}
\usepackage{amssymb}
\usepackage{xcolor}
\usepackage{mathtools}
\usepackage{comment}
\usepackage[main=english]{babel}
\usepackage{subcaption}
\usepackage{dsfont}
\usepackage{bbm}
\usepackage{lmodern}
\usepackage{indentfirst}
\usepackage{tabto}
\usepackage{color}
\usepackage{systeme}
\usepackage{eso-pic}
\usepackage{float}
\usepackage{fancyhdr}
\usepackage{appendix}
\usepackage{listings}
\usepackage{pgf, tikz}
\usepackage[utf8]{inputenc}
\usepackage{lastpage}

\definecolor{bf}{rgb}{0,0,0.6} 
\definecolor{mygray}{gray}{0.85}
\definecolor{darkWhite}{rgb}{0.94,0.94,0.94}

\graphicspath{{images/}}

\newcommand{\espcond}[2]{\mathbb{E}\mathopen{}\left[#1\middle|#2\right]}
\newcommand{\esp}[1]{\mathbb{E}\mathopen{}\left[#1\right]}

\newcommand{\reels}{\mathbb{R}}

\newcommand{\tor}{\mathbb{T}}

\usepackage{thmtools,thm-restate}

\newcommand{\FuncDef}[4]{\ensuremath{\left\{\begin{array}{c} #1\longrightarrow #2\\ #3\mapsto #4\\\end{array}\right.}}

\newcommand{\noise}{\theta}
\newcommand{\noiseb}{\tilde{\noise}}
\usepackage[backend=bibtex,natbib=true,style=numeric,maxbibnames=99]{biblatex}
\renewbibmacro{in:}{}
\usepackage{hyperref}

\numberwithin{equation}{section}

\usepackage{amsthm, amssymb}

\DeclareNameAlias{default}{family-given}

\newtheorem{thm}{Theorem}[section]
\newtheorem{definition}[thm]{Definition}
\newtheorem{prop}[thm]{Proposition}
\newtheorem{lemma}[thm]{Lemma}
\newtheorem{hyp}[thm]{Hypothesis}
\newtheorem{corol}[thm]{Corollary}
\theoremstyle{definition}
\newtheorem{remarque}[thm]{Remark}
\newtheorem{exemple}[thm]{Example}

\title{A study of common noise in mean field games}
\author{Charles Bertucci, Charles Meynard}

\begin{document}

\begin{abstract}
  This paper is concerned with the study of mean field games master equation involving an additional variable modelling common noise. We address cases in which the dynamics of this variable can depend on the state of the game, which requires in general additional monotonicity assumptions on the coefficients. We explore the link between such a common noise and more traditional ones, as well as the links between different monotone regimes for the master equation.
\end{abstract}

\maketitle
\setcounter{tocdepth}{1}
\tableofcontents
\section{Introduction}

\subsection{General introduction}
This article is concerned with the study of mean field games (MFG) master equations with a source of common randomness affecting all players. Master equations which are non-linear partial differential equations (PDEs) on the space of probability measures, are usually satisfied by the value function of MFGs. The existence of such a value over arbitrary long intervals of time is in general not clear, as solutions may blow-up in finite time. However, it appears as a natural tool to study MFG, namely in the presence of so-called common noise, whether it be exogenous as in \cite{mfg-common-noise} or coming from a major player as in \cite{Carmona-major-minor,MFG-major-Lions}. The source of noise we consider in this article belongs to neither of those two categories: we investigate situations in which the dynamics of players is affected by an endogenous noise process, which is to say that the dynamic of common noise is affected by the state of the game. We believe it to be natural in many applications \cite{bitcoin-mining,trade-crowding}. This noise process can model environmental factors or even an index followed by all players. 

The existence of a solution to the master equation is intrinsically linked to uniqueness properties of equilibria to the MFG, the very concept of value being undefined otherwise. Uniqueness almost always arises from monotonicity conditions. We investigate the wellposedness of master equations with endogenous common noise under two different types of monotonicity which are most frequently studied: flat monotonicity and displacement monotonicity.

Mean field games were introduced by Lasry and Lions in \cite{LasryLionsMFG,Lions-college}. First through a coupled forward backward system of PDEs and then with the master equation for games with common noise \cite{Lions-college}. The existence of smooth solutions to this equation was obtained in \cite{convergence-problem} under what we call flat monotonicity in this paper (and is often referred to as the Lasry-Lions monotone regime). In a series of paper \cite{disp-monotone-2,globalwellposednessdisplacementmonotone,disp-monotone-1} wellposedness of the master equation was also obtained with purely probabilistic arguments under different monotonicity assumptions the author called displacement monotonicity. Other regime of wellposedness have been identified, most notably the potential regime \cite{cecchin2022weak} in which the master equation can be integrated into a HJB equation. Uniqueness can also be recovered under different monotonicity conditions \cite{OnMonotonicityMFG} that we do not study here.  There has also been a great deal of interest in the definition of weak solutions to MFG. Whether it be at the level of the MFG system \cite{weak-sol-seeger,weaksol-dipmono} or directly on the master equation \cite{ bertucci-monotone,JEP_2021__8__1099_0,lipschitz-sol,cardaliaguet-monotone}, the wellposedness of smooth solutions imposing very strong restrictions on the behavior of the data with respect to the measure argument. In this article, we shall use the concept of Lipschitz solution introduced in \cite{lipschitz-sol}.

Concerning MFG with common noise, we refer the reader to \cite{mfg-common-noise,Lions-college}. A usual modelling of common noise is to linearly impact the state of all players by a random shift. This gives rise to the presence of a second order infinite dimensional term in the master equation \cite{convergence-problem}. The common noise we here model does not raise such difficulty, as it acts on the dynamics of players through a finite dimensional parameter. Such modelling has seen some interest in the literature \cite{bertucci-monotone} especially in applications of MFG to real world phenomenon \cite{bitcoin-mining,trade-crowding}. In \cite{noise-add-variable}, monotonicity conditions were given for the wellposedness of the corresponding problem in finite state space MFG. More recently, some wellposedness results were given in \cite{tangpi2024meanfieldgamescommon} for situations in which the noise and control of players are decoupled.

Let us finally comment on one main problem in MFG theory that is closely related to the wellposedness of the master equation but that we do not treat in this article: the convergence of the corresponding $N-$players game to the MFG system \cite{convergence-problem}. A possible motivation for studying MFG is the definition of approximate controls for game with a finite (but large) number of players \cite{Lacker-cv-equilibria}.

\subsection{Setting and motivation of our study}
To introduce the equation we study in this paper, we start by presenting a MFG. In MFGs, a continuum of players plays a differential games during which they interact with each others through their distribution. 
Given a probability flow $(m_t)_{t\in[0,T]}$ for the distribution of players, the noise process $(\theta_t)_{t\geq 0}$, observed by all players, evolves according to a stochastic differential equation, and a single player of state $(X_t)_{t\geq 0}$ choose their control $(\alpha_t)_{t\in[0,T]}$ in such a way that it solves the following optimal control problem 
\begin{gather}
\label{optimal control problem of a player}
\left\{
\begin{array}{c}
\displaystyle\underset{\alpha}{\sup}\quad \esp{U_0(X^\alpha_T,m_T)-\int_0^T L(X_s^\alpha,\theta_t,\alpha_s,m_s)ds},\\
dX_t^\alpha=-h(X_t^\alpha,\theta_t,\alpha_t,m_t) dt+\sqrt{2\sigma_x} dW_t+\sqrt{2\beta} dB^c_t,\\
d\theta_t= -b(\theta_t,m_t)dt+\sqrt{2\sigma_\theta}dB^\theta_t,
\end{array}
\right.
\end{gather}
for some coefficients $U_0,L,h,b$ and $(W_s)_{s\geq 0}$ a Brownian motion independent of $(B^c_s,B^\theta_s)_{s\geq 0}$. The process $(B^c_s,B^\theta_s)_{s\geq 0}$ is a Brownian motion shared among all players and is the source of common noise associated to the filtration $\mathcal{F}_t=\sigma\left((B^c_s,B^\theta_s)_{s\leq t}\right)$. The Brownian motion $(B^c_s)_{s\geq 0}$ is associated to a standard modelling of common randomness in MFG that acts by translating the state of all players with the same shift, see \cite{mfg-common-noise,Lions-college}. In what follows, we refer to this type of noise as additive common noise. In contrast, $(W_t)_{t\geq 0}$ correspond to the noise that is unique to each player. All players are affected by different Brownian motions independent of each other, this individual noise is usually called idiosyncratic noise.  The main novelty of this paper is the source of noise $(B^\theta_s)_{s\geq 0}$ and the associated stochastic process $(\theta_s)_{s\geq 0}$, common to all players and whose randomness has a direct effect on the coefficients of the game.

A mean field equilibrium, whenever it exists, is characterized by a control $(\alpha^*_t)_{t\in [0,T]}$, optimal for \eqref{optimal control problem of a player} such that 
\[m_t=\mathcal{L}\left(X^{\alpha^*}_t|\mathcal{F}_t\right).\]
The former being equivalent to writing that players are indeed distributed along $m_t$ for each realization of the common noise. Due to the presence of common noise, expressing this equilibrium in terms of a forward-backward system is depreciated, and we focus rather on the master equation 
\begin{equation}
\label{eq: ME with both common noise}
\left\{
\begin{array}{c}
\displaystyle -\partial_t U+H(x,\noise,m,\nabla_xU(t,x,\noise,m))+b(\noise,m)\cdot\nabla_\noise U-\sigma_\noise\Delta_\noise U\\
\displaystyle -(\sigma_x+\beta)\int \text{div}_{y}[D_mU](t,x,\noise,m,y)m(dy)-(\sigma_x+\beta)\Delta_x U\\
\displaystyle+\int D_m U(t,x,\noise,m,y)\cdot D_p H(y,\noise,m,\nabla_x U(t,y,\theta,m))m(dy)\\
\displaystyle -2\beta \int \text{div}_{x}[D_m U](t,x,m,y)]m(dy)-\beta \int \text{Tr}\left[D^2_{mm}U(t,x,m,y,y')\right]m(dy)m(dy')=0,\\
\text{ for } (t,x,\noise,m) \in (0,T)\times \reels^d \times\reels^n\times  \mathcal{P}(\reels^d),\\
U(T,x,m)=U_0(x,\noise,m) \text{ for } (x,\noise,m) \in \reels^d\times\reels^n \times  \mathcal{P}(\reels^d).
\end{array}
\right.
\end{equation}
The solution $U$ of this equation is the value function of the game, the derivation of this equation from the control problem is done explicitly in  \cite{derivation-ME,Lions-college}. The hamiltonian
\[H(x,\theta,m,p)=\underset{\alpha}{\inf}\left\{L(x,\theta,\alpha,m)+h(x,\theta,\alpha,m)\cdot p\right\}\]
comes from the optimization problem each player solves against the crowd of others players. The term $D_m U$ refers to the derivative of $U$ with respect to the measure argument (more precise statements will be made on this notion later on). The presence of terms involving this derivative reflects the evolution of the distribution of players throughout the game. Finally, $U(t,x,\noise,m)$ is the value of the game for a player in state $x$ whenever the time elapsed since the beginning of the game is $t$, the distribution of all players is $m$ and the realization of the exterior stochastic process is $\noise$. Whenever $U$ does not depend on $\noise$ we get back the typical master equation associated to a mean field game, see for instance \cite{convergence-problem}. Let us comment a bit more on modelling concerns behind this variable $\noise$. In various applications, we believe it is natural for the drift $b$ of such process to depend on the distribution and choices of players. This justifies that $b$ may depend on $\nabla_x U$ and be of the form
\[b(\theta,m)=\Tilde{b}(\theta,m,\nabla_x U(t,\cdot,\theta,m))\]
as this noise process is affected by the general direction (or control) of players. Moreover, the difficulty raised by the presence of additive common noise and by the common noise process $\theta$ are orthogonal, which is why we  start by studying the wellposedness of the following master equation 
\begin{equation}
\label{eq: general ME}
\left\{
\begin{array}{c}
\displaystyle -\partial_t U+H(x,\noise,m,\nabla_xU(t,x,\noise,m))+b(\noise,m,\nabla_x U(t,\cdot,\theta,m))\cdot\nabla_\noise U-\sigma_x\Delta_x U-\sigma_\noise\Delta_\noise U\\
\displaystyle -\sigma_x\int \text{div}_{y}[D_mU](t,x,\noise,m,y)m(dy)
\displaystyle+\int D_m U(t,x,\noise,m,y)\cdot D_p H(y,\noise,m,\nabla_x U)m(dy)=0,\\
\text{ for } (t,x,\noise,m) \in (0,T)\times \reels^d \times\reels^n\times  \mathcal{P}(\reels^d),\\
U(T,x,m)=U_0(x,\noise,m) \text{ for } (x,\noise,m) \in \reels^d\times\reels^n \times  \mathcal{P}(\reels^d).
\end{array}
\right.
\end{equation}
and extend our results to the situation $\beta\neq 0$ at the end. Whenever $b$ depends on $\theta$ only, we will see that the general theory of MFG adapts quite well. However, as soon as $b$ depends on the measure argument or on the value function, classical arguments fail and the theory must be adapted. Indeed, without further assumptions on the coefficients singularities may appear in finite time as observed in \cite{noise-add-variable}.

\subsection{Organization of the paper}
In each Section of this article, our analysis relies on the concept of Lipschitz solutions to the master equation \cite{lipschitz-sol} that we recall in Section \ref{section Lipschitz solutions}. In Section \ref{section LL monotone solutions} we tackle the wellposedness of \eqref{eq: general ME} in the Lasry-Lions monotone regime. By adapting arguments from \cite{convergence-problem} we show the existence of a solution for measure belonging to $\mathcal{P}_1(\reels^d)$. In contrast, in Section \ref{section L2 monotone solutions} we consider the existence of solutions in the space $\mathcal{P}_2(\reels^d)$ by expending on Lions's Hilbertian approach \cite{Lions-college}. We highlight the equivalence between this method and displacement monotonicity. Let us also mention that the analysis is carried in a slightly more general setting to account not just for \eqref{eq: general ME} but also for equations stemming from mean field forward backward differential equations (FBSDE) \cite{mean-field-FBSDE} and some extended MFG \cite{extended-mfg}. Finally, in Section \ref{section common noise} we tackle the original problem of \eqref{eq: ME with both common noise}. Our treatment of additive common noise appears quite general and practical to handle the second order term coming from such modelling.

\subsection{Notation}
\begin{enumerate}
    \item[-] Consider $(d,q)\in \mathbb{N}\times\reels$, $d,q\geq 1$, and let $\mathcal{P}(\reels^d)$ be the set of (Borel) probability measures on $\reels^d$, we use the usual notation 
\[\mathcal{P}_q(\reels^d)=\left\{\mu\in \mathcal{P}(\reels^d), \quad \int_{\reels^d} |x|^q \mu(dx)<+\infty\right\},\]
for the set of all probability measures with a finite $q-$th moment.
\item[-] For two measures $\mu,\nu\in\mathcal{P}(\reels^d)$ we define $\Gamma(\mu,\nu)$ to be the set of all probability measures $\gamma\in\mathcal{
P}(\reels^{2d})$ satisfying 
\[\gamma(A\times \reels^d)=\mu(A) \quad \gamma(\reels^d\times A)=\nu(A),\]
for all Borel set $A$ on $\reels^d$. 
\item[-]The Wasserstein $q-$distance between two measures belonging to $\mathcal{P}_q(\reels^d)$ is defined as
\[\mathcal{W}_q(\mu,\nu)=\left(\underset{\gamma\in \Gamma(\mu,\nu)}{\inf}\int_{\mathbb{\reels}^{2d}} |x-y|^q\gamma(dx,dy)\right)^{\frac{1}{q}}.\]
In what follows $\mathcal{P}_q(\reels^d)$ is always endowed with the associated Wasserstein distance, $(\mathcal{P}_q,\mathcal{W}_q)$ being a complete metric space. Let us also remind that for $\mathcal{W}_1$ there is the following dual representation 
\[\mathcal{W}_1(\mu,\nu)=\underset{\varphi, \|\varphi\|_{Lip}\leq 1}{\sup}\left(\int_{\reels^d}\varphi(x)\mu(dx)-\int_{\reels^d}\varphi(x)\nu(dx)\right),\]
for $\|\varphi\|_{Lip}$ the usual Lipschitz semi-norm
\[\|\varphi\|_{Lip}=\underset{\underset{x\neq y}{(x,y)}}{\sup}\frac{|\varphi(x)-\varphi(y)|}{|x-y|}.\]

\item[-] We say that a function $U:\mathcal{P}_q(\reels^d)\to \reels^d$ is Lipschitz if
\[\exists C\geq 0, \quad \forall(\mu,\nu)\in\left(\mathcal{P}_q(\reels^d)\right)^2, \quad |U(\mu)-U(\nu)|\leq C\mathcal{W}_q(\mu,\nu).\]
In which case we use the notation $\|U\|_{Lip(\mathcal{W}_q)}$ for its Lipschitz semi-norm:
\[\|U\|_{Lip(\mathcal{W}_q)}= \sup_{\substack{\mu,\nu\in\mathcal{P}_q(\mathbb{T}^d)\\ \mu\neq \nu}}\frac{|U(\mu)-U(\nu)|}{\mathcal{W}_q(\mu,\nu)}.\]

\item[-] For a measure $m\in\mathcal{P}(\reels^{d})$, and $1\leq i\leq d$, we use the notation $\pi_i m$ for the marginal of $m$ over its first $i$ variables and $\pi_{-i}$ for its marginal over the last $i$ variables. Which is to say that for any borel subset $A$ of $\reels^i$
\[\pi_i m(A)=m(A\times \reels^{d-i}) \quad \pi_{-i} m(A)=m(\reels^{d-i}\times A).\]
\item[-] Consider $(\Omega, \mathcal{F},\mathbb{P})$ a probability space,
\begin{itemize}
    \item[-] for $q\geq 0$ we define 
\[L^q(\Omega, \reels^d)=\left\{ X: \Omega\to \reels^d, \quad \esp{|X|^q}<+\infty\right\}.\]
\item[-] Whenever a random variable $X:\Omega\to \reels^d$ is distributed along $\mu\in \mathcal{P}(\reels^d)$ we use equivalently the notations $\mathcal{L}(X)=\mu$ or  $X\sim\mu$. 
\item[-] If another probability measure $\mathbb{Q}$ is defined on $(\Omega,\mathcal{F})$, the expectation under $\mathbb{Q}$ of a random variable $X:\Omega\to \reels^d$ is noted $\mathbb{E}^\mathbb{Q}\left[X\right]$.
\end{itemize}
\item[-] $\mathcal{M}_{d\times n}(\reels)$ is the set of all matrices of size $d\times n$ with reals coefficients, with the notation $\mathcal{M}_n(\reels)\equiv \mathcal{M}_{n\times n}(\reels)$. $\mathcal{S}_n(\reels)$ is the subset of symmetric matrices of size $n$. 
\item[-] $B_d$ refers to the block matrix 
\[B_d=\left(\begin{array}{cc}
     I_d & I_d  \\
     I_d& I_d
\end{array}
\right),\]
with $I_d$ the identity matrix of size $d$. 
\item[-] $C_b(\reels^d,\reels^k)$ is the set of all continuous bounded functions from $\reels^d$ to $\reels^k$.
\item[-] For $f:\reels^+\times\reels^d\times \reels^n\times \mathcal{P}(\reels^d)\to \reels^d$ and a function $b:\reels^n\times \reels^d\times C(\reels^d,\reels^d)\to \reels^d$ we use either the notations 
$(t,\theta,m)\mapsto b(\theta,m,f(t,\cdot,\theta,m))$ or $(t,\theta,m)\mapsto b[f](t,\theta,m)$ to say $b$ depends functionally on $y\mapsto f(t,y,\theta,m)$ for a given $(t,\theta,m)\in \reels^+\times\reels^n\times \mathcal{P}(\reels^d)$.
\item[-] We say that $U:\mathcal{P}(\reels^d)\to \reels$ is derivable at $m$ if there exists a continuous map $\psi:\reels^d\to \reels$ such that
\[\forall \nu\in\mathcal{P}(\reels^d) \quad \underset{\varepsilon\to 0}{\lim} \frac{U(m+\varepsilon(\nu-m))-U(m)}{\varepsilon}=\int_{\reels^d}\psi(x)(\nu-m)(dx).\]
In this case we define the derivative of $U$ at $m$, $\nabla_mU(m)$ to be one such $\psi$ such that 
\[\int_{\reels^d} \nabla_mU(m)(y)m(dy)=0.\]
Whenever $\nabla_mU(m)$ is differentiable, we define the Wasserstein derivative
\[D_m U(m,y)=\nabla_y \nabla_mU(m)(y).\]
Derivatives of higher order on the space of measures can be defined by induction.
\item[-]  U is said to belong to $C^{1,k}\left(\mathcal{P}(\reels^d)\right)$ whenever for all $m\in\mathcal{P}(\reels^d)$, $\nabla_mU(m):\reels^d\to \reels$ is $C^k$ and the mapping
\[
\left\{
\begin{array}{c}
     \mathcal{P}(\reels^d)\to C^k(\reels^d,\reels) \\
     m\mapsto \left(y\mapsto\nabla_mU(m)(y)\right)
\end{array}
\right.
\]
is continuous for the topology of narrow convergence.
\item[-] Finally, for a measure $\mu$ and a measurable function $f$, we use the notation $f_\#\mu$ for the pushforward of $\mu$ by $f$. In particular, for $\noise\in \reels^d$ $(id_{\reels^d}+\noise)_\#\mu$ is defined as the pushforward of $\mu$ by the map $x\mapsto x+\noise$.
\end{enumerate}
\subsection{Preliminaries}
\label{prelim}
In this article we work on a standard probability space $(\Omega,\mathcal{A},\mathbb{P})$ featuring a collection of Brownian motions. We start with a reminder on the behaviour of functions of measure at points of minimum
\begin{prop}
\label{prop: comparison minimum measure smooth}
    Consider $U\in C^{1,1}\left(\mathcal{P}(\reels^d)\right)$ and suppose that there exists $m_0\in\mathcal{P}(\reels^d)$ such that 
    \[U(m_0)=\underset{\mu\in\mathcal{P}(\reels^d)}{\inf}U(\mu).\]
    Then 
    \begin{gather*}
       \forall \phi\in C_b(\reels^d,\reels^d) \quad  \int_{\reels^d} D_m U(m_0,y)\cdot \phi(y)m_0(dy)=0.
    \end{gather*}
    If furthermore $U\in C^{1,2}\left(\mathcal{P}(\reels^d)\right)$ then
    \[ \forall \phi\in C_b(\reels^d,\reels^d) \quad 
 \int_{\reels^d} \phi(y)\cdot D_y D_mU(m_0,y)\cdot \phi(y) m_0(dy)\geq 0.\]

\end{prop}
A more general result is proved in \cite{bertucci-monotone} Proposition 1.1. This proposition justifies a definition of viscosity solution we will introduce later for a particular equation. Throughout this article we also use a simple adaptation of Lemma 2.3 of \cite{disp-monotone-1}
\begin{prop}
\label{prop: expectation to pointwise}
    Let $f:\FuncDef{\mathcal{P}_2(\reels^d)\times \reels^{2d}}{\reels}{(\mu,x,y)}{f(\mu,x,y)}$ be a continuous function, symmetric in $(x,y)$, such that for all $\mu\in\mathcal{P}_2(\reels^d)$ $x\in\reels^d$ and $X,Y\in L^2(\Omega,\reels^d)$ both with law $\mu$:
    \[\esp{f(\mu,X,Y)}\geq 0 \text{ and } f(\mu,x,x)=0.\]
Then the inequality holds pointwise:
\[\forall \mu\in\mathcal{P}_2(\reels^d)\quad (x,y)\in\reels^{2d} \quad f(\mu,x,y)\geq 0.\]
\end{prop}
It suffices to notice that the proof they gave can be carried out under slightly more general assumptions.

Finally, we remind this singular Grönwall's Lemma of which a proof can be found in \cite{singular-gronwall}. 
\begin{restatable}{lemma}{backwardgronwall}
\label{Lemma: backward gronwall}
    Let $(u_s)_{s\in[0,t]}$ be a positive bounded function such that for some $p\in[0,1)$
    \[\forall l\leq t \quad u_l\leq C+\int_0^l \frac{1}{(l-s)^p}u_s ds.\]

    Then there exist a constant $C_t$ depending on $C$ and $t$ only such that
    \[\forall s\leq t \quad u_s\leq C_te^{C_t s}.\]
\end{restatable}

\section{Notion of solution and monotonicity}
\label{section Lipschitz solutions}
In this article, to show the existence of solutions to \eqref{eq: general ME} for arbitrary time horizon $T$, we will work as follows: we first show that there exists a solution until an eventual blow-up time $T_c$, then we show that this solution does not blow-up in finite time, thus implying global existence. Local in time existence is obtained by means of Lipschitz solutions \cite{lipschitz-sol,extensioncauchylipschitztheorem}, which is the notion of solution to the master equation we will be using throughout this article. We remind the key ideas and main properties of such solutions thereafter. To prove that solutions do not blow up later on, we will require some monotonicity assumptions on the data. We will show that wellposedness holds under two different regimes of monotonicit that we remind at the end of this section.
\subsection{Lipschitz solutions}
\subsubsection{A first look}

\quad 

Consider a solution $U$ to the master equation
\begin{equation}
\label{eq: ME Wq}
\left\{
\begin{array}{c}
\displaystyle \partial_t U+H(x,\noise,m,\nabla_xU(t,x,\noise,m))+b[\nabla_x U](t,\noise,m)\cdot\nabla_\noise U-\sigma_x\Delta_x U-\sigma_\noise\Delta_\noise U\\
\displaystyle-\sigma_x\int \text{div}_{y}[D_mU](t,x,\noise,m,y)m(dy)+\int D_m U(t,x,\noise,m,y)\cdot D_p H(y,\noise,m,\nabla_x U)m(dy),\\
\text{ for } (t,x,\noise,m) \in (0,T)\times \reels^d \times\reels^n\times  \mathcal{P}_q(\reels^d),\\
U(0,x,m)=U_0(x,\noise,m) \text{ for } (x,\noise,m) \in \reels^d\times\reels^n \times  \mathcal{P}_q(\reels^d),
\end{array}
\right.
\end{equation}
in which time has been reversed for convenience. Should $\nabla_xU$ be known and Lipschitz, $U$ can, at least formally, be seen as satisfying a linear equation with Lipschitz coefficients. $U$ is then easily defined by integrating through the characteristics. This is the main idea behind Lipschitz solutions: Lipschitz regularity of $\nabla_x U$ is sufficient to be able to define solutions to the master equation. The interest of such solutions is twofold. First, we will see that uniqueness always hold in this class of solutions. Second, since we do not ask for \eqref{eq: ME Wq} to hold pointwise but define instead solutions by integrating through non-linear characteristics,  for solutions to be well-defined locally in time we do not need as much regularity on the data. In particular, no differentiability with respect to the measure argument is required. We refer the interested reader to \cite{lipschitz-sol} as we present the main theory concisely thereafter.
\subsubsection{Definition and basic properties}
Since we are interested in the regularity of $\nabla_x U$, a key idea in Lipschitz solutions is to work directly on $W=\nabla_x U$ by deriving \eqref{eq: ME Wq} with respect to $x$. We expect $W$ to be a solution of 
\begin{equation}
\label{eq: lip sol Pq}
\left\{
\begin{array}{c}
\displaystyle\partial_t W+D_p H(x,\noise,m,W)\cdot \nabla_x W+b[W](t,\noise,m)\cdot\nabla_\noise W-\sigma_x \Delta_x W-\sigma_\noise \Delta_\noise W\\
\displaystyle +\int_{\reels^d} D_pH (y,\!\noise,\!m,\!W)\cdot D_m W( t,\!x,\!\noise,\!m)(y)m(dy)-\sigma_x\int_{\reels^d} \text{div}_y (D_mW(t,\!x,\!\noise,\!m)(y))m(dy)\\
\displaystyle =-D_x H(x,\noise,m,W) \text{ in } (0,T)\times\reels^d\times\reels^n\times\mathcal{P}_q(\reels^d),\\
\displaystyle W(0,x,\noise,m)=\nabla_x U_0(x,\noise,m) \text{ for } (x,\noise,m)\in\reels^d\times\reels^n\times\mathcal{P}_q(\reels^d).
 \end{array}
 \right.
\end{equation}
The proper definition of a solution $W$ to this equation, with Lipschitz regularity only, is then done through a fixed point. Consider a linear version of \eqref{eq: lip sol Pq}
\begin{equation}
\label{eq: linear lip sol}
\left\{
\begin{array}{c}
\displaystyle\partial_t V+A(t,x,\noise,m)\cdot \nabla_x V+B(t,\noise,m)\cdot\nabla_\noise V-\sigma_x \Delta_x V-\sigma_\noise \Delta_\noise V\\
\displaystyle +\int_{\reels^d} A(t,y,\noise,m)\cdot D_m V(t,x,\noise,m)(y)m(dy)-\sigma_x \int_{\reels^d} \text{div}_y (D_mV(t,x,\noise,m)(y))m(dy)\\
\displaystyle =E(t,x,\noise,m) \text{ in } (0,T)\times\reels^d\times\reels^n\times\mathcal{P}_q(\reels^d),\\
\displaystyle V(0,x,\noise,m)=V_0(x,\noise,m) \text{ for } (x,\noise,m)\in\reels^d\times\reels^n\times\mathcal{P}_q(\reels^d),
 \end{array}
 \right.
\end{equation}
for 4 vector fields $(A,B,E,V_0):(0,T)\times\reels^d\times\reels^n\times\mathcal{P}_q(\reels^d)\longrightarrow (\reels^d)^3\times\reels^n$. 
A solution of this linear system is given by integrating along the characteristics for all $t<T$
\begin{equation}
\label{eq: Feynman-kac measure}
\begin{array}{c}
\displaystyle V(t,x,\noise,\mu)=\espcond{V_0(X_t,\noise_t,m_t)+\int_0^t E(t-s,X_s,\noise_s,m_s)ds}{X_0=x,\noise_0=\noise,m_0=\mu},\\
\displaystyle dX_s=-A(t-s,X_s,\noise_s,m_s)ds+\sqrt{2\sigma_x}dB_s,\\
\displaystyle d\noise_s=-B(t-s,\noise_s,m_s)ds+\sqrt{2\sigma_\noise}dB^\noise_s,\\
dm_s=\left(-\text{div}\left(A(t-s,x,\noise_s,m_s)m_s\right)+\sigma_x \Delta_x m_s\right)ds,
\end{array}
\end{equation}
for $(B_s,B^\noise_s)_{s\geq 0}$ a $d+n$ dimensional Brownian motion. As soon as $(x,\noise,\mu)\mapsto (A,B)(s,x,\noise,\mu)$ is Lipschitz (with respect to the $\mathcal{W}_q$ distance for the measure argument) uniformly in $s\in (0,T)$, this system of a coupled SDE and SPDE $(X_s,\noise_s,m_s)_{s\in[0,t]}$ is well-defined (see \cite{probabilistic-mfg}). If $(V_0,E)$ is also Lipschitz, then it is easy to see that the function $V$ given by this formula is in fact Lipschitz itself. Consider now the functional $\psi$ that to a Lipschitz $(A,B,E,V_0)$ associate this function V, i.e.
\[\psi(T,A,B,E,V_0)=\FuncDef{(0,T)\times\reels^d\times \reels^n\times\mathcal{P}_q(\reels^d)}{\reels^d,}{(t,x,\noise,\mu)}{V(t,x,\noise,\mu),}\]
where $V$ is given by \eqref{eq: Feynman-kac measure}. The definition of a Lipschitz solution to \eqref{eq: lip sol Pq} is given by a fixed point of this operator. 
\begin{definition}
\label{def: lipschitz sol W2}
Let $T>0$, $W:[0,T)\times \reels^d\times\reels^n\times\mathcal{P}_q(\reels^d)\to \reels^d$ is said to be a Lipschitz solution of \eqref{eq: lip sol Pq} if :
\begin{itemize}
    \item[-] $W$ is Lipschitz in $(x,\noise,\mu)$ uniformly in $t\in[0,\alpha]$ for all $\alpha$ in $[0,T)$.
    \item[-] for all $t<T$: \[W=\psi(t,D_p H(\cdot,W),b[W],-D_x H(\cdot,W),\nabla_x U_0).\]
\end{itemize}
\end{definition}
Following what we just saw on \eqref{eq: linear lip sol}, such a definition makes sense provided $(D_pH,D_xH,\nabla_x U_0,b)$ are at least Lipschitz in some sense. We shall work under the following assumptions

\begin{hyp}
\label{hyp: Lipschitz Wq}
$\exists C>0 \quad \forall (x,y)\in(\reels^d)^2, (\noise,\noise')\in(\reels^n)^2, (\mu,\nu)\in(\mathcal{P}_q(\reels^d))^2, (p,q)\in(\reels^d)^2,$
    \begin{enumerate}
    \item[-] $|\nabla_x U_0(x,\noise,\mu)-\nabla_x U_0(y,\noise',\nu)|\leq C\left(|x-y|+|\noise-\noise'|+\mathcal{W}_q(\mu,\nu)\right).$
        \item[-] $|D_p H(x,\noise,\mu,p)-D_p H(y,\noise',\nu,q)|\leq C\left(|x-y|+|\noise-\noise'|+|p-q|+\mathcal{W}_q(\mu,\nu)\right).$
        \item[-]$|D_x H(x,\noise,\mu,p)-D_xH(y,\noise',\nu,q)|\leq C\left(|x-y|+|\noise-\noise'|+|p-q|+\mathcal{W}_q(\mu,\nu)\right).$
        \item[-]For any two continuous Lipschitz bounded $f,g: \reels^d\to \reels^d$
        \[|b(\noise,\mu,f(\cdot))-b(\noise',\nu,g(\cdot))|\leq C\left(|\noise-\noise'|+(1+\|f\|_{Lip}+\|g\|_{Lip} )\mathcal{W}_q(\mu,\nu)+\|f-g\|_\infty\right).\]
    \end{enumerate}
\end{hyp}
It is necessary to assume the functional dependency of $b$ is Lipschitz in the $\|\cdot\|_\infty$ norm, the norm for which the fixed point definition of Lipschitz solutions is carried out. Letting the Lipschitz constant of $b$ in $\mathcal{W}_q$ depend on the Lipschitz constant of its functionnal argument allows to consider functions of the form 
\[b[f](\mu)=\int f(x)\mu(dx)\]
which are natural in practice. 
\begin{thm}
\label{lip sol b(p)}
       Under Hypothesis \ref{hyp: Lipschitz Wq} the following hold:
    \begin{enumerate}
        \item[-] There is always an existence time  $T>0$ such that there is a unique solution of \eqref{eq: lip sol Pq} in the sense of Definition \ref{def: lipschitz sol W2} on $[0,T)$.
        \item[-] There exist $T_c>0$ and a maximal solution $W$ defined on $[0,T_c[$ such that for all Lipschitz solutions $V$ defined on $[0,T)$: $T\leq T_c$ and $W|_{[0,T)}\equiv V$.
        \item[-] If $T_c<\infty$ then $\underset{t\to T_c}{\lim} \|W(t,\cdot)\|_{Lip}=+\infty$.
    \end{enumerate}
\end{thm}
This is both a uniqueness and existence result. Uniqueness is a consequence of the Lipschitz regularity of solutions, formally there can be at most one sufficiently smooth solution. This theorem is slightly different from predating existence results on Lipschitz solution. The function $W$ is here Lipschitz for the $\mathcal{W}_q$ distance, while in \cite{lipschitz-sol}, the authors considered Lipschitz solutions with respect to $\mathcal{W}_1$. The extension is very straightforward, the only real difference being that we now need the following estimate for the distance between measures in $\mathcal{W}_q$. 
\begin{restatable}{lemma}{wassersteindistanceestimate}
\label{Appendix lemma: W2 distance estimate}
    let $q\geq 1$ and consider $(\mu^i_t)_{t\geq 0}$ weak solution to
    \[\partial_t \mu^i_t=-\text{div}\left(c_i(t,x)\mu^i_t\right)+\sigma_x \Delta \mu^i_t\text{ in }(0,T)\times\reels^d, \quad \mu^i_0=\mu_0\in \mathcal{P}_q(\reels^d),\]
    for $i\in\{1,2\}$ and two drifts $c_1,c_2$:$[0,T)\times\times\reels^d\to \reels^d$ Lipschitz in $x$ uniformly in $t\in[0,T)$.
    Then there exist C depending only on $\|c_1\|_{Lip}\wedge \|c_2\|_{Lip}$ and T such that
    \[\forall t<T \quad \mathcal{W}_q(\mu^1_t,\mu^2_t)\leq C\left(\int_0^t \|c_1(s,\cdot)-c_2(s,\cdot)\|^q_\infty ds\right)^{\frac{1}{q}}.\]
\end{restatable}
\begin{proof}
    This well-known result can be obtained by taking $(X^i_t)_{t\geq 0}$ solution to 
    \[dX^i_t=c^i(t,X^i_t)dt+\sqrt{2\sigma_x}dW_t \quad X^i_0=X_0\in L^q(\Omega,\reels^d),\]
    for $(W_t)_{t\geq 0}$ a Brownian motion and $\mathcal{L}(X_0)=\mu_0$ and applying Grönwall's Lemma to the quantity
    \[t\mapsto \esp{|X^1_t-X^2_t|^q}.\]
\end{proof}
Before we explain how $W$ relates to the solution of the master equation, let us present the following property that will be used extensively throughout this paper
\begin{lemma}
\label{prop: representation of lipschitz solutions}
    Let $W$ be a Lipschitz solution to equation \eqref{eq: lip sol Pq} on $[0,T)$ for some $T>0$. Then, for all $s<t<T$, $(x,\noise,\mu)\in\reels^d\times\reels^n\times\mathcal{P}_q(\reels^d)$, $W$ satisfy the following dynamic programming principle 
        \begin{equation}
        \label{Def: Lipschitz solution b(p)}
\begin{array}{c}
\displaystyle W(t,x,\noise,\mu)=\esp{W(s,X_{t-s},\noise_{t-s},m_{t-s})-\int_0^{t-s}D_xH(X_u,\noise_u,m_u,W(t-u,X_u,\noise_u,m_u))du},\\
\displaystyle dX_s=-D_p H(X_s,\noise_s,m_s,W(t-s,X_s,\noise_s,m_s))ds+\sqrt{2\sigma_x}dB_s \quad X_0=x,\\
\displaystyle d\noise_s=-b[W](t-s,\noise_s,m_s)ds+\sqrt{2\sigma_\noise}dB^\noise_s \quad \noise_0=\noise,\\
dm_s=\left(-\text{div}\left(D_pH(x,\noise_s,m_s,W(t-s,x,\noise_s,m_s))m_s\right)+\sigma_x \Delta_x m_s\right)ds \quad m_0=\mu,
\end{array}
\end{equation}
\end{lemma}
\begin{proof}
By definition of a Lipschitz solution
\begin{equation*}
\displaystyle W(t,x,\noise,\mu)=\esp{\nabla_x U_0(X_t,\noise_t,m_t)-\int_0^t D_x H(X_s,\noise_s,m_s,W(t-s,X_s,\noise_s,m_s))ds}.
\end{equation*}
Let us first remark that $(X_t,\noise_t,m_t)_{t\geq 0}$ satisfies the flow property. This is a natural consequence of the uniqueness of solutions. As a consequence, letting $(\mathcal{F}_{u})_{u\geq 0}$ be the completed filtration associated to $(B_u,B^\noise_u)_{u\geq 0}$, for any $s\leq t$, it holds that
\[\espcond{W_0(X_t,\noise_t,m_t)-\int_{t-s}^tD_x H(X_u,\noise_u,m_u,W(t-u,X_u,\noise_u,m_u))du}{\mathcal{F}_{t-s}}=W(s,X_{t-s},\noise_{t-s},m_{t-s}) \quad a.s.\]
This is obtained by first applying the flow property and then the definition of a Lipschitz solution. The dynamic programming principle then follows by taking the expectation on both sides. 
\end{proof}
So long as a Lipschitz solution $W$ to \eqref{eq: lip sol Pq} exists, we define the solution of the original problem \eqref{eq: ME Wq} by integrating along the characteristics of the equation with $\nabla_x U$ replaced with $W$. Which is to say $U$ is defined through
\begin{gather}
    \nonumber U(t,x,\noise,m)=\esp{U_0(x+\sqrt{2\sigma_x}B_t,\noise_t,m_t)-\int_0^t \tilde{H}(t-s,x+\sqrt{2\sigma_x}B_t,\noise_s,m_s)ds},\\
    \label{definition: U for Lipschitz solutions}
    d\noise_s=-b[W](t-s,\noise_s,m_s)+\sqrt{2\sigma_\noise}dB^\noise_s \quad \noise_0=\noise,\\
    \nonumber dm_s=\left(-\text{div}\left(D_pH(x,\noise_s,W(t-s,x,\noise_s,m_s),m_s)m_s\right)+\sigma_x\Delta_x m_s\right)ds\quad m_0=m,
\end{gather}
for $(B_s,B^\noise_s)_{s\geq 0}$ a Brownian motion and with the notation 
\[\tilde{H}(t,x,\noise,m)=H(x,\noise,W(t,x,\noise,m),m).\]
Formally we expect that $W=\nabla_x U$. At this point, it is not totally clear that this is indeed the case. Proving this result is the object of the following Lemma, legitimating the function $U$ defined by \eqref{definition: U for Lipschitz solutions} as the value function, solution of the master equation.
\begin{lemma}
\label{Lemma: gradient U=W}
    Assume Hypothesis \ref{hyp: Lipschitz Wq} holds, and let $U$ be defined through \eqref{definition: U for Lipschitz solutions} for a Lipschitz solution $W$ of \eqref{eq: lip sol Pq} defined on $[0,T_c)$. 
    Then $U$ is $C^1$ in $x$ and the following equality holds
    \[\forall t<T_c \quad \nabla_x U(t,x,\noise,m)=W(t,x,\noise,m) \text{ in } \reels^d\times\reels^n\times\mathcal{P}_q(\reels^d).\]
\end{lemma}
\begin{proof}
We start by treating the case $\sigma_x>0$ and $(x,p)\mapsto (\nabla_x U_0(x,\noise,m),D_xH(x,\noise,m,p),D_pH(x,\noise,m,p))$ is $C^1$ uniformly in the other variables, we will extend the result to the general case at the end. 

\quad 

\noindent\textit{Step 1: under those condition U is $C^1$ in $x$}

We are going to show that under those conditions $s\mapsto \left(x\mapsto W(s,x,\cdot)\right)$ belongs to $C^0([0,t],C^1(\reels^d,\reels))$ for any $t<T_c$. Indeed, it suffices to notice that the fixed point definition of Lipschitz solution can be carried out in the Banach space 
\[C^{1}_x:=\{f:[0,t]\times \reels^d\times \reels^n\times \mathcal{P}_q(\reels^d)\to\reels^d \quad \|f\|_{Lip}<+\infty,\quad  f \text{ is } C^1 \text{ in } x\},\]
endowed with the norm $\|\cdot\|_{C^1_x}$
\[f\mapsto \underset{(s,x,\noise,\mu)\in[0,t]\times\reels^d\times \reels^n\times\mathcal{P}_q(\reels^d)}{\sup}\frac{|f(s,x,\noise,\mu)|}{1+|x|+|\noise|+\mathcal{W}_q(\mu,\delta_{0_{\reels^d}})}+\| f\|_{Lip}.\]
Let $\psi$ be defined as in Definition \ref{def: lipschitz sol W2} and consider
\[\phi: V\to \psi(D_p H(\cdot,V),b(\cdot,V),-D_x H(\cdot,V),\nabla_x U_0).\]
We claim that $\phi$  sends $C^{1}_x$ into itself. Let $V\in C^1_x$, the fact that $\phi(V)$ is Lipschitz has already been proven in \cite{lipschitz-sol}. Now observe that in \eqref{eq: Feynman-kac measure} $(\noise_s,m_s)_{s\geq 0}$ do not depend on $(X_s)_{s\geq 0}$. In light of the assumed regularity in $x$ of our coefficient, we may derive the whole equation with respect to the initial condition $x_0$ of $(X_s)_{s\geq 0}$ showing that $\nabla_x \phi(V)$ is indeed continuous (the wellposedness of the tangent process being trivial since the coefficients are Lipschitz $C^1$). The fixed point existence proof for Lipschitz solution is then done for the $\|\cdot\|_\infty$ as in \cite{lipschitz-sol}. 

\quad 

\noindent\textit{Step 2: equality of $\nabla_x U$ and $W$}

We start by fixing $t<T_c$. As long as $W$ stays Lipschitz, classic estimates on the moments of $(X_s,\noise_s)_{s\in[0,t]}$ given by \eqref{Def: Lipschitz solution b(p)} with respect to initial conditions hold. In \eqref{definition: U for Lipschitz solutions}, $(\theta_s,m_s)_{s\in[0,t]}$ do not depend on $(B_s)_{s\geq 0}$, hence the expectations can be computed separately. In particular this allows to see that $U$ depends on $x$ through a convolution with the heat kernel allowing to define its derivative through
\begin{equation}
\label{equation: representation formula heat kernel W}
\nabla_x U(t,x,\noise,m)=\esp{\nabla_x U_0(x+\sqrt{2\sigma_x}B_t,\noise_t,m_t)-\int_0^{t} \tilde{H}(t-s,x+\sqrt{2\sigma_x}B_s,\noise_s,m_s)\frac{B_s}{\sqrt{2\sigma_x} s} ds}.\end{equation}
From its definition as a Lipschitz solution, $W$ is given by
\[
\begin{array}{c}
\displaystyle W(t,x,\noise,m)=\esp{\nabla_x U_0(X_t,\noise_t,m_t)-\int_0^t D_x H(X_s,\noise_s,m_s,W(t-s,X_s,\noise_s,m_s))ds},\\
\displaystyle dX_s=-D_pH(X_s,\noise_s,m_s,W(t-s,X_s,\noise_s,m_s))ds+\sqrt{2\sigma_x}dB_s \quad X_0=x,\\
\end{array}
\]
 and $(M_s)_{s\in[0,t]}$ defined by
\[M_u=W(t-u,X_u,\noise_u,m_u)-\int_0^u D_xH(X_s,\noise_s,W(t-s,X_s,\noise_s,m_s),m_s) ds\]
is consequently a martingale with respect to the natural filtration $\mathcal{F}$ of the Brownian motion $(B_s,B^\noise_s)_{s\geq 0}$.
Let us now define $Y_u=\frac{1}{\sqrt{2\sigma_x}}\int_0^u D_p H(X_s,\noise_s,W(t-s,X_s,\noise_s,m_s),m_s)dB_s$ and the associated stochastic exponential $\mathcal{E}(Y_u)=Z_u$. Since coefficients are Lipschitz, Novikov's condition is satisfied (see \cite{Karatzas1998} Corollary 5.16), and so the martingale exponential $(Z_s)_{s\in[0,t]}$ is a true martingale. 
By Girsanov theorem $(Z_s)_{s\in[0,t]}$ define a new probability $\mathbb{Q}$ under which $(\tilde{M}_s)_{s\in[0,t]}$ defined as 
\[\tilde{M}_s=M_s-\langle M,Y\rangle_s\]
is a martingale and $(\tilde{B}_s)_{s\in[0,t]}$ is a Brownian motion with
\[\tilde{B}_u=B_u-\frac{1}{\sqrt{2\sigma_x}}\int_0^u D_p H(X_s,\noise_s,W(t-s,X_s,\noise_s,m_s),m_s)ds.\]
Because the integral of a measurable function always has bounded variation, the covariation of $M$ and $Y$ is given by the covariation between $Y$ and $(W(t-s,X_s,\noise_s,m_s))_{s\in[0,t]}$. It is then a consequence of \cite{covariation} Proposition 3.8 that we have 

\[\langle M,Y\rangle_u=\int_0^u D_pH(X_s,\noise_s,W(t-s,X_s,\noise_s,m_s),m_s)\cdot D_x W(t-s,X_s,\noise_s,m_s)ds. \]
 Using the fact that $\tilde{M}_0=M_0=W(t,x,\noise,m)$, we get that 
\[W(t,x,\noise,m)=\mathbb{E}^\mathbb{Q}\left[\nabla_x U_0(X_t,\noise_t,m_t)-\int_0^t (D_xH+D_pH\cdot D_x W)(t-s,X_s,\noise_s,m_s)ds \right].\]
Remark now that $(X_s)_{s\in[0,t]}$ is given by 
\[X_s=x+\sqrt{2\sigma_x}\tilde{B}_u.\]
By an integration by part against the heat kernel, we see that $W$ and $\nabla_x U$ are given by the same formula, which ends to show by uniqueness of Lipschitz solution that we have indeed 
\[\nabla_x U\equiv W.\]

\noindent\textit{Step 3: extension to general Lipschitz data}

We now still assume that $\sigma_x>0$, but remove the smoothness assumption on $(\nabla_x U_0,D_xH,D_pH)$. In this case, we introduce a $C^1$ regularization in $(x,p)$ only (which can be done with any smooth compactly supported kernel) $(\nabla_x U_0^\varepsilon,D_xH^\varepsilon,D_pH^\varepsilon)$ that converges locally uniformly to $(\nabla_xU_0,D_xH,D_pH)$ and consider the associated sequence of Lipschitz solution $W^\varepsilon$. By the above proof, we have $W^\varepsilon=\nabla_x U^\varepsilon$ for any $\varepsilon>0$. Since we can find an existence time which depends only on the Lipschitz constant of the data for Lipschitz solution, there exists $T>0$ such that for any $\varepsilon$, $W^\varepsilon$ is defined at least on $[0,T)$ with a $C^1_x$ norm independent of $\varepsilon$. By Arzelà–Ascoli theorem, $U^\varepsilon$ converges locally uniformly toward a $C^1$ function in $x$ along a subsequence. By stability of Lipschitz solutions, its gradient is a Lipschitz solution of \eqref{eq: lip sol Pq} on $[0,T)$. Since such a solution is unique, $(U_\varepsilon)_{\varepsilon>0}$ has a limit, and the whole sequence converges locally uniformly to this element. Assume now that $T<T_c$, for $T_c$ the maximal time of existence for Lipschitz solutions to \eqref{eq: lip sol Pq} This implies that there exists a constant C such that for any $\eta>0$
\[\|\nabla_x U(T-\eta,\cdot)\|_{Lip}\leq C.\]
by considering a Lipschitz solution with initial condition 
\[W_0=\nabla_xU_0(T-\eta,\cdot),\]
we can extend the $C^1$ function $U$ beyond $T$ by repeating the above procedure. This bootstrapping argument allows us to extend this solution up until $T_c$ which guarantees that 
\[\forall t<T_c \quad \nabla_x U(t,\cdot)=W(t,\cdot).\]
Whenever $\sigma_x=0$ the idea is the same, but we furthermore consider a sequence of Lipschitz solutions with $\sigma^\varepsilon_x=\varepsilon$.
\end{proof}
\subsection{Monotonicity in mean field games}
In MFG, uniqueness of solutions appears as a key element to obtain existence and regularity of solutions to the master equation. In this paragraph we present two notions of monotonicity under which long time uniqueness of solutions is known.

\quad

\noindent \textit{a) Flat monotonicity}
\quad
\begin{definition}
A function $U:\reels^d\times\mathcal{P}(\reels^d)\to \reels$ is said to be flat monotone or just monotone if it is monotone seen as a function of measure only (ie $U:\mathcal{P}(\reels^d)\to\left(\reels^d\to \reels\right)$), which translates to 
\[\forall (\mu,\nu)\in\left(\mathcal{P}(\reels^{d})\right)^2 \quad \langle U(\cdot,\mu)-U(\cdot,\nu),\mu-\nu\rangle \geq 0,\]
where the bracket here design the duality product, which is to say for all continuous bounded function $f$
\[\langle f,\mu\rangle=\int_{\reels^d} f(x)\mu(dx).\]
\end{definition}
This notion of monotonicity is what leads to the so called Lasry-Lions monotone regime. Uniqueness and wellposedness were obtained in this regime for the MFG system \cite{LasryLionsMFG,Lions-college} and directly at the level of the master equation \cite{Lions-college,convergence-problem} for games with common noise. We extend this setting to \eqref{eq: general ME} in section \ref{section LL monotone solutions}.

\quad

\noindent \textit{b) $L^2$ monotonicity}
\quad

In contrast, we now introduce $L^2-$monotonicity a concept closely related to displacement monotonicity \cite{disp-monotone-1,disp-monotone-2} and to Lions' Hilbertian lifting
\begin{definition}
A function $U:\reels^d\times\mathcal{P}_2(\reels^d)\to \reels^d$ is $L^2-$monotone if 
\[\forall (X,Y)\in\left(L^2(\Omega,\reels^d)\right)^2\quad \esp{(U(X,\mathcal{L}(X))-U(Y,\mathcal{L}(Y)))\cdot (X-Y)}\geq 0.\]
\end{definition}
Consider Lions' lift of $U$ on the space of random variables $\hat{U}$ defined by
\[\hat{U}:\left\{
\begin{array}{l}
     L^2(\Omega,\reels^d)\to L^2(\Omega,\reels^d),  \\
     X\mapsto U(X,\mathcal{L}(X)).
\end{array}
\right.
\]
The $L^2-$monotonicity of $U$ is equivalent to the monotonicity of $\hat{U}$ in the classical sense on the Hilbert space $(L^2(\Omega,\reels^d),\langle \cdot,\cdot\rangle|_{L^2})$, with the canonical inner product 
\[\langle \cdot,\cdot\rangle|_{L^2}:\left\{
\begin{array}{l}
     L^2(\Omega,\reels^d)\times L^2(\Omega,\reels^d)\to \reels,  \\
     (X,Y)\mapsto \esp{X\cdot Y},
\end{array}
\right.
\]
hence the name. Let us remark that $L^2-$monotonicity is a strong property that also implies monotonicity in a more traditional sense 
\begin{lemma}
    Let $U:\reels^d\times \mathcal{P}_2(\reels^d)\to \reels^d$ be a continuous $L^2-$monotone function then $\forall \mu\in\mathcal{P}_2(\reels^d)$, $x\mapsto U(x,\mu)$ is monotone in the classical sense, which is to say    
    \[\forall (x,y)\in (\reels^d)^2 \quad (U(x,\mu)-U(y,\mu))\cdot (x-y)\geq 0.\]
\end{lemma}
This is a simple consequence of Proposition \ref{prop: expectation to pointwise}. Wellposedness of the master equation with $L^2-$monotone data was obtained in \cite{Lions-college} for MFG without idiosyncratic noise with the Hilbertian approach. Wellposedness was also obtained for more general MFG under displacement monotonicity \cite{disp-monotone-1,disp-monotone-2}. 
\begin{definition}
    A function $F:\reels^d\times \mathcal{P}_2(\reels^d)\to \reels$ is displacement monotone if and only if its gradient in the space variable $\nabla_x F$ is well defined and $L^2-$monotone. 
\end{definition}

The reason we introduce the concept of $L^2-$monotonicity is that in Section \ref{section L2 monotone solutions}, in which we extend those ideas to the context of \eqref{eq: general ME}, we do not need to make the assumption that our data are gradients, our method of proof being heavily inspired by the Hilbertian formulation of MFG. In fact, one goal of Section \ref{section L2 monotone solutions} is to show that the Hilbertian approach and the displacement monotone setting are one and the same in that they yield the same regularity theory, something which is not always made clear in the literature.

\section{Wellposedness in the flat monotone regime}
\label{section LL monotone solutions}
For simplicity's sake, we place ourselves on the $d$ dimensional flat torus $\tor^d$, but the results we obtain extend easily to the whole space if growth conditions are imposed at infinity. In this setting we work with the $\mathcal{W}_1$ Wasserstein distance and wellposedness of the master equation \eqref{eq: general ME} will be obtained in $\mathcal{P}(\tor^d)$. Throughout most of this section we will also assume that $b$ does not depend on $\nabla_x U$, however the problem remains just as challenging as the main difficulty comes from the coupled evolution of the noise and distribution of players. The master equation we now consider is
\begin{equation}
\label{eq: ME W2 with b periodic}
\left\{
\begin{array}{c}
\displaystyle \partial_t U+H(x,\noise,m,\nabla_xU)+b(\noise,m)\cdot\nabla_\noise U-\sigma_x\Delta_x U-\sigma_\noise\Delta_\noise U\\
\displaystyle-\sigma_x\int \text{div}_{y}[D_mU](t,x,\noise,m,y)m(dy)+\int D_m U\cdot D_p Hdm,\\
\text{ for } (t,x,\noise,m) \in (0,T)\times \mathbb{T}^d \times\mathbb{\reels}^n\times  \mathcal{P}(\mathbb{T}^d),\\
U(0,x,\noise,m)=U_0(x,\noise,m) \text{ for } (x,\noise,m) \in \mathbb{T}^d\times\mathbb{\reels}^n \times  \mathcal{P}(\mathbb{T}^d).
\end{array}
\right.
\end{equation}

We make the following assumption throughout this section
\begin{hyp}
\label{hyp: Lipschitz W1}
$(x,\noise,m,p)\mapsto H(x,\noise,m,p)$ is continuous on $\tor^d\times\reels^n\times\mathcal{P}(\tor^d)\times\reels^d$, $\nabla_x U_0,D_xH,D_p H$ exists and 
\[\exists C>0 \quad \forall (x,y)\in(\tor^d)^2, (\noise,\noiseb)\in(\reels^n)^2, (\mu,\nu)\in(\mathcal{P}(\tor^d))^2, (u,v)\in(\reels^d)^2,\]
    \begin{enumerate}
    \item[-] $|\nabla_x U_0(x,\noise,\mu)-\nabla_xU_0(y,\noiseb,\nu)|\leq C\left(|x-y|+|\noise-\noiseb|+\mathcal{W}_1(\mu,\nu)\right).$
        \item[-] $|D_x H(x,\noise,\mu,u)-D_x H(y,\noiseb,\nu,v)|\leq C\left(|x-y|+|\noise-\noiseb|+|u-v|+\mathcal{W}_1(\mu,\nu)\right).$
        \item[-]$|D_p H(x,\noise,\mu,u)-D_pH(y,\noiseb,\nu,v)|\leq C\left(|x-y|+|\noise-\noiseb|+|u-v|+\mathcal{W}_1(\mu,\nu)\right).$
        \item[-]
        $|b(\noise,\mu)-b(\noiseb,\nu)|\leq C\left(|\noise-\noiseb|+\mathcal{W}_1(\mu,\nu)\right).$
    \end{enumerate}
\end{hyp}
Those are only analogous of Hypothesis \ref{hyp: Lipschitz Wq} for the $\mathcal{W}_1$ Wasserstein distance. Following Section \ref{section Lipschitz solutions}, under Hypothesis \ref{hyp: Lipschitz W1} there always exists a $\mathcal{W}_1$ Lipschitz solution to \eqref{eq: ME W2 with b periodic} on sufficiently small time intervals. 
Throughout this section, we will also use the assumption
 \begin{hyp}
 \label{Hyp: boundedness + linear growth}
    There exists a constant C such that $\forall (x,\noise,\mu,p)\in \tor^d\times \reels^n\times\mathcal{P}(\tor^d)\times\reels^d$,
    \begin{enumerate}
        \item[-] $|\nabla_x U_0(x,\noise,\mu)|\leq C,$
        \item[-] $|D_xH(x,\noise,\mu,p)|+|D_p H(x,\noise,\mu,p)|\leq C(1+|p|),$
        \item[-] $\sigma_x>0$
    \end{enumerate}
 \end{hyp}
 In optimal control this is a somewhat usual assumption that is equivalent to saying the Hamiltonian $(x,\noise,\mu,p)\mapsto H(x,\noise,\mu,p)$ is Lipschitz in $(x,p)$ locally in $p$ and has at most quadratic growth in this last variable and that the noise is non degenerate. Obviously the part on the initial condition implies that the terminal cost of the optimal control problem solved by an individual player is Lipschitz in $x$. Under those conditions the parabolic regularity stemming from the idiosyncratic noise gives the following a priori estimate
\begin{lemma}
\label{Lemma: heat kernel estimate for derivative in x}
    Under Hypothesis \ref{hyp: Lipschitz W1} and Hypothesis \ref{Hyp: boundedness + linear growth},  for any $T\geq 0$ there exists a constant $C_T$ such that for any Lipschitz solution $W$ to \ref{eq: lip sol Pq} on $[0,T_c)$ 
    \[\underset{t\leq T\wedge T_c}{\sup}\|D_xW(t,\cdot)\|_\infty\leq C_T,\]
    holds
\end{lemma}
\begin{proof}
We begin by remarking that under Hypothesis \ref{Hyp: boundedness + linear growth}, $W$ stays bounded over time. Indeed, using the representation formula \eqref{Def: Lipschitz solution b(p)} combined with the linear growth of $D_xH$ in $p$ yields:
\[\|W(t,\cdot)\|_\infty\leq \|W_0\|_\infty+C\int_0^t (1+\|W(t-s,\cdot)\|_\infty ds.\]
From there on, an application of Grönwall's Lemma allows us to conclude to 
\[\|W(t,\cdot)\|_\infty\leq (\|W_0\|_\infty+C t)e^{Ct}\]
for some $C$ depending on $H$.
The key point is that with this bound on $W$ we may make use of the fact that $H$ is locally Lipschitz in $(x,p)$. Indeed, fix $T>0$ and let $t\leq T$, using the notation 
\[\tilde{H}(t,x,\noise,m)=H(x,\noise,W(t,x,\noise,m)),\]
we get from \eqref{equation: representation formula heat kernel W}: 
\begin{align*}
\,|W(t,x,\noise,m)\!-\!W(t,y,\noise,m)|&\!\leq \!\|\nabla_x U_0\|_{Lip}|x\!-\!y|\\
&+\esp{\left|\int_0^t (\tilde{H}(t\!-\!s,x\!+\!\sqrt{2\sigma_x}B_s,\noise_s,m_s)\!-\!\tilde{H}(t\!-\!s,y\!+\!\sqrt{2\sigma_x}B_s,\noise_s,m_s))\frac{B_s}{\sqrt{2\sigma_x }s}ds\right|}.
\end{align*}
Consequently, there exists a constant C depending on $H,T,\sigma_x,b,U_0$ such that

\[\frac{|W(t,x,\noise,m)-W(t,y,\noise,m)|}{|x-y|}\leq \|\nabla_x U_0\|_{Lip}+C\left( 1+\int_0^t \frac{1}{\sqrt{s}}\|D_xW(t-s,\cdot)\|_\infty ds\right).\]
Taking the supremum over $x\neq y$ combined with a simple change of variable yields 
\[\|D_xW(t,\cdot)\|_\infty \leq C+\|\nabla_xU_0\|_{Lip}+C\int_0^t\frac{1}{\sqrt{t-s}}\|D_xW(s,\cdot)\|_\infty ds.\]
We may conclude by using Lemma \ref{Lemma: backward gronwall} to the existence of a constant $C_T$ such that 
\[\forall t\leq T \quad \|D_xW(t,\cdot)\|_\infty \leq C_T.\]
\end{proof}
In MFG, it is classical that parabolic regularity yields estimates \cite{convergence-problem}. Essentially, Lemma \ref{Lemma: heat kernel estimate for derivative in x} guarantees that under Hypothesis \ref{Hyp: boundedness + linear growth} the Lipschitz norm of $W$ with respect to the state variable $x$ stays bounded over time. Since Lipschitz solutions are well defined so long as their Lipschitz norm do not blow up, this ensure any explosive behavior will not come from $x$ as soon as $\sigma_x>0$. Under this assumption, we can focus on showing that $W$ stays Lipschitz with respect to the measure and noise argument, which are more problematic. 
\begin{remarque}
While the presence of non-degenerate idiosyncratic noise (ie $\sigma_x>0$), is a reasonable and practical assumption, it is not a necessary assumption. If we assume this noise to be degenerate, estimates on the space variable can be recovered under different structural conditions on the Hamiltonian \cite{rockafellar,degenerate-idiosyncratic}. To account for such possible extensions, all intermediary lemmas in Section \ref{subsection: existence monotone solutions} are proved without the assumption that $\sigma_x>0$.
\end{remarque}
\subsection{Long time existence of Lipschitz solution}
\label{subsection: existence monotone solutions}
\subsubsection{Autonomous noise process}

\quad 

In this section, we consider the case of an autonomous noise process, 
\[b:\noise\mapsto b(\noise).\]
Because in such situation the noise evolves independently of the other variables and can essentially be seen as a randomization of coefficients, we expect that previous existence theorems for smooth solutions \cite{convergence-problem} could be adapted to this setting by studying the associated MFG system. We take another approach, though the method of proof is still heavily based on an argument of propagation of monotonicity for solutions of the master equation presented in \cite{Lions-college,convergence-problem}. We carry out a proof directly at the level of the master equation, using the notion of Lipschitz solutions. This will be insightful for more general $b$ that we shall consider later on. To make the analysis easier we also make the following assumption
\begin{hyp}
\label{hyp: separated Hamiltonian +LL monotonicity}
    \[H(x,\noise,m,p)=\bar{H}(x,\noise,p)-f(x,\noise,m).\]
Furthermore, $\bar{H}$ is $\alpha-$convex in $p$ for some $\alpha_H>0$ and $U_0,f$ are flat monotone, which is to say that
\[\forall \noise\in\reels^n,  (\mu,\nu)\in\left(\mathcal{P}(\mathbb{T}^d)\right)^2 \quad \langle g(\cdot,\noise,\mu)-g(\cdot,\noise,\nu),\mu-\nu\rangle \geq 0,\]
for $g=U_0$ or $g=f$.
\end{hyp}
As is already well-known, the Hamiltonian being separated in the gradient and measure argument is not a necessary condition, some more technical examples of sufficient conditions were given in Pierre Louis Lions's Lectures at Collège de France. It is however a convenient setting to prove the following lemma, which consists in an adaptation of Theorem 4.3 (or alternatively Proposition 3.2 for a version without common noise) of \cite{convergence-problem} to Lipschitz solutions. 
\begin{lemma}
\label{lemma: monotonicity estimate in wasserstein 1 b(p)}
Under Hypotheses \ref{hyp: Lipschitz W1} and \ref{hyp: separated Hamiltonian +LL monotonicity}, if $b$ depends on $\noise$ only and $U$ is a Lipschitz solution of \eqref{eq: ME W2 with b periodic} on $[0,T_c)$, the following inequality holds 
\begin{gather*}
\forall t<T_c \quad (\noise,\mu,\nu)\in\reels^n\times \left(\mathcal{P}(\tor^{d})\right)^2,\\
\langle U(t,\cdot,\noise,\mu)-U(t,\cdot,\noise,\nu),\mu-\nu \rangle \geq \alpha_H\mathbb{E}\left[\int_0^t\!\int_{\mathbb{T}^d}\left|\nabla_x U(t-s,x,\noise_s,\mu_s)\!-\!\nabla_x U(t-s,x,\noise_s,\nu_s)\right|^2(\mu_s\!+\!\nu_s)(dx)ds \right]\!.
\end{gather*}
with 
\[
\left\{\begin{array}{ll}
d\noise_s=b(\noise_s)ds+\sqrt{2\sigma_\noise}dB^\noise_s& \noise_{s=0}=\noise,\\
d\mu_s=\left(-\text{div}\left(F(x,\noise_s,\mu_s,W(t-s,x,\noise_s,\mu_s))\right)+\sigma_x \Delta_x \mu_s\right)ds & \mu_{s=0}=\mu,\\
d\nu_s=\left(-\text{div}\left(F(x,\noise_s,\nu_s,W(t-s,x,\noise_s,\nu_s))\right)+\sigma_x \Delta_x \nu_s\right)ds & \nu_{s=0}=\nu,\\
\end{array}
\right.
\]
\end{lemma}
\begin{proof}
We start by assuming $\sigma_x>0$. Let us recall that $U$ is given by the Feynman-Kac representation \eqref{definition: U for Lipschitz solutions},
and that $(M_s^\mu)_{s\in[0,t]}$ given by
\begin{gather*}M^\mu_s=U(t-s,X_s,\noise_s,\mu_s)-\int_0^s H(X_u,\noise_u,W(t-u,X_u,\noise_u,\mu_u),\mu_u)du,\\
d\noise_s=b(\noise_s)ds+\sqrt{2\sigma_\noise}dB^\noise_s,\\
d\mu_s=\left(-\text{div}\left(F(x,\noise_s,\mu_s,W(t-s,x,\noise_s,\mu_s))\right)+\sigma_x \Delta_x \mu_s\right)ds\\
X_s=x+\sqrt{2\sigma_x}B_s,
\end{gather*}
with $\mu_{s}|_{s=0}=\mu$, is a martingale. Of course, it follows that $(M_s^\mu-M_s^\nu)_{s\in[0,t]}$ is also a martingale. Let us now define the change of probability $\frac{d\mathbb{P}}{d\mathbb{Q}}$ associated to the martingale exponential $\mathcal{E}(Y)$ for the process 
\[Y^\mu_s= -\frac{1}{\sqrt{2\sigma_x}}\int_0^s D_p H(t-u,X_u,\noise_u,\mu_u)dB_u.\]
Following a proof similar to the one of Lemma \ref{Lemma: gradient U=W}, we have 
\begin{align}
\label{equation: U^mu-U^nu under Q}
U(t,x,\noise,\mu)-U(t,x,\noise,\nu)&=\mathbb{E}^\mathbb{Q}\left[U_0(X_t,\noise_t,\mu_t)-U_0(X_t,\noise_t,\nu_t)\right]\\
\nonumber &+\mathbb{E}^\mathbb{Q}\left[\int_0^t\left( H^\nu_s-H^\mu_s+D_pH^\mu_s\cdot(\nabla_x U^\mu_s-\nabla_x U^\nu_s)\right)ds\right]
\end{align}
with the notation $H^\mu_s=H(t-s,X_s,\noise_s,\nabla_x U^\mu_{s})$ and $\nabla_x U^\mu_s=\nabla_x U(t-s,X_s,\noise_s,\mu_s)$. Let us also remark that the dynamic of $(\mu_s,\nu_s,\noise_s)_{s\in[0,t]}$ does not change under $\mathbb{Q}$ and that 
\[X_s=x-\int_0^s D_p\tilde{H}(t-u,X_u,\noise_u,\mu_u)du+\sqrt{2\sigma_x}\tilde{B}_s\]
for $(\tilde{B}_s)_{s\geq 0}$ a Brownian motion under $\mathbb{Q}$. We now make the following observation: for any continuous bounded function $g$ and $s\in[0,t]$, one has
\begin{align*}
\int \mathbb{E}^\mathbb{Q}\left[{g(X^x_s)}\right]\mu(dx)=\mathbb{E}\left[\int g(x)\mu_s(dx)\right].
\end{align*}
This is a natural consequence of almost sure pathwise uniqueness for $(\mu_s)_{s\in [0,t]}$ \cite{convergence-problem}.
Integrating with respect to $\mu$ on both sides of \eqref{equation: U^mu-U^nu under Q} yields
\begin{align*}
\langle U(t,\cdot,\noise,\mu)-U(t,\cdot,\noise,\nu),\mu \rangle&=\mathbb{E}\left[\langle U_0(\cdot,\noise_t,\mu_t)-U_0(\cdot,\noise_t,\nu_t),\mu_t\rangle\right]\\
\nonumber &+\mathbb{E}\left[\int_0^t\int_{\mathbb{T}^d}\Theta (t-s,\cdot,\noise_s,\mu_s,\nu_s)\mu_s(dx)ds\right].
\end{align*}
for \[\Theta(t,x,\theta,\mu,\nu)= \tilde{H}(t,x,\theta,\mu)-\tilde{H}(t,x,\theta,\nu)+D_p\tilde{H}(t,x,\theta,\mu)\cdot(\nabla_x U(t,x,\theta,\mu)-\nabla_x U(t,x,\theta,\nu)).\]
A similar computation on $\nu$, using 
\[Y^\nu_s=-\frac{1}{\sqrt{2\sigma_x}}\int_0^s D_p\tilde{H}(t-u,X_u,\noise_u,\nu_u)dB_u\]
combined with our assumptions on H (namely the monotonicity of $f$ and the strong convexity of $\bar{H}$) and $U_0$ finally gives:
\[\langle U(t,\cdot,\noise,\mu)-U(t,\cdot,\noise,\nu),\mu-\nu \rangle \geq \alpha_H \mathbb{E}\left[\int_0^t\int_{\mathbb{T}^d}\left|\nabla_x U(t-s,x,\noise_s,\mu_s)\!-\!\nabla_x U(t-s,x,\noise_s,\nu_s)\right|^2(\mu_s\!+\!\nu_s)(dx)ds \right]\!.\]
Whenever the idiosyncratic noise is degenerate this result is proven by considering a converging sequence of Lipschitz solution with $\sigma_x=\varepsilon$, as we did in Lemma \ref{Lemma: gradient U=W}. 
\end{proof}
This inequality is the key element that allows us to get Lipschitz estimate in $\mathcal{W}_1$ by the very method presented in \cite{convergence-problem} Theorem 4.3. 
\begin{thm}
\label{existence b(p) W1}
    Assume Hypothesis \ref{hyp: Lipschitz W1}, \ref{Hyp: boundedness + linear growth} and \ref{hyp: separated Hamiltonian +LL monotonicity} hold, and further assume that $b$ is a function of $\noise$ only. Then there exists a unique Lipschitz solution to equation \eqref{eq: ME W2 with b periodic} on $[0,+\infty)$. 
\end{thm}
\begin{proof}
    Local existence follows from Hypothesis \ref{hyp: Lipschitz W1}. Lemma \ref{Lemma: heat kernel estimate for derivative in x} gives us an estimate on the Lipschitz norm of $W$ in $x$. Because the parameter $\noise$ is associated to an autonomous drift, an estimate on the Lipschitz norm of $W$ in $\noise$ is obtained by mean of Grönwall's Lemma, so long as $W$ is Lipschitz in the other variables. It only remains to show that the inequality from Lemma \ref{lemma: monotonicity estimate in wasserstein 1 b(p)} is sufficient to bound the Lipschitz norm of $W$ in $\mathcal{W}_1$.
    
    Let us first remark that thanks to the dual representation of the Wasserstein 1 distance:
    \[\langle U(t,\cdot,\noise,\mu)-U(t,\cdot,\noise,\nu),\mu-\nu\rangle\leq \|\nabla_x U(t,\cdot,\noise,\mu)-\nabla_x U(t,\cdot,\noise,\nu)\|_\infty \mathcal{W}_1(\mu,\nu).\]
Fix $t<T_c$ and consider the solution to 
\[
\left\{
\begin{array}{l}
     dX^\mu_s= F(X^\mu_s,\noise_s,\mu_s,\nabla_x U(t-s,X^\mu_s,\noise_s,\mu_s))+\sqrt{2\sigma_x}dB_s \quad X^\mu_s|_{s=0}=X_0\sim \mu,\\
     \mu_s=\mathcal{L}(X^\mu_s|(\noise_u)_{u\leq s}),\\
     \displaystyle d\noise_s=b(\noise_s)ds+\sqrt{2\sigma_\noise}dB^\noise_s \quad \noise_0=\noise.
\end{array}
\right.
\]
Similarly we define $(X_s^\nu)_{s\in[0,t]}$ with an initial condition 
\[X^\nu_s|_{s=0}=Y_0\sim \nu,\] 
$X_0,Y_0$ being chosen such that the optimal coupling between $\mu$ and $\nu$ is realized in $\mathcal{W}_1$. Letting $s\leq t$, we get
\begin{align*}\esp{\left|X^\mu_s-X^\nu_s\right|}&\leq \mathcal{W}_1(\mu,\nu)+C\int_0^s\esp{\left|X^\mu_u-X^\nu_u\right|}du\\
&+C\esp{\int_0^t\int_{\mathbb{T}^d}\left|\nabla_xU(t-u,x,\noise_u,\mu_u)-\nabla_x U(t-u,x,\noise_u,\nu_u)\right|\mu_u(dx)du},
\end{align*}
where the constant C depends on $H,t,\|\nabla_x U\|_{Lip(x)}$ only. By Lemma \ref{lemma: monotonicity estimate in wasserstein 1 b(p)} and an application of Grönwall's Lemma we conclude to 
\[\forall s\leq t \quad \esp{\mathcal{W}_1(\mu_s,\nu_s)}\leq C\left(\mathcal{W}_1(\mu,\nu)+\sqrt{\|\nabla_x U(t,\cdot,\noise,\mu)-\nabla_x U(t,\cdot,\noise,\nu)\|_\infty \mathcal{W}_1(\mu,\nu)}\right)e^{Cs}.\]
Thanks to this inequality, we may now estimate $\|\nabla_xU(t,\cdot,\mu)-\nabla_xU(t,\cdot,\nu)\|_\infty$. Using the representation formula \eqref{equation: representation formula heat kernel W} we get the following inequality:

\begin{align*}|\nabla_x U(t,x,\noise,\mu)-\nabla_xU (t,x,\noise,\nu)|\leq& C\esp{\mathcal{W}_1(\mu_t,\nu_t)+\int_0^t \mathcal{W}_1(\mu_s,\nu_s)}\\
&+C\int_0^t \frac{1}{\sqrt{s}}\esp{\|\nabla_xU(t-s,\cdot,\mu_s)-\nabla_xU(t-s,\cdot,\nu_s)\|_\infty}ds,
\end{align*}
for a constant C depending on $\|\nabla_x U_0\|_{Lip}, f,\tilde{H},U_0$. Applying Lemma \ref{Lemma: backward gronwall}   to \[s\to \esp{\|\nabla_x U(s,\cdot,\mu_{t-s})-\nabla_x U(s,\cdot,\nu_{t-s})\|_\infty}\]
yields 
\[\|\nabla_x U(t,\cdot,\mu)-\nabla_x U(t,\cdot,\nu)\|_\infty\leq C\left(\mathcal{W}_1(\mu,\nu)+\sqrt{\|\nabla_x U(t,\cdot,\mu)-\nabla_x U(t,\cdot,\nu)\|_\infty\mathcal{W}_1(\mu,\nu)}\right),\]
for a constant C depending both on $t$ and the data of the problem. This yields the required Lipschitz estimate on $\nabla_x U=W$ in $\mathcal{W}_1$.

Because, we have estimates on the Lipschitz semi-norm of $W$ in all variables for any time $t<\infty$ we may then conclude to the existence of the Lipschitz solution on any time interval. Indeed, if $T_c<\infty$ held, then necessarily
\[\underset{t\to T_c}{\lim}\|W\|_{Lip}=+\infty,\]
which is absurd in light of the above proof. 
\end{proof}
\subsubsection{Toward more general noise models}
\quad

We now turn to more general choices of a function $b$, namely allowing it to depend on the measure argument. For such models estimates presented in the above section fail in general \cite{noise-add-variable}. Let us explain the strategy of proof in this more convoluted case. The key point is that even when perturbed by this noise process, monotonicity should still have a regularizing effect on the equation. The goal is still to get estimates on the equation by looking at the quantity
\[(t,\mu,\nu)\mapsto\langle U(t,\cdot,\theta,\mu)-U(t,\cdot,\theta,\nu),\mu-\nu\rangle.\]
However since the evolution of $\noise$ now depends on the measure argument, we cannot treat it as an independent perturbation as in the previous section. In fact since the evolution of the noise process and the distribution of players now depend on each other, it appears necessary to rather consider 
\[(t,\mu,\nu,\noise,\noiseb)\to \langle U(t,\cdot,\theta,\mu)-U(t,\cdot,\noiseb,\nu),\mu-\nu\rangle.\]
Obviously we don't expect this quantity to stay positive for $\theta\neq \noiseb$. A naive way to solve this problem is to add a quadratic form in $\theta$ acting as a way to recover non-negativity (or as a penalization in some sense) by considering instead
\[(t,\mu,\nu,\noise,\noiseb)\to \frac{1}{2}(\noise-\noiseb)\cdot A(\noise-\noiseb)+ \langle U(t,\cdot,\theta,\mu)-U(t,\cdot,\noiseb,\nu),\mu-\nu\rangle,\]
for some matrix $A\geq 0$. We explain in a later section how this idea can be generalized to more general penalization. Here we focus on a simple case in which we assume that $\bar{H}$ does not depend on $\noise$. 
\begin{hyp}
\label{hyp: separated Hamiltonian +LL monotonicity b(p,m)}
    \[H(x,\noise,p,m)=\bar{H}(x,p)-f(x,\noise,m),\]
for a function $\bar{H}$ $\alpha_H-$convex in $p$. Furthermore, there exists a symmetric matrix $A\in M_m(\reels)$ such that 
\begin{gather}
\nonumber\forall (\noise,\noiseb)\in\reels^{2m},  (\mu,\nu)\in\left(\mathcal{P}(\mathbb{T}^d)\right)^2 \quad\\
\label{LL monotinicity + in p U0}
  \frac{1}{2}(\noise-\noiseb)\cdot A(\noise-\noiseb)+\langle U_0(\cdot,\noise,\mu)-U_0(\cdot,\noiseb,\nu),\mu-\nu\rangle \geq 0,\\
\label{joint monotonicity (f,b,H)}
\langle f(\cdot,\noise,\mu)-f(\cdot,\noiseb,\nu),\mu-\nu\rangle+(b(\noise,\mu)-b(\noiseb,\nu))\cdot A(\noise-\noiseb)\geq 0.
\end{gather}
\end{hyp}
\begin{remarque}
    If $U_0$ is strongly monotone, it is always possible to find a matrix $A=cI_n$ such that \eqref{LL monotinicity + in p U0} holds. The second part \eqref{joint monotonicity (f,b,H)} is a joint monotonicity assumptions on $(f,Ab)$ in $\mathcal{P}(\tor^d)\times \reels^n$. It is a much stronger assumption, which in most cases we expect is not going to be fulfilled unless either $f$ or $b$ is strongly monotone.
\end{remarque}
\begin{exemple}
    Consider this simple model in dimension 1: Letting $\theta$ represent the selling price of a product and $x$ the production of a firm. We fix
    \[b(\theta,\mu)=r\theta-\alpha \int y\mu(dy).\]
    The term $r\theta$ models inflation while the price of the good decrease in function of the average production of each company. For one single firm we set 
    \[f(x,\theta,\mu)=x\theta-c(x).\]
    Thus the reward of a firm depends only on its choice of production $x$ and the price of goods $\theta$, the cost of production for a given quantity $x$ being $c(x)$. Then for $A=\frac{1}{\alpha}$, \eqref{joint monotonicity (f,b,H)} is verified for the couple $(f,Ab)$.
\end{exemple}
We can now show an equivalent to Lemma \ref{lemma: monotonicity estimate in wasserstein 1 b(p)} in this new setting. 
\begin{lemma}
\label{lemma: monotonicity estimate in wasserstein 1 b(p,m)}
Under Hypotheses \ref{hyp: Lipschitz W1} and \ref{hyp: separated Hamiltonian +LL monotonicity b(p,m)}, if $U$ is a Lipschitz solution of \eqref{eq: ME W2 with b periodic} on $[0,T_c)$ as defined in \eqref{definition: U for Lipschitz solutions}, then 
\begin{gather}
\nonumber \forall t<T_c \quad (\noise,\noiseb,\mu,\nu)\in\left(\reels^n\right)^2\times \left(\mathcal{P}(\tor^{d})\right)^2,\\
\nonumber  \frac{1}{2}(\noise-\noiseb)\cdot A(\noise-\noiseb)+\langle U(t,\cdot,\noise,\mu)-U(t,\cdot,\noiseb,\nu),\mu-\nu \rangle \\
    \label{monotonicity estimate W1 gradient}
    \geq \alpha_H \mathbb{E}\left[\int_0^t\int_{\mathbb{T}^d}\left|\nabla_x U(t-s,x,\noise_s,\mu_s)-\nabla_x U(t-s,x,\noiseb_s,\nu_s)\right|^2(\mu_s+\nu_s)(dx)ds \right].
\end{gather}
with 
\[
\left\{\begin{array}{ll}
d\noise_s=b(\noise_s,\mu_s)ds+\sqrt{2\sigma_\noise}dB^\noise_s& \noise_{s=0}=\noise,\\
d\noiseb_s=b(\noiseb_s,\nu_s)ds+\sqrt{2\sigma_\noise}dB^\noise_s& \noise_{s=0}=\noiseb,\\
d\mu_s=\left(-\text{div}\left(F(x,\noise_s,\mu_s,W(t-s,x,\noise_s,\mu_s))\right)+\sigma_x \Delta_x \mu_s\right)ds & \mu_{s=0}=\mu,\\
d\nu_s=\left(-\text{div}\left(F(x,\noiseb_s,\nu_s,W(t-s,x,\noiseb_s,\nu_s))\right)+\sigma_x \Delta_x \nu_s\right)ds & \nu_{s=0}=\nu,\\
\end{array}
\right.
\]
\end{lemma}
\begin{proof}
As in Lemma \ref{lemma: monotonicity estimate in wasserstein 1 b(p)}, it is sufficient to carry out the proof for $\sigma_x>0$. Fix $t<T_c$ and consider 
    \[
    \left\{
    \begin{array}{cc}
         d\noise_s=b(\noise_s,\mu_s)ds+\sigma_\noise dB^\noise_s &\noise_0=\noise  \\
         d\mu_s=(-\text{div}(F(x,\noise_s,\mu_s,\nabla_x U(t-s,\cdot,\noise_s,\mu_s))+\sigma_x \Delta_x \mu_s)ds& \mu_0=\mu\\
    \end{array}
    \right.
    \]
    and let $(\noiseb_s,\nu_s)_{s\in[0,t]}$ be defined for the same Brownian motion $(B^\theta_s)_{s\geq 0}$ but with initial condition $(\noiseb,\nu)$. Observe the dynamic of those 4 processes stays the same under the change of probability associated to the martingale exponential of the stochastic process
    \[Y^{\theta,\mu}_s= -\frac{1}{\sqrt{2\sigma_x}}\int_0^s D_p \tilde{H}(t-u,X_u,\noise_u,\mu_u)dB_u,\]
    whenever $(B_s)_{s\geq 0}$ is a Brownian motion independent of $(B^\theta_s)_{s\geq 0}$. The change of probability is then carried out as in Lemma \ref{lemma: monotonicity estimate in wasserstein 1 b(p)}, the difference being the change of probability introduced now depends on the couples $(\noise,\mu), (\noiseb,\nu)$. Using this change of probability on the representation formula \eqref{definition: U for Lipschitz solutions} and remarking $(\noise,\noiseb)$ satisfy
    \[\frac{1}{2}(\noise-\noiseb)\cdot A(\noise-\noiseb)=\frac{1}{2}(\noise_t-\noiseb_t)\cdot A(\noise_t-\noiseb_t)+\int_0^t (b(\noise_s,\mu_s)-b(\noiseb_s,\nu_s))\cdot A(\noise_s-\noiseb_s)ds,\]
    yields using Hypothesis \ref{hyp: separated Hamiltonian +LL monotonicity b(p,m)}
    \begin{gather*}\nonumber \frac{1}{2}(\noise-\noiseb)\cdot A(\noise-\noiseb)+\langle U(t,\cdot,\noise,\mu)-U(t,\cdot,\noiseb,\nu),\mu-\nu \rangle \\
    \geq \alpha_H\mathbb{E}\left[\int_0^t\int_{\mathbb{T}^d}\left|\nabla_x U(t-s,x,\noise_s,\mu_s)-\nabla_x U(t-s,x,\noiseb_s,\nu_s)\right|^2(\mu_s+\nu_s)(dx)ds \right].
    \end{gather*}
\end{proof}
A first remark is that for $\noise=\noiseb$ we recover the same inequality we had in Lemma \ref{lemma: monotonicity estimate in wasserstein 1 b(p)}. However, this inequality is much stronger as it will also help us control the Lipschitz norm of $\nabla_x U$ with respect to $\noise$.

\begin{thm}
\label{thm: existence b(p,m) W1}
    Assume Hypotheses \ref{hyp: Lipschitz W1}, \ref{Hyp: boundedness + linear growth}, and \ref{hyp: separated Hamiltonian +LL monotonicity b(p,m)} hold. Then there exists a unique Lipschitz solution to equation \eqref{eq: ME W2 with b periodic} on $[0,+\infty)$ 
\end{thm}
\begin{proof}
    First, we prove a stability estimate with respect to initial conditions on the system
\[
\left\{
\begin{array}{l}
     \displaystyle X^{X_0,\noise}_s=X_0-\int_0^s D_p H(X^{X_0,\noise}_u,\noise^{X_0,\noise}_u,\mu^{X_0,\noise}_u,\nabla_x U(t-u,X^{X_0,\noise}_u,\noise^{X_0,\noise}_u,\mu^{X_0,\noise}_u))+\sqrt{2\sigma_x}dB_s,  \\
     \displaystyle \noise_s^{X_0,\noise}=\noise-\int_0^s b(\noise_u^{X_0,\noise},\mu^{X_0,\noise}_u)du+\sqrt{2\sigma_\noise}B^\noise_s,\\
     \mu^{X_0,\theta}_u=\mathcal{L}(X_u|(\noise_l)_{l\leq u}).
\end{array}
\right.
\]
 The sketch of proof is similar to Theorem \ref{existence b(p) W1} but we now estimate
\[h(s):s\mapsto\esp{|X_s^{X_0,\noise}-X_s^{Y_0,\noiseb}|+|\noise_s^{X_0,\noiseb}-\noise_s^{Y_0,\noiseb}|},\]
instead of fixing $\theta$. Let $X_0$ and $Y_0$ respectively be distributed along $\mu$ and $\nu$ and such that the optimal coupling in $\mathcal{W}_1$ is realized. 
Hypothesis \ref{hyp: Lipschitz W1} implies that
\[\forall s\leq t \quad h(s)\leq h(0)+C\int_0^s\left(h(u)+\esp{\int_{\tor^d}|\nabla_x U(t-u,x,\noise^{X_0,\noise}_s)-\nabla_x U(t-u,x,\noise^{Y_0,\noiseb}_s)|\mu_u^{X_0,\theta}(dx)}\right),\]
for a constant $C$ depending on $H,b$ and $\|\nabla_x U\|_{Lip(x)}$ which we know is bounded thanks to Lemma \ref{Lemma: heat kernel estimate for derivative in x}. We now apply Grönwall's Lemma to $s\mapsto h(s)$ and then use Lemma \ref{lemma: monotonicity estimate in wasserstein 1 b(p,m)} to deduce that
\begin{gather*}
    \forall s\leq t \quad \esp{\mathcal{W}_1(\mu^{X_0,\theta}_s,\nu^{Y_0,\noiseb}_s)+|\noise^{X_0,\theta}_s-\noise^{Y_0,\noiseb}|}\\
    \leq C(\mathcal{W}_1(\mu,\nu)+|\noise-\noiseb|+\sqrt{\|\nabla_x U(t,\cdot,\noise,\mu)-\nabla_x U(t,\cdot,\noiseb,\nu)\|_{\infty}\mathcal{W}_1(\mu,\nu)},
\end{gather*}
for a constant C depending on $H,b,\alpha_H$ and $\|\nabla_x U\|_{Lip(x)}$. We finally come back to the representation formula \eqref{equation: representation formula heat kernel W}, which yields by combining the above inequality with Lemma \ref{Lemma: backward gronwall}
\[\exists C>0 \quad \|\nabla_x U(t,\cdot,\noise,\mu)-\nabla_x U(t,\cdot,\noiseb,\nu)\|_{\infty}\leq C(\mathcal{W}_1(\mu,\nu)+|\noise-\noiseb|).\]
Having estimates on the Lipschitz norm of $\nabla_xU(t,\cdot)$ in all variables for any $t\geq 0$, we conclude to the existence of a Lipschitz solution on $[0,+\infty)$.
\end{proof}
\begin{remarque}
Some extensions to this result are possible. In particular, it is possible for $b$ to depend on $\nabla_x U$ provided this dependency is such that
\begin{gather}
\label{condition on b(U) W1}
|b(\noise,\mu,\!\nabla_x U(t,\cdot,\theta,\mu))-b(\noiseb,\nu,\!\nabla_x U(t,\cdot,\noiseb,\nu))|\\
\nonumber \leq C(|\noise-\noiseb|+\mathcal{W}_1(\mu,\nu)+\langle |\nabla_x U(t,\cdot,\theta,\mu)-\nabla_x U(t,\cdot,\noiseb,\nu)|,\mu+\nu\rangle).
\end{gather}
So that the nonlinear term in $\nabla_x U$ coming from $b$ can be controlled by the convexity of $H$. It is also possible to take a function $\bar{H}$ depending on $\noise$ provided this dependency can be absorbed by the monotonicity of $b$. Those extensions are straightforward by adapting slightly the above proof and changing appropriately equation \ref{joint monotonicity (f,b,H)} in Hypothesis \ref{hyp: separated Hamiltonian +LL monotonicity b(p,m)} which then becomes an interplay between $\bar{H},f$ and $b$.
\end{remarque}
\section{Wellposedness under \texorpdfstring{$L^2-$}{TEXT}monotonicity}
\label{section L2 monotone solutions}
In this section, we are interested in the existence of solutions to the master equation \eqref{eq: general ME} in $\reels^d\times\reels^n\times \mathcal{P}_2(\reels^d)$. In the previous section, we worked directly at the level of the value function $U$ to show the wellposedness of solutions. Following Section \ref{section Lipschitz solutions}, it is sufficient to show the Lipschitz solution $W$ of \eqref{eq: linear lip sol} do not blow up on the interval $[0,T]$. The approach we take in this section is to present wellposedness results at the level of $W$ which is more natural in an $L^2$ monotone setting. Indeed, we consider instead this non-linear system of transport equations
\begin{equation}
\label{eq: ME W2 Lip}
\left\{
\begin{array}{c}
\displaystyle\partial_t W+F(x,\noise,m,W)\cdot \nabla_x W+b[W](t,\noise,m)\cdot\nabla_\noise W-\sigma_x \Delta_x W-\sigma_\noise \Delta_\noise W\\
\displaystyle +\int_{\reels^d} F(y,\noise,m,W)\cdot D_m W(t,x,\noise,m)(y)m(dy)-\sigma_x\int_{\reels^d} \text{div}_y (D_mW(t,x,\noise,m)(y))m(dy)\\
\displaystyle =G(x,\noise,m,W) \text{ in } (0,T)\times\reels^d\times\reels^n\times\mathcal{P}_2(\reels^d),\\
\displaystyle W(0,x,\noise,m)=W_0(x,\noise,m) \text{ for } (x,\noise,m)\in\reels^d\times\reels^n\times\mathcal{P}_2(\reels^d).
 \end{array}
 \right.
\end{equation}
We insist on the fact that $W$ is a vector function taking values in $\reels^d$ and \eqref{eq: ME W2 Lip} is indeed a system of non linear PDEs.  Recall that for $(F,G,W_0)=(\nabla_pH,-\nabla_xH,\nabla_x U_0)$, \eqref{eq: ME W2 Lip} is obtained from \eqref{eq: general ME} by taking the gradient in $x$ of the equation, in which case $W=\nabla_x U$ for U the solution of the master equation. Let us first remark that the adaptation of the notion of Lipschitz solution to \eqref{eq: ME W2 Lip} is not a problem, a solution being defined as a fixed point of the associated linear transport problem.

We consider general data $(F,G,W_0)$ instead of taking them as gradient (which correspond to the situation of a MFG) for the following reason: Equation \eqref{eq: ME W2 Lip} also corresponds to the equation satisfied by the decoupling field of a general mean field forward backward stochastic differential equation (FBSDE) \cite{mean-field_G-monotonicity,mean-field-FBSDE}. As a consequence, the results we give in this section are also valid for mean field forward-backward stochastic differential equations and some extended MFGs \cite{extended-mfg}. As in the previous section, we start with the case of an autonomous noise process.

\subsection{Autonomous noise process}
\label{subsection: autonomous W2}
\subsubsection{Propagation of monotonicity}
We consider the following equation 
\begin{equation}
\label{eq: ME W2 Lip b(p)}
\left\{
\begin{array}{c}
\displaystyle\partial_t W+F(x,\noise,m,W)\cdot \nabla_x W+b(\noise)\cdot\nabla_\noise W-\sigma_x \Delta_x W-\sigma_\noise \Delta_\noise W\\
\displaystyle +\int_{\reels^d} F(y,\noise,m,W)\cdot D_m W(t,x,\noise,m)(y)m(dy)-\sigma_x\int_{\reels^d} \text{div}_y (D_mW(t,x,\noise,m)(y))m(dy)\\
\displaystyle =G(x,\noise,m,W) \text{ in } (0,T)\times\reels^d\times\reels^n\times\mathcal{P}_2(\reels^d),\\
\displaystyle W(0,x,\noise,m)=W_0(x,\noise,m) \text{ for } (x,\noise,m)\in\reels^d\times\reels^n\times\mathcal{P}_2(\reels^d).
 \end{array}
 \right.
\end{equation}
Even though the problem becomes simpler whenever the noise process $(\theta_t)_{t\geq 0}$ is autonomous, \eqref{eq: ME W2 Lip b(p)} is still a system of nonlinear transport equation on the space of measure. It is classic that such a system is usually well-posed only up until a finite blow-up time without some assumptions on the coefficients. The following assumption
\begin{hyp}
\label{hyp: lip W2 b(p)}
   The drift $b$ is Lipschitz and $F,G,W_0$ are Lipschitz on $\reels^d\times \reels^n\times \mathcal{P}_2(\reels^d)$ for the $\mathcal{W}_2$ distance. 
\end{hyp}
is in force throughout this section. Short time existence of Lipschitz solutions to \eqref{eq: ME W2 Lip b(p)} follows naturally from Theorem \ref{lip sol b(p)}.
In his course at Collège de France \cite{Lions-college}, P.-L. Lions, showed through the Hilbertian approach that in the absence of noise (i.e. without $\theta$ and for $\sigma_x=0$), $L^2-$monotonicity of the data has a regularizing effect on the equation \eqref{eq: ME W2 Lip b(p)}. In this section we are going to show that such considerations still hold in this setting. Following the works \cite{noise-add-variable,Lions-college} we expect that a good starting point to obtain some regularity on the solution of \eqref{eq: ME W2 Lip b(p)} is to take a look at the auxiliary function
\begin{equation}
\label{def: Z b(p)}
Z(t,\gamma,\noise)=\int_{\reels^{2d}}( W(t,x,\noise,\mu)-W(t,y,\noise,\nu))\cdot(x-y) \gamma(dx,dy)\\
 \end{equation}
 for any $\gamma\in \Gamma(\mu,\nu),(\mu,\nu)\in\left(\mathcal{P}(\reels^d)\right)^2$. One may remark that if we lift $W$ on the space $L^2(\Omega,\reels^d)$, then
 \[Z(t,\mathcal{L}\left(X,Y\right),\noise)=\langle W(t,X,\noise)-W(t,Y,\noise),X-Y\rangle_{L^2},\]
 for any $(X,Y)\in \left(L^2(\Omega,\reels^d)\right)^2$. As such, the non-negativity of $Z$ is equivalent to the $L^2-$monotonicity of $W$. The key difference with the Hilbertian approach is that we place ourselves directly at the level of probability measures, which allows us to work with a more subtle notion of derivative. Let us first prove a uniqueness result for equation \eqref{eq: ME W2 Lip b(p)} which we believe will help the reader understanding why looking at this function in indeed a good idea. We start with a maximum principle for smooth functions on the space of probability measures. Since we are on an unbounded space, we will need to impose growth conditions at infinity.
 \begin{definition}
      A function $f:\mathcal{P}_2(\reels^d)\to \reels^d$ has polynomial growth if there exists a constant $C>0$ and an integer $q\geq 0$ such that 
      \[\forall \mu\in \mathcal{P}_2(\reels^d) \quad |f(\mu)|\leq C\left(1+\left(\int_{\reels^d} |x|^2\mu(dx)\right)^\frac{q}{2}\right).\]
      In particular for $q=1$ (resp. $q=2$), $f$ is said to have linear (resp. quadratic) growth at infinity. 
 \end{definition}
 \begin{lemma}
 \label{smooth comparison with 0}
 Let $f:[0,T]\times\reels^n\times \mathcal{P}_2(\reels^{2d})\to \reels^{2d}, b:[0,T]\times\reels^n\times \mathcal{P}_2(\reels^{2d})\to \reels^{n}$ be continuous functions with linear growth. Suppose there exist $Z$, a smooth function with quadratic growth at infinity, satisfying
\begin{equation*}
\begin{array}{c}
\partial_t Z+b(\noise,m)\cdot \nabla_\noise Z-\sigma_\noise\Delta_\noise Z\\
+\displaystyle\int_{\reels^{2d}} f(u,\noise,m)
\cdot D_\gamma Z(t,m,\noise,u) m(du)-\sigma_x\int_{\reels^{2d}}\text{Tr}\left(B_dD_{u} \left[D_mZ\right](t,m,\noise,u)\right)m(du)\geq 0,\\
\text{ for } (t,m,\noise)\in (0,T)\times\mathcal{P}_2(\reels^{2d})\times \reels^n.
\end{array}
\end{equation*}
with $Z|_{t=0}\geq 0$. Then 
\[\forall t<T, \forall (\noise,m)\in \reels^n\times \mathcal{P}_2(\reels^{2d}) \quad Z(t,\theta,m)\geq 0.\]
 \end{lemma}
 \begin{proof}
For $m\in\mathcal{P}_3(\reels^d)$, we define 
 \[E_3(m)=\left(\int_{\reels^{2d}} |x|^3m(dx)\right)^\frac{1}{3}. \]
 
Let $\beta:s\mapsto \alpha e^{\lambda_\beta s}$ with $\alpha>0$, and consider $t\leq T$. We claim that 
\[Z_\alpha:(s,m,\theta)\to Z(s,m,\theta)+\beta(s)\left(1+E^3_3(m)+|\theta|^3\right)+\frac{\alpha}{t-s}\]
reach a point of minimum in $[0,t)\times\mathcal{P}_2(\reels^{2d})\times\reels^n$. Indeed, since $Z$ has at most quadratic growth, There exists a constant C such that the infimum of $Z_\alpha$ is reached on the set  
\[\{(s,\noise,m)\in [0,t-\frac{1}{C}]\times \mathcal{P}_2(\reels^{2d})\times \reels^n, \quad E_3(m),|\noise|\leq C\}.\]
By Markov inequality a subset of $\mathcal{P}_2(\reels^d)$ bounded for the pseudo-norm $E_3(\cdot)$ is tight and is, as a consequence, compact for the topology of narrow convergence by Prokhorov's Theorem. Since probability measures belonging to this set also have a uniformly integrable second moment (due to the bound on the third moment) this is a compact set for the $\mathcal{W}_2$ topology (see \cite{Ambrosio2013}). As $Z$ is continuous for this topology and $m\mapsto \int |z|^3m(dz)$ is lower semi-continuous with respect to narrow convergence, $Z_\alpha$ is a lower semi-continuous function for the $\mathcal{W}_2$ topology reaching its infimum in a compact. As a consequence, a minimum $(s^*,m^*,\theta^*)$ is achieved in this set. 

Let us now fix $\alpha>0$ and assume by contradiction that there exists a point $(s,m,\noise)$ such that $Z_\alpha(s,m,\noise)<0$ holds. By assumptions on $Z|_{t=0}$ this means the minimum is not reached for $s^*=0$. Optimality conditions imply by evaluating the equation satisfied by $Z_\alpha$ at this point of minimum that
\begin{align*}0&\geq \frac{\alpha}{(t-s^*)^2}+\beta'(s^*)\left(1+E^3_3(m^*)+|\theta^*|^3\right)\\
&+b(\noise^*,m^*)\cdot 3\beta(s^*)\noise^*|\noise^*|-\sigma_p\beta(s^*) (n+1)|\noise^*|\\
&+3\beta(s^*) \int f(u,\theta^*,m^*)\cdot u|u|m^*(du)-\sigma_x \int 3 \beta(s^*)\left((2d+1)|u|+2\sum_{i\leq d} \frac{u_i u_{n+i}}{|u|}\right)m^*(du).
\end{align*}
By the linear growth of $f,b$, there exists a constant $c$ depending on $n,d,b,f,\sigma_x,\sigma_p$ such that 
\begin{align*}
0&\geq \frac{\alpha}{(t-s^*)^2}+(\beta'(s^*)-c\beta(s^*))(1+|\noise^*|^3+E_3^3(m^*)).
\end{align*}
Choosing $\lambda_\beta>c$, we get a contradiction, hence 
\[\underset{[0,t]\times \mathcal{P}_2(\reels^{2d})\times \reels^n}{\inf} Z_\alpha \geq 0.\]
Letting $\alpha$ tends to 0 ends the proof since the inequality holds for any $t<T$.
 \end{proof}
 
 We may now proceed to prove a uniqueness result based on this maximum principle under the following assumption 
 \begin{hyp}
 \label{hyp: monotonie not strong b(p)}
     $\forall (\noise,X,Y,U,V)\in\reels^n\times \left(L^2(\Omega,\reels^d)\right)^4$, 
\begin{equation*}
    \begin{array}{c}
         \esp{\left(W_0(X,\noise,\mathcal{L}(X)-W_0(Y,\noise,\mathcal{L}(Y))\right)\cdot(X-Y)}\geq 0,  \\
    \esp{F(X,\noise,\mathcal{L}(X),U)-F(Y,\noise,\mathcal{L}(Y),V)\cdot(U-V)+(G(X,\noise,\mathcal{L}(X),U)-G(Y,\noise,\mathcal{L}(Y),V))\cdot(X-Y)}\geq 0,
    \end{array}
\end{equation*}
 \end{hyp}
 The assumption on $W_0$ is exactly $L^2-$monotonicity, while the condition imposed on $(G,F)$ is a joint $L^2-$monotonicity in $(x,W)$.

 \begin{prop}
 \label{prop: uniqueness W2 b(p)}
Under Hypotheses \ref{hyp: lip W2 b(p)} and \ref{hyp: monotonie not strong b(p)}, a smooth solution to \eqref{eq: ME W2 Lip b(p)} is $L^2$-monotone
 \end{prop}
\begin{proof}
Let us define 
    \[\forall \gamma\in\mathcal{P}_2(\reels^{2d})\quad Z(t, \gamma,\noise)=\int_{\reels^{2d}}( W(t,u,\noise,\pi_d  \gamma)-W(t,v,\noise,\pi_{-d} \gamma))\cdot(x-y)\rangle  \gamma(du,dv).\]
Let us first take a look at the derivative of $y\mapsto D_\gamma Z(t, \gamma,\theta,y)$, $D_yD_\gamma Z$, for a fixed $(t, \gamma,\theta)$. Letting $y=(y_1,y_2)\in \left(\reels^d\right)^2$, taken at this point, it is a block matrix of the form 
\begin{align*}
    \begin{array}{c}
         D_yD_\gamma Z(t,\gamma,p,y)=\left(\begin{array}{cc}
           M_{11}  & M_{12} \\
            (M_{12})^T  & M_{22}
         \end{array}\right),  \\
         \displaystyle M_{11}=D_x^2W^{y_1}\cdot(y_1-y_2)+(D_x W^{y_1}+\left(D_x W^{y_1}\right)^T+\int D_{y_1}D_m W^1(t,u,\mu,\theta,y_1)\cdot(u-v)\gamma(du,dv),\\
         M_{12}=-D_x W^{y_2}-\left(D_x W^{y_1}\right)^T,\\
         \displaystyle M_{22}=D_x^2W^{y_2}\cdot(y_2-y_1)+(D_x W^{y_2}+\left( D_x W^{y_2}\right)^T-\int D_{y_2}D_m W^2(t,u,\nu,\theta,y_2)\cdot(u-v)\gamma(du,dv),\\
         W^{y_1}=W(t,y_1,\mu,\noise),\\
         W^{y_2}=W(t,y_2,\nu,\noise),\\
         \pi_d  \gamma=\mu \quad \pi_{-d} \gamma=\nu.\\
    \end{array}
\end{align*}
Since the terms composing $M_{12}$ and its transpose are the opposite of terms appearing in $M_{11},M_{22}$, they obviously cancel out in 
\begin{align*}
\text{Tr}\left(B_d D_yD_mZ(t,\gamma,p,y)\right)&=(\Delta_xW^{y_1}-\Delta_x W^{y_2})\cdot(y_1-y_2)\\
&+\int (\text{div}_{y_1}(D_m W^1)(t,u,\mu,\theta,y_1)-\text{div}_{y_2}(D_mW^2)(t,v,\nu,m,\theta,y_2))\cdot(u-v)\gamma(du,dv).
\end{align*}
Expressing $\partial_t Z$ using the fact that $W$ is a solutions of \eqref{eq: ME W2 Lip b(p)} and using the above equality, we get that $Z$ is a solution of
\begin{equation}
\label{Z b(p) supersol}
\left\{
\begin{array}{c}
\partial_t Z+b(\noise)\cdot \nabla_\noise Z-\sigma_\noise\Delta_\noise Z\\
+\displaystyle\int_{\reels^{2d}}\left(\begin{array}{c}
     F^{u}  \\
     F^{v}
\end{array}
\right)(u,v,\noise,\gamma)
\cdot D_\gamma Z(t,\gamma,\noise,u,v) \gamma(du,dv)\\
-\displaystyle\sigma_x\int_{\reels^{2d}}\text{Tr}\left(B_d D_{(u,v)} \left[D_mZ\right](t,\gamma,\noise,u,v)\right)\gamma(du,dv)\\
=\displaystyle\int_{\reels^{2d}} (F^{u}-F^{v})\cdot(W^{u} -W^{v})\gamma(du,dv)+\int_{\reels^{2d}}(G^{u}-G^{v})\cdot(u-v)\gamma(du,dv),\\
\text{ for } (t,\gamma,\noise)\in (0,T)\times\mathcal{P}_2(\reels^{2d})\times \reels^n,
\end{array}
\right.
\end{equation}
with $W^{u}=W(t,u,\noise,\pi_d \gamma), W^{v}=W(t,v,\noise,\pi_{-d}\gamma), F^{u}=F(u,\noise,\pi_d \gamma,W^{u})$ and so on. 
By assumptions on $W_0,F,G$ this leads to $Z_{t=0}\geq 0$ and 
\begin{equation*}
\begin{array}{c}
\partial_t Z+b(\noise)\cdot \nabla_\noise Z-\sigma_\noise\Delta_\noise Z\\
+\displaystyle\int_{\reels^{2d}}\left(\begin{array}{c}
     F^{u}  \\
     F^{v}
\end{array}
\right)
\cdot D_\gamma Z(t,\gamma,\noise,u,v) \gamma(du,dv)-\sigma_x\int_{\reels^{2d}}\text{Tr}\left(B_dD_{(u,v)} \left[D_\gamma Z\right](t,\gamma,\noise,u,v)\right)\gamma(du,dv)\geq 0,\\
\text{ for } (t,\gamma,\noise)\in (0,T)\times\mathcal{P}_2(\reels^{2d})\times \reels^n.
\end{array}
\end{equation*}
By Lemma \ref{smooth comparison with 0}, $Z$ is non-negative, which implies that $W$ is indeed $L^2-$monotone.
\end{proof}
\begin{remarque}
      It is evident that this method of proof would still work if $(G,F)$ were to depend on $\mathcal{L}(W(t,X,\noise,\mathcal{L}(X)))$ so long as the joint monotonicity condition is satisfied. This is also true for all results of wellposedness we present in Section 4. Since this corresponds to some MFGs of control, it seems to be a possible interesting extension. 
\end{remarque}
\subsubsection{Existence of global Lipschitz solutions}
The study of the auxiliary function $Z$ gave us both a uniqueness and a monotonicity result. Essentially, what we showed is that given a $L^2-$monotone initial condition and some assumptions on the data $(F,G)$ the solution propagates monotonicity. However, monotonicity, usually only gives a one-sided bound on the derivative. Since we are interested in getting Lipschitz estimates to show the long time existence of Lipschitz solutions, we make stronger assumptions to get a bound from below. 
\begin{hyp}
\label{hyp: (G,F,w0) b(p) W2}
     $\exists \alpha >0 \quad \forall (\noise,X,Y,U,V)\in\reels^n\times \left(L^2(\Omega,\reels^d)\right)^4$, 
\begin{equation}
\label{Hyp: W0 b(p) W2}
         \esp{\left(W_0(X,\noise,\mathcal{L}(X)-W_0(Y,\noise,\mathcal{L}(Y))\right)\cdot(X-Y)}\geq \alpha \|W_0(X,\noise,\mathcal{L}(X))-W_0(Y,\noise,\mathcal{L}(Y))\|^2_{L^2},
\end{equation}
\begin{gather}
\label{hyp: (F,G) b(p) W2}
    \nonumber \esp{(F(X,\noise,\mathcal{L}(X),U)-F(Y,\noise,\mathcal{L}(Y),V))\cdot(U-V)+(G(X,\noise,\mathcal{L}(X),U)-G(Y,\noise,\mathcal{L}(Y),V))\cdot(X-Y)}\\
    \geq \alpha \underset{\lambda \in[0,1]}{\min}\|G(X,\noise,\mathcal{L}(X),\lambda U+(1-\lambda)V)-G(Y,\noise,\mathcal{L}(Y),\lambda U+(1-\lambda )V)\|^2_{L^2}.
\end{gather}
\end{hyp}
\begin{remarque}
The last inequality may appear unusual, this is a way of translating the fact that the function $(X,U)\mapsto (F(X,\mathcal{L}(X),U),G(X,\mathcal{L}(X),U))$ is monotone, with a monotonicity in $X$ that is not too degenerate. For Lipschitz data $(F,G)$, it is weaker than strong monotonicity for which the inequality would be 
\[\geq \alpha \|X-Y\|^2_{L^2},\]
as it allows for $G$ to be degenerate monotone and bounded. We already made such remark in the case of finite state space mean field game (see \cite{noise-add-variable} Corollary 3.13) and extend it here to this setting. 
\begin{exemple}
A simple example, satisfying the requirements of Hypothesis \ref{hyp: (F,G) b(p) W2} is to take $F(x,\theta,\mu,u)\equiv \tilde{F}(u)$ for any monotone function $\tilde{F}$ and 
\[G(x)=x^+\wedge 1.\]
Such $G$ is not strongly monotone but satisfy
\[\forall (x,y)\in \reels^d \quad  (G(x)-G(y))\cdot(x-y)\geq |G(x)-G(y)|^2.\] 
\end{exemple}
\end{remarque}
Under Hypothesis \ref{hyp: (F,G) b(p) W2}, we can show that the relation $\eqref{Hyp: W0 b(p) W2}$ propagates for any time. This will be done by mean of another, closely related, auxiliary function
\begin{gather}
\label{def: Z beta b(p)}
Z_\beta(t,\gamma,\noise)=\\
\nonumber \int_{\reels^{2d}} \bigg((W(t,u,\noise,\pi_d \gamma)-W(t,v,\noise,\pi_{-d}\gamma))\cdot(u-v)-\beta(t)|W(t,u,\noise,\pi_d \gamma)-W(t,v,\noise,\pi_{-d} \gamma)|^2 \bigg)\gamma(du,dv),
 \end{gather}
 for some smooth function of time $\beta$ to be determined later on. While the key idea is still to apply the maximum principle, we now work with Lipschitz solutions. As a consequence, it is necessary to adapt the arguments of Proposition \ref{prop: uniqueness W2 b(p)} to non-smooth functions. 

 \quad
 
 To that end, we use tools from the theory of viscosity solutions to show that $Z_\beta$ satisfies a comparison principle with $0$. The developments thereafter are very much in the spirit of \cite{noise-add-variable}, though some arguments have to be adapted to account for the fact that we are now considering an equation on the space of measure. First, we have to be clear by what we mean by viscosity solutions. In utmost generality, the adaptation of the theory of viscosity solutions to equations on the space of measure is a notoriously difficult problem. Since we are here interested in comparison argument of \eqref{def: Z beta b(p)} with a smooth function (namely the constant function equal to 0, to show non negativity), comparison arguments are much simpler. The extension of viscosity solutions to our problem raises few difficulties. 
Let $(k,l)\in\mathbb{N}^2$,$\Gamma$ be a symmetric, positive semi-definite $k\times k$ matrix, $\mathcal{H}$ be a continuous mapping from $\reels^+\times\reels^k\times \mathcal{P}_2(\reels^l)\times\reels\times\reels^k\to\reels$ and $f$ a continuous bounded function from $\reels^+\times\reels^k\times\mathcal{P}_2(\reels^l)\times \reels^l\to \reels^l$. Consider the following nonlinear partial differential equation
\begin{gather}
\nonumber \partial_t \mathcal{U}(t,\noise,\mu)+\mathcal{H}(\noise,\mu,\mathcal{U}(t,\noise,\mu),\nabla_\noise \mathcal{U}(t,\noise,\mu))-\text{Tr}\left(\Gamma D^2_\noise \mathcal{U}(t,\noise,\mu)\right)\\
\label{visc}
\tag{E} +\int_{\reels^l} f(t,\noise,\mu,y)\cdot D_m \mathcal{U}(t,\noise,\mu,y)\mu(dy)-\int_{\reels^l}\text{Tr}\left(B_l D_y D_m \mathcal{U}(t,\noise,\mu,y) \right)\mu(dy)=0,\\
\nonumber \text{ for } (t,\noise,\mu) \text{ in } (0,T)\times\reels^k\times\mathcal{P}_2(\reels^l).
\end{gather}
We introduce the space of test function $H_{test}$: a smooth function $\varphi:[0,T]\times \reels^k\times\mathcal{P}_2(\reels^l)$ is said to belong to $H_{test}$ if 
\begin{enumerate}
    \item[-] $\exists C>0, \quad \forall \mu \in\mathcal{P}_2(\reels^l) \quad \|\varphi(\cdot,\mu)\|_{C^{1,2}}\leq C$ 
    \item[-] $\forall (t,\noise,\mu), \quad y\mapsto D_m\varphi(t,\noise,\mu,y),y\mapsto D_yD_m\varphi(t,\noise,\mu,y)$ are continuous and 
    \[|D_m\varphi(t,\noise,\mu,y)|+|D_yD_m\varphi(t,\noise,\mu,y)|\leq C(1+|y|^{q-1}),\]
    with $q=\sup\{p, \mu\in\mathcal{P}_p(\reels^l)\}$,
    uniformly in $(t,\mu,\noise)$
\end{enumerate}
\begin{definition}
Let $u:(0,T)\times\reels^k\times\mathcal{P}_2(\reels^l)\to \reels$ be a continuous function. We say that u is a viscosity supersolution of \eqref{visc} if for any function $\varphi\in H_{test}$ such that a point of minimum $(t^*,\noise^*,\mu^*)$ of $u-\varphi$ is achieved in $(0,T)\times\reels^k\times\mathcal{P}_2(\reels^l)$ the following holds
\begin{gather*}
\partial_t \varphi(t^*,\noise^*,\mu^*)+\mathcal{H}(t^*,\noise^*,\mu^*,u(t^*,\noise^*,\mu^*),\nabla_q\varphi(t^*,\noise^*,\mu^*))-\text{Tr}\left(\Gamma D^2_q\varphi(t^*,\noise^*,\mu^*)\right)\\
 +\int_{\reels^l} f(t^*,\noise^*,\mu^*,x)\cdot D_m \varphi (t^*,\noise^*,\mu^*,x)\mu^*(dx)-\int_{\reels^l}\text{Tr}\left(B_l D_x D_m \varphi(t^*,\noise^*,\mu^*,x) \right)\mu^*(dx)\geq 0.
\end{gather*}
\end{definition}
One can readily check that as soon as $\Gamma$ is indeed symmetric positive semi-definite(in which case degenerate ellipticity holds) a classical supersolution of $(E)$ is indeed a viscosity supersolution thanks to Proposition \ref{prop: comparison minimum measure smooth}, thereby justifying this definition. We may now state the following lemma
\begin{lemma}
\label{lemma: Z supersol b(p) W2}
Consider $W$ a Lipschitz solution to \eqref{eq: ME W2 Lip b(p)} defined on $[0,T)$ for some $T>0$, and let $Z_\beta$ be defined as in \eqref{def: Z beta b(p)} for a $C^1$ positive function of time $\beta$. 
Then, $Z_\beta$ is a viscosity supersolution of 
\begin{gather}
\label{eq: Z supersol b(p) W2}
\nonumber \partial_t Z_\beta(t,\noise,\gamma)+b(\noise)\cdot \nabla_\noise Z_\beta-\sigma_\noise\Delta_\noise Z_\beta\\
\nonumber+\int_{\reels^{2d}}\left(\begin{array}{c}
     F^{u}  \\
     F^{v}
\end{array}
\right)
\cdot D_\gamma Z_\beta(t,\noise,\gamma,u,v) \gamma(du,dv)-\sigma_x\int_{\reels^{2d}}\text{Tr}\left(BD_{(u,v)} D_\gamma Z_\beta(t,\noise,\gamma ,u,v)\right)\gamma(du,dv)\\
=\int_{\reels^{2d}} (F^{u}-F^{v})\cdot(W^{u}-W^{v})\gamma (du,dv)+\int_{\reels^{2d}}(G^{u}-G^{v})\cdot(u-v)\gamma (du,dv)\\
\nonumber-2\beta(t) \int_{\reels^{2d}} ( G^{u}-G^{v})\cdot(W^{u}-W^{v}) \gamma (du,dv)-\frac{d\beta}{dt}\int_{\reels^{2d}} |W^{u}-W^{v}|^2\gamma (du,dv),\\
\nonumber\text{ for }(t,\noise,\gamma)\in (0,T)\times\reels^n\times\mathcal{P}_2(\reels^{2d}).
\end{gather}
with $W^{u}=W(t,u,\noise,\pi_d \gamma ), W^{v}=W(t,v,\noise,\pi_{-d} \gamma ), F^{u}=F(u,\noise,\pi_d \gamma ,W^{u})$ and so on.
\end{lemma}
\begin{proof}
\noindent\textit{Step 1: an inequality satisfied by Z}\\
Fix $(t,x_0,y_0,\noise_0,\gamma_0)\in[0,T)\times(\reels^d)^2\times(\reels^n)^2\times\mathcal{P}_2(\reels^{2d})$,
we are first going to take a look at 
\[V(t,x_0,y_0,\noise_0,m_0)=\langle W(t,x_0,\noise_0,\pi_dm_0)-W(t,y_0,\noise_0,\pi_{-d}m_0),x_0-y_0\rangle.\]
Our goal here is to show $Z_\beta$ satisfy a dynamic programming principle in the form of an inequality. For that we are going to go all the way from $V$ to $Z_\beta$ while using the dynamic programming principle satisfied by $W$. Which is why we introduce the following SDE-SPDE system
\[
\left\{
\begin{array}{c}
d\noise_s=-b(\noise_s)ds+\sqrt{2\sigma_\noise}dB^\noise_s \quad \noise_{t=0}=\noise_0,\\
dX_s=-F(X_s,\noise_s,\mu_s,W(t-s,X_s,\noise_s,\mu_s))ds+\sqrt{2\sigma_x} dB_s \quad X_{t=0}=x_0,\\
dY_s=-F(Y_s,\noise_s,\nu_s,W(t-s,Y_s,\noise_s,\nu_s))ds+\sqrt{2\sigma_x} dB_s \quad Y_{t=0}=y_0,\\
     d\gamma_s=(-\text{div}\left(\left(\begin{array}{c}F(x,\noise_s,\mu_s,W(t-s,x,\noise_s,\mu_s))\\ F(y,\noise_s,\nu_s,W(t-s,y,\noise_s,\nu_s)) \end{array}\right)\gamma_s\right)+\sigma_x \text{Tr}\left(BD^2_{(x,y)}\gamma_s\right))ds \quad \gamma_{t=0}=\gamma_0,\\
\end{array}
\right.
\]
for $(B^\noise_s,B_s)_{s\geq 0}$ a $d+n$ dimensional Brownian motion, with the notation $\mu_s=\pi_d \gamma_s$, $\nu_s=\pi_{-d} \gamma_s$. Since $t<T$, and $W$ is Lipschitz at least up until $T$, the well-posedness of $(\gamma_s)_{s\in[0,t]}$ and hence of $(X_s,Y_s)_{s\in[0,t]}$ is fairly simple \cite{probabilistic-mfg}.
Throughout the proof we will be using the notation  $W^x_s=W(t-s,X_s,\noise_s,\mu_s)$, $F^x_s=F(X_s,\noise_s,\mu_s,W^x_s)$, $G^x_s=G(X_s,\noise_s,\mu_s,W^x_s)$ and similarly $W^y_s=W(t-s,Y_s,\noise_s,\nu_s)$ and so on. 
By the dynamic programming principle satisfied by $W$, we know
\[V(t,x_0,y_0,\noise_0,m_0)=\esp{\langle W^x_s-W^y_s,x-y\rangle+\int_0^{t-s}\langle G_u^x-G_u^y,x-y\rangle du}.\]
Because $(X_s)_{s\in[0,t]}$ and $(Y_s)_{s\in[0,t]}$ have been generated with the same set of Brownian motion the following holds 
\[x-y=X_{t-s}-Y_{t-s}+\int_0^{t-s}(F^x_u-F^y_u)du.\]
From which we get 
\begin{align}
\label{DPP b(p)}
V(t,x_0,y_0,\noise_0,\gamma_0)&=\esp{V(s,X_{t-s},Y_{t-s},\noise_{t-s},\gamma_{t-s})}\\
&\nonumber +\esp{\int_0^{t-s}\left(\langle G_u^x-G_u^y,x-y\rangle+\langle F^x_u-F^y_u,W^x_s-W^y_s\rangle\right) du}.
\end{align}
We are also going to need the following formula
\begin{equation}
\label{square martingale inequality b(p)}
\forall s\in[0,t] \quad |W(t,x_0,\noise_0,\mu_0)-W(t,y_0,\noise_0,\nu_0)|^2\leq \esp{\left|W^x_s-W^y_s+\int_0^{t-s}(G^x_u-G^y_u)du\right|^2},
\end{equation}  
which is a consequence of Jensen inequality applied to \eqref{Def: Lipschitz solution b(p)}. Let us now take $\beta$ be a $C^1$ positive function of time, and consider \begin{align*}\Lambda(t,x_0,y_0,\noise_0,\gamma_0)&=V(t,x_0,y_0,\noise_0,\gamma_0)-\beta(t)|W(t,x_0,\noise_0,\mu_0)-W(t,y_0,\noise_0,\nu_0)|^2,\\
&=\underbrace{V(t,x_0,y_0,\noise_0,\gamma_0)}_{I}-\beta(s)\underbrace{|W(t,x_0,\noise_0,\mu_0)-W(t,y_0,\noise_0,\nu_0)|^2}_{II}\\
&\quad \quad\quad \quad -(\beta(t)-\beta(s))|W(t,x_0,\noise_0,\mu_0)-W(t,y_0,\noise_0,\nu_0)|^2.
\end{align*}
Applying \eqref{DPP b(p)} to $I$ and \eqref{square martingale inequality b(p)} to $II$ leads to the following inequality:
\begin{align*}
\Lambda(t,x_0,y_0,\noise_0,\gamma_0)&\geq \esp{\Lambda(s,X_{t-s},Y_{t-s},\noise_{t-s},\gamma_{t-s})+\int_0^{t-s}\left(\langle G_u^x-G_u^y,x-y\rangle+\langle F^x_u-F^y_u,W^x_s-W^y_s\rangle\right) du}\\
&-\beta(s)\esp{\left|\int_0^{t-s} (G^x_u-G^y_u)du\right|^2+2\int_0^{t-s}\langle G^x_u-G^y_u,W^x_s-W^y_s\rangle du}\\
&-(\beta(t)-\beta(s))|W(t,x_0,\noise_0,\mu_0)-W(t,y_0,\noise_0,\nu_0)|^2.
\end{align*}
We may finally take a look at our end goal 
\[Z_\beta(t,\gamma_0,\noise_0)=\int_{\reels^{2d}} \Lambda(t,x_0,y_0,\noise_0,\gamma_0) \gamma_0(dx_0,dy_0).\]
For that, let us first remind that for any continuous function f, $s<t$ letting $\mathcal{F}^\noise_s=\sigma((\noise_u)_{u\leq s})$
\begin{align*}
\forall s\leq t \quad \int_{\reels^{2d}}\esp{f\left((X,Y)_s^{(x_0,y_0)}\right)}\gamma_0(dx_0,dy_0)&=\int_{\reels^{2d}}\esp{\espcond{f\left((X,Y)^{(x_0,y_0)}_s\right)\gamma_0(dx_0,dy_0)}{\mathcal{F}^\noise_s}},\\
&=\esp{f(x,y)\gamma_s(dx,dy)}.
\end{align*}
which is a consequence of almost sure pathwise uniqueness to our SPDE \cite{probabilistic-mfg}. 
With that in mind, it is easy to see that 
\begin{align*}
\int_{\reels^{2d}}\esp{\Lambda (s,X^{x_0}_{t-s},Y^{y_0}_{t-s},\noise_{t-s},\gamma_{t-s})}\gamma_0(dx_0,dy_0)&=\esp{\int_{\reels^{2d}} \Lambda (s,x,y,\noise_{t-s},\gamma_{t-s})\gamma_{t-s}(dx,dy)},\\
&=\esp{Z(s,\noise_{t-s},m_{t-s})}.
\end{align*}
Finally, we obtain the inequality satisfied by $Z$
\begin{align}
\label{Z inequality b(p)}
\nonumber Z(t,\noise_0,\gamma_0)&\geq \esp{Z(s,\noise_{t-s},\gamma_{t-s})+\int_{\reels^{2d}}\int_0^{t-s}\left(\langle G_u^x-G_u^y,x-y\rangle+\langle F^x_u-F^y_u,W^x_s-W^y_s\rangle\right) \gamma_0(dx,dy)}du\\
\nonumber &-\beta(s)\int_{\reels^{2d}}\esp{\left|\int_0^{t-s} (G^x_u-G^y_u)du\right|^2+2\int_0^{t-s}\langle G^x_u-G^y_u,W^x_s-W^y_s\rangle du}\gamma_0(dx,dy)\\
&-(\beta(t)-\beta(s))\int_{\reels^{2d}}|W(t,x,\noise_0,\mu_0)-W(t,y,\noise_0,\nu_0)|^2\gamma_0(dx,dy).
\end{align}
\noindent\textit{Step 2: viscosity supersolution}\\
Let $\varphi\in H_{test}$  and assume that there exists $(t^*,\noise^*,\gamma^*)$ such that
\[\min (Z-\varphi)=Z(t^*,\noise^*,\gamma^*)-\varphi(t^*,\noise^*,\gamma^*).\]
By adding a constant to $\varphi$ if necessary, we may assume without loss of generality that this minimum is equal to 0. 
Because we must have $Z(s,\noise_{t-s},\gamma_{t-s})\geq \varphi(s,\noise_{t-s},\gamma_{t-s})$ and the equality holds at $(t^*,\noise^*,\gamma^*)$, $\varphi$ satisfies the inequality \eqref{Z inequality b(p)} for any $s\leq t$. Applying Lemma 5.15 of \cite{convergence-problem} to $\varphi$, and dividing by $t-s$ we can take the limit without much difficulty. Pointwise convergence holds thanks to the mean value theorem, and we may directly apply the dominated convergence theorem as we have uniform bounds on $\varphi,\partial_t\varphi, \nabla_\noise \varphi, D^2_\noise \varphi$ and the integral terms in $D_\gamma\varphi, D_yD_\gamma \varphi$.
\end{proof}
\begin{remarque}
    Observe that the proof still works if we only assume a test function $\varphi$ satisfy the assumptions we put on test functions in $H_{test}$ only in a neighbourhood of the point of minimum $(t^*,\noise^*,\gamma^*)$. Indeed, by introducing an appropriate sequence of stopping time we can then bound $\varphi$, and since the convergence of $\gamma_{t^*-s}$ to $\gamma^*$ holds in the $\mathcal{W}_q$ topology for $q=\sup\{p, \gamma^*\in\mathcal{P}_p(\reels^l)\}$ we know that for $t^*-s$ sufficiently small $(t^*-s,\noise_{t^*-s},\gamma_{t^*-s})$ belongs to any neighbourhood (in $\mathcal{W}_q$), letting the conclusion be unchanged. In particular, this justifies that we may take 
    \[\varphi: (s,\noise,\gamma)\mapsto -|\noise|^3-E_3^3(\gamma)-\frac{1}{t-s},\]
    as we did in the proof of Proposition \ref{prop: uniqueness W2 b(p)}.
\end{remarque}
With this technical lemma proven, we may use a comparison principle with $0$ to show the following
\begin{lemma}
\label{Z beta positive b(p)}
    Under Hypothesis \ref{hyp: lip W2 b(p)} and \ref{hyp: (G,F,w0) b(p) W2}, if $W$ is a Lipschitz solution to \eqref{eq: ME W2 Lip b(p)} on $[0,T)$ for some $T>0$, then there exists a strictly positive function of time $\beta$ such that 
    \[\forall t<T, \quad (\gamma,\theta)\in\mathcal{P}_2(\reels^{2d})\times \reels^n \quad Z_\beta(t,\theta,\gamma)\geq 0,\]
    for $Z_\beta$ defined as in \eqref{def: Z beta b(p)}.
\end{lemma}
\begin{proof}
We first make the following observation: fixing $(t,\gamma,\theta)$, for any $\lambda \in[0,1]$
 \begin{gather*}
 \int_{\reels^{2d}} ( G^{u}-G^{v})\cdot (W^{u}-W^{v}) \gamma(du,dv)\\
 \leq \int_{\reels^{2d}} |G(u,\pi_d \gamma,\noise,\lambda W^{u}+(1-\lambda )W^{v})-G(v,\pi_d \gamma,\noise,\lambda W^{u}+(1-\lambda )W^{v})|^2m(du,dv)\\
 +(2\|G\|_{Lip}+1)\int_{\reels^{2d}}|W^{u}-W^{v}|^2\gamma(du,dv),
 \end{gather*}
 with the notation $W^u=W(t,u,\theta,\pi_d\gamma), W^v=W(t,v,\theta,\pi_{-d}\gamma)$ that we remind. 
 As a consequence, under Hypothesis \ref{hyp: (G,F,w0) b(p) W2}, for any $(t,\gamma,\noise)\in[0,T_c)\times\mathcal{P}_2(\reels^{2d})\times\reels^n$ the following holds
\[
\begin{array}{c}
\displaystyle\int_{\reels^{2d}} (F^{u}-F^{v})\cdot(W^{u}-W^{v})d\gamma+\int_{\reels^{2d}}(G^{u}-G^{v})\cdot(u-v)\gamma(du,dv)\\
\displaystyle-2\beta(t) \int_{\reels^{2d}} \langle G^{u}-G^{v},W^{u}-W^{v}\rangle \gamma(du,dv)-\frac{d\beta}{dt}\int_{\reels^{2d}} |W^{u}-W^{v}|^2\gamma(du,dv)\\
\geq 0,
\end{array}
\]
for  $\beta(t)=\beta_0e^{-k t}$ with $(\beta_0,k)$ chosen such that:
\[
\begin{array}{c}
     \beta_0\leq \alpha,\\
     k\geq 1+4\|G\|_{Lip}.
\end{array}
\]
For such a $\beta$, we recover using Lemma \ref{eq: Z supersol b(p) W2} that $Z_\beta$ is a viscosity supersolution of 
\begin{gather*}
\partial_t Z_\beta(t,\noise,m)+b(\noise)\cdot \nabla_\noise Z_\beta-\sigma_\noise \Delta_\noise Z_\beta\\
\nonumber+\int_{\reels^{2d}}\left(\begin{array}{c}
     F^{m_1}  \\
     F^{m_2}
\end{array}
\right)
\cdot D_\gamma Z_\beta(t,\noise,m,u,v) m(du,dv)-\sigma_x\int_{\reels^{2d}}\text{Tr}\left(BD_{(u,v)} D_mZ_\beta(t,\noise,m,u,v)\right)m(du,dv)\\
\geq 0,
\end{gather*}
with
\[Z_\beta|_{t=0}\geq 0.\]
Since $W$ is Lipschitz, $Z_\beta$ is a continuous function with at most quadratic growth. Fixing $t<T$ and using the same reasoning as in the proof of Lemma \ref{smooth comparison with 0}, this means the function
\[Z_\alpha: (s,\theta,\gamma)\to Z_\beta(s,\theta,\gamma)+\alpha e^{\kappa s}(1+E_3^3(\gamma)+|\theta|^3)+\frac{\alpha}{t-s},\]
reaches a minimum $(s^*,\theta^*,\gamma^*)$ with $s^*<t$, for any positive $\alpha,\kappa$. Fixing $\alpha>0$, we now assume by contradiction that there exists a point $(s,\theta,\gamma)$ with $s<t$ such that $Z_\alpha(s,\theta,\gamma)<0$ . By assumption \ref{hyp: (G,F,w0) b(p) W2}, this implies the minimum of $Z_\alpha$ is not reached for $s^*=0$. Using the viscosity property at this point of minimum, we find that there exists a constant $c$ depending on $t$ and the linear growth of $F,G,b,W$ such that 
\[0\geq \frac{\alpha}{(t-s^*)^2}+\alpha e^{\kappa s^*}(\kappa-c)(1+|\theta^*|^3+E_3^3(\gamma^*)).\]
For $\kappa>c$ we get a contradiction yielding 
\[\underset{[0,t)\times\reels^n\times \mathcal{P}_2(\reels^{2d})}{\inf} Z_\alpha \geq 0. \]
Since this is true for any $\alpha>0$ and $t<T$, the lemma is proved. 
\end{proof}
It now remains to show that the non negativity of $Z_\beta$ is a sufficiently strong property to get Lipschitz estimates on the solution $W$. We first remark that this gives directly an estimate on the Lipschitz norm in the state variable $x$.
\begin{lemma}
\label{lemma: estimate in x b(p) W2}
    Let $W:\reels^d\times  \mathcal{P}_2(\reels^d)\to \reels^d$ be a continuous function such that for some $\beta>0$ and for all $(X,Y)\in (L^2(\Omega,\reels^d))^2$ 
    \begin{equation}
    \label{eq: space estimate lip sol b(p)}
    \esp{(W(X,\mathcal{L}(X))-W(Y,\mathcal{L}(Y)))\cdot (X-Y)}\geq \beta \esp{|W(X,\mathcal{L}(X))-W(Y,\mathcal{L}(Y))|^2},\end{equation}
    then 
    \begin{enumerate}
    \item[-] $\forall (X,Y)\in (L^2(\Omega,\reels^d))^2 \quad \esp{|W(X,\mathcal{L}(X))-W(Y,\mathcal{L}(Y))|^2}\leq\frac{1}{\beta^2} \esp{|X-Y|^2}$
        \item[-] $\|W(\cdot)\|_{Lip(x)}:=\underset{\mu\in \mathcal{P}_2(\reels^d)}{\sup}\|W(\cdot,\mu))\|_{Lip} \leq \frac{1}{\beta}$
    \end{enumerate}
\end{lemma}
\begin{proof}
    The first part of the claim is just a direct consequence of applying Cauchy-Schwartz inequality in $(L^2(\Omega,\reels^d),\langle \cdot,\cdot\rangle|_{L^2})$ to \eqref{eq: space estimate lip sol b(p)}. As for the second part, we define 
    \[f(\mu,x,y)=(W(x,\mu)-W(y,\mu))\cdot(x-y)-\beta|W(x,\mu)-W(y,\mu)|^2.\]
As a direct consequence of the above inequality, for any $X,Y$ both of law $\mu$
\[\esp{f(\mu,X,Y)}\geq 0.\]
Let us also observe that $f$ is a continuous, symmetric function of $ (x,y)\in\reels^{2d}$ such that \[\forall \mu \in\mathcal{P}(\reels^d)\quad \forall x\in\reels^d \quad f(\mu,x,x)=0.\] 
We may now call upon Proposition \ref{prop: expectation to pointwise} to conclude that 
\[\forall \mu \in \mathcal{P}(\reels^d),\quad \forall (x,y)\in \reels^{2d} \quad f(\mu,x,y)\geq 0.\]
From the non-negativity of $f$ and an application of Cauchy-Schwartz inequality we finally conclude that
\[\|W(\cdot)\|_{Lip(x)}:=\underset{\mu\in \mathcal{P}_2(\reels^d)}{\sup}\|W(\cdot,\noise,\mu))\|_{Lip} \leq \frac{1}{\beta}.\]
\end{proof}
The non negativity of $Z_\beta$ for a Lipschitz solution defined on $[0,T)$ is equivalent to $\forall t<T, \theta\in \reels^n, (X,Y)\in (L^2(\Omega,\reels^d))^2, X\sim \mu, Y\sim \nu$
\[\esp{(W(t,X,\theta,\mu)-W(t,Y,\theta,\nu))\cdot(X-Y)}\geq \beta(t) \esp{|W(t,X,\theta,\mu)-W(t,Y,\theta,\nu)|^2}.\]
Fixing $(t,\theta)$, the above lemma can be applied to deduce an estimate on the Lipschitz norm of $W$ in $x$. To show that $W$ is also Lipschitz in $\mathcal{P}_2(\reels^d)$, the trick is to notice that the inequality implied by the first claim
\begin{equation}
\label{inequality: W Lipschitz in L2}
\forall (X,Y)\in (L^2(\Omega,\reels^d))^2 \quad  \esp{|W(t,X,\noise,\mathcal{L}(X))-W(t,Y,\noise,\mathcal{L}(Y))|^2}\leq \frac{1}{\beta^2(t)}\esp{|X-Y|^2},
\end{equation}
is sufficient to get a Lipschitz estimate in $\mathcal{W}_2$ for Lipschitz solutions. 
\begin{thm}
\label{thm: b(p) W2}
Under Hypothesis \ref{hyp: lip W2 b(p)} and Hypothesis \ref{hyp: (G,F,w0) b(p) W2}, there exists a unique Lipschitz solution to \eqref{eq: ME W2 Lip b(p)} on $[0,+\infty)$.
\end{thm}
\begin{proof}
Local existence of a Lipschitz solution on $[0,T_c)$ for some $T_c>0$ follows from Theorem \ref{lip sol b(p)}. Thanks to Lemma \ref{Z beta positive b(p)} and \ref{lemma: estimate in x b(p) W2} we know there exists a function $\beta$ of the form $\beta(t)=\beta_0e^{-\kappa t}$ such that \[\forall t<T_c \quad \|W(t,\cdot)\|_{Lip(x)}\leq \frac{1}{\beta(t)},\]
and \eqref{inequality: W Lipschitz in L2} holds. We are now going to use this inequality to get a stability estimate with respect to initial conditions for the following stochastic Fokker-Planck equation
\[s\leq t<T_c, \quad  \left\{\begin{array}{l}
 d\mu_s^{\mu_0}=\left(-\text{div}\left(F(x,\noise_s,W(t-s,x,\noise_s,\mu_s^{\mu_0})\right)+\sigma_x \Delta \mu_s^{\mu_0}\right)ds \quad 
\mu_0\in\mathcal{P}_2(\reels^d),\\
d\noise_s=-b(\noise_s)ds+\sqrt{2\sigma_\noise}dB^\noise_s \quad \noise_0=\noise\in\reels^n.\\
\end{array}\right.\]
Since $t<T_c$, $W$ is Lipschitz in both the space and the measure variable and so existence of a strong solution to this system is not a problem. Consider two initial conditions $\mu_0,\nu_0\in\mathcal{P}_2(\reels^d)$ and let $\rho_0\in \Gamma(\mu_0,\nu_0)$ be chosen as an optimal coupling for $\mathcal{W}_2$. We define $(\rho_s)_{s\in[0,t]}$ by
\[ d\rho_s=\left(-\text{div}\left(\left(\begin{array}{c}F(x,\noise_s,\pi_d\rho_s,W(t-s,x,\noise_s,\pi_d\rho_s))\\ F(y,\noise_s,\pi_{-d}\rho_s,W(t-s,y,\noise_s,\pi_{-d}\rho_s)) \end{array}\right)\rho_s\right)+\sigma_x \text{Tr}\left(BD^2_{(x,y)}\rho_s\right)\right)ds.\]
By pathwise almost sure uniqueness \cite{probabilistic-mfg}, it holds that 
\[\forall s\leq t \quad \pi_d\rho_s=\mu_s^{\mu_0} \quad \pi_{-d} \rho_s= \mu_s^{\nu_0} \quad a.s. \]
Let us now fix $\omega\in \Omega$ up to a negligible set. Since $\rho_0\in\mathcal{P}_2(\reels^{2d})$ so does $\rho_s(\omega):=\rho^\omega_s$. Using the definition of a weak solution combined with the dominated convergence theorem we get that for any $s\leq t$
\begin{align*}
&\int_{\reels^{2d}} |x-y|^2 \rho^\omega_s(dx,dy)=\mathcal{W}^2_2(\mu_0,\nu_0)\\
&+\int_0^s\int_{\reels^{2d}}\left(F(x,\noise_u,\pi_d\rho^\omega_u,W(t-u,x,\noise_u\pi_d\rho^\omega_u))-F(y,\noise_u,\pi_{-d}\rho^\omega_u,W(t-u,\noise_u,\pi_{-d}\rho^\omega_u)\right)\cdot(x-y)\rho^\omega_u(dx,dy)du.
\end{align*}
Using \eqref{inequality: W Lipschitz in L2} we deduce that 
\[\int_{\reels^{2d}} |x-y|^2 \rho^\omega_s(dx,dy)\leq \mathcal{W}_2^2(\mu_0,\nu_0)+(1+\|F\|_{Lip}+\frac{1}{\beta^2(s)})\int_0^s\int_{\reels^{2d}} |x-y|^2 \rho^\omega_u(dx,dy)du.\]
Since this is true for almost every $\omega$ Grönwall's Lemma allows us to conclude that there exists a constant $C$
\[\forall s\leq t \quad \mathcal{W}_2(\mu^{\mu_0}_s,\mu^{\nu_0}_s)\leq C\mathcal{W}_2(\mu_0,\nu_0) \quad a.s\]
The strength of this estimate is that the constant C does not depend directly on the Lipschitz constant of $W$ in $\mathcal{W}_2$ but rather on the Lipschitz constant of its lift 
\[\tilde{W}: \left\{\begin{array}{c}
     L^2(\Omega,\reels^d)\to L^2(\Omega,\reels^d) \\
     X\mapsto W(t,X,\theta,\mathcal{L}(X))
\end{array}\right.
\] 
on $L^2(\Omega,\reels^d)$ which we know is bounded by $\frac{1}{\beta(t)}$ thanks to \eqref{inequality: W Lipschitz in L2}. As $W$ is given by the representation formula \eqref{Def: Lipschitz solution b(p)}, thanks to the above estimates we deduce that 
\[|W(t,x,\noise,\mu_0)-W(t,x,\noise,\nu_0)|\leq C\left(1+t+\int_0^t\|W(s,\cdot)\|_{Lip(\mathcal{W}_2)}ds\right)\mathcal{W}_2(\mu_0,\nu_0), \]
for a constant C depending on $\|F\|_{Lip},\|W_0\|_{Lip},\|G\|_{Lip}$ and $\|W(t,\cdot)\|_{Lip(x)}$.
Dividing by $\mathcal{W}_2(\mu_0,\nu_0)$ and taking the supremum for $\mu_0\neq \nu_0$, we apply Grönwall's Lemma to conclude that there exists another constant such that
\[\forall s\leq t \quad  \|W(s,\cdot)\|_{Lip(\mathcal{W}_2)}\leq C.\]
Let us remark that since $(\noise_s)_{s\geq 0}$ is an autonomous process with a Lipschitz driver, and $F,G,W_0$ are also Lipschitz in $\noise$, a similar argument gives estimates on $\|W(t,\cdot)\|_{Lip(\noise)}$

\quad

\noindent \textit{Conclusion}\\

Now that we have estimates on the Lipschitz constant of $W$ in $\reels^d\times\reels^n\times \mathcal{P}_2(\reels^d)$, we may conclude to the existence of a Lipschitz solution on any time interval. Let us assume by contradiction that $T_c<T$ for some $T>0$. Then 
\[\underset{t\to T_c}{\lim} \|W(t,\cdot)\|_{Lip}=+\infty\]
must hold. However, we have a uniform estimate on the Lipschitz semi-norm of $W$ in time that holds for any $t<T_c$ with a constant depending on $T,\|F\|_{Lip},\|b\|_{Lip},\|W_0\|_{Lip},\alpha$ only. Obviously, this contradicts the claim that $T_c<T$. Because it is true for any $T<\infty$, we conclude to the existence of a Lipschitz solution on any time interval. 
\end{proof}

\quad

\subsection{Comparison with previous results on MFG master equations}
As mentioned before, the above existence result applies to the MFG master equation \eqref{eq: ME W2 Lip} whenever $F=D_p H,G=-D_xH$ and $W_0=\nabla_x U_0$. However, the hypotheses we made are different from the conditions found in the literature on displacement monotone mean field games \cite{disp-monotone-1,disp-monotone-2,globalwellposednessdisplacementmonotone}. Forgetting the common noise $\noise$, the usual assumptions are as follows
 \begin{hyp}
 \label{hyp: disp monotone (F,G,W0)}
     $\forall (X,Y,U,V)\in\reels^n\times \left(L^2(\Omega,\reels^d)\right)^4$, 
\begin{equation*}
    \begin{array}{c}
         \esp{\left(W_0(X,\mathcal{L}(X)-W_0(Y,\mathcal{L}(Y))\right)\cdot(X-Y)}\geq 0,  \\
    \esp{F(X,\mathcal{L}(X),U)-F(Y,\mathcal{L}(Y),V)\cdot(U-V)+(G(X,\mathcal{L}(X),U)-G(Y,\mathcal{L}(Y),V))\cdot(X-Y)}\geq 0,
    \end{array}
\end{equation*}
\begin{equation}
\label{strong monotonicity F disp monotone}
\exists \alpha >0, \quad \forall(x,p,q,m)\in\reels^{3d}\times \mathcal{P}_2(\reels^d) \quad (F(x,m,p)-F(x,m,q))\cdot(p-q)\geq  \alpha |p-q|^2 
\end{equation}
 \end{hyp}
Obviously, Hypothesis \ref{hyp: (G,F,w0) b(p) W2} and Hypothesis \ref{hyp: disp monotone (F,G,W0)} are in dichotomy. This is mainly due to two differences, the first one is that we did not assume that the functions $(F,G,W_0)$ were gradients, to account for mean field forward backward differential equations. This has its importance as the following holds:
\begin{lemma}
    Let $x\mapsto U(x)$ be a $C^1$ function with a Lipschitz continuous gradient, the following two propositions are equivalent
    \begin{enumerate}
        \item[(i)]$\forall (x,y)\in (\reels^d)^2$ \[(\nabla U(x)-\nabla U(y))\cdot(x-y)\geq 0.\]
        \item[(ii)]$\exists \alpha>0, \quad \forall (x,y)\in (\reels^d)^2$ \[(\nabla U(x)-\nabla U(y))\cdot(x-y)\geq \alpha |\nabla_x U(x)-\nabla U(y)|^2.\]
    \end{enumerate}
    \begin{proof}
        $(ii)\implies (i)$ is self-evident. As for the other implication, since $(i)$ implies that $x\mapsto U(x)$ is a convex function with a Lipschitz gradient, it is a classic result from convex analysis \cite{Convex_book} that 
        \[\forall (x,y)\in (\reels^d)^2 \quad (\nabla U(x)-\nabla U(y))\cdot(x-y)\geq \frac{1}{L}|\nabla U(x)-\nabla U(y)|^2,\]
        where $L=\|\nabla U\|_{Lip}$.
    \end{proof}
\end{lemma}
The other point being that the strategy of proof is slightly different in \cite{disp-monotone-1,disp-monotone-2}. The strong joint $L^2-$monotonicity Hypothesis \ref{hyp: (G,F,w0) b(p) W2} gives us directly an estimate on the Lipschitz norm of $W$ in $\mathcal{W}_2$ while in their approach it is obtained by a weaker monotonicity assumption combined with strong convexity of the Hamiltonian in the non-linearity. A reasonable question is whether we can present a long time existence result for Lipschitz solutions to \eqref{eq: ME W2 Lip b(p)} that resolves to previous results given in the displacement monotone framework whenever $(F,G,W_0)$ are gradients. 
\begin{thm}
    Under Hypothesis \eqref{hyp: monotonie not strong b(p)}, if $\forall (x,y,p,q,\noise,\mu)\in(\reels^d)^4\times \reels^n\times \mathcal{P}_2(\reels^d)$
\begin{equation}
\label{storng monotonicity in p}
F(x,\noise,\mu,p)-F(y,\noise,\mu,q))\cdot (p-q)+(G(x,\noise,\mu,p)-G(y,\noise,\mu,q))\cdot(x-y)\geq \alpha_H|p-q|^2,\end{equation}
Then there exists a unique Lipschitz solution to \eqref{eq: ME W2 Lip b(p)} on $[0,+\infty)$.
\end{thm}
\begin{proof}
The fact that under \eqref{storng monotonicity in p} we can estimate the Lipschitz norm of $W$ in $x$ by coming back to the characteristics of $W$ is well known \cite{Lions-college,noise-add-variable}. Once the Lipschitz norm in $x$ is bounded, an estimate in $\mathcal{W}_2$ follows from adapting Proposition 3.4 of \cite{globalwellposednessdisplacementmonotone} to Lipschitz solutions.
\end{proof}

We now show that whenever $(F,G)$ are gradients \eqref{storng monotonicity in p} holds under Hypothesis \ref{hyp: disp monotone (F,G,W0)}.
\begin{lemma}
    Let $F,G$ be continuous functions such that 
    $\forall (x,y,p,q)\in (\reels^d)^4$ 
    \begin{equation}
    \label{eq: (F,G) monotone+ grad = strong monotone}
    (F(x,p)-F(y,q))\cdot (p-q)+(G(x,p)-G(y,q))\cdot(x-y)\geq 0.
    \end{equation}
    If there exists a function $H:(\reels^d)^2\to \reels$ such that $F=D_p H,G=-D_x H$ and $H$ is $\alpha_H-$strongly convex in $p$ uniformly in $x$, then 
    $\forall (x,y,p,q)\in (\reels^d)^4$ 
    \begin{gather*}
    (F(x,p)-F(y,q))\cdot (p-q)+(G(x,p)-G(y,q))\cdot(x-y)\\
    \geq \alpha_H|p-q|^2.
    \end{gather*}
\end{lemma}
\begin{proof}
    We first assume that $H$ is $C^2$. 
    \begin{align*}
    &(F(x,p)-F(y,q))\cdot (p-q)+(G(x,p)-G(y,q))\cdot(x-y)\\
&=\int_0^1 \left(\begin{array}{c}
         p-q  \\
         x-y 
    \end{array}\right)
    \cdot \left(\begin{array}{cc}
         D^2_p H & D^2_{xp} H  \\
         -D^2_{xp} H & -D^2_x H
    \end{array}\right)((1-t)x+ty,(1-t)p+tq)\cdot\left(\begin{array}{c}
         p-q  \\
         x-y 
    \end{array}\right) dt,\\
   &=\int_0^1 \left(\begin{array}{c}
         p-q  \\
         x-y 
    \end{array}\right)
    \cdot \left(\begin{array}{cc}
         D^2_p H & 0_{\mathcal{M}_d(\reels)}  \\
         0_{\mathcal{M}_d(\reels)} & -D^2_x H
    \end{array}\right)((1-t)x+ty,(1-t)p+tq)\cdot\left(\begin{array}{c}
         p-q  \\
         x-y 
    \end{array}\right) dt,\\
    &\geq \alpha_H |p-q|^2
    \end{align*}
    Whenever $H$ is only $C^1$, we consider a sequence $H_\varepsilon$ obtained by a monotonicity preserving regularization. This can be done by convolution with a smooth, compactly supported, positive kernel of mass 1. Since $H_\varepsilon$ satisfies \eqref{eq: (F,G) monotone+ grad = strong monotone}, is $\alpha_H-$strongly convex in $p$ and converges to $H$ in $C^1$, the result is obtained by taking the limit. 
\end{proof}
\begin{remarque}
    In fact, this shows that assuming $F,G$ are the gradient of an Hamiltonian is a very constraining assumption. In general, it is not sufficient for $G$ to be monotone in $x$ and $F$ in $p$ to get joint monotonicity which is much stronger. However, in the gradient case this becomes a sufficient assumption as the cross derivatives cancel out. 
\end{remarque}
\begin{remarque}
It might appear surprising that under strong monotonicity in $p$, we can show the existence of solutions under weaker assumptions on the monotonicity of $W_0$. It was already observed in \cite{Lions-college} that strong monotonicity in $p$ has a greater regularizing effect on the equation and allows starting from less regular initial data.
\end{remarque}
\subsection{Noise depending on the distribution of players}
\label{subsection: b(p,m) W2}
We now turn back to the original problem of \eqref{eq: ME W2 Lip}
\begin{equation*}
\left\{
\begin{array}{c}
\displaystyle\partial_t W+F(x,\noise,m,W)\cdot \nabla_x W+b[W](t,\noise,m)\cdot\nabla_\noise W-\sigma_x \Delta_x W-\sigma_\noise \Delta_\noise W\\
\displaystyle +\int_{\reels^d} F(y,\noise,m,W)\cdot D_m W(t,x,\noise,m)(y)m(dy)-\sigma_x\int_{\reels^d} \text{div}_y (D_mW(t,x,\noise,m)(y))m(dy)\\
\displaystyle =G(x,\noise,m,W) \text{ in } (0,T)\times\reels^d\times\reels^n\times\mathcal{P}_2(\reels^d),\\
\displaystyle W(0,x,\noise,m)=W_0(x,\noise,m) \text{ for } (x,\noise,m)\in\reels^d\times\reels^n\times\mathcal{P}_2(\reels^d).
 \end{array}
 \right.
\end{equation*}
We now make a precise statement on the regularity of $b$ with respect to its functionnal argument
\begin{hyp}
\label{hyp: Lipschitz W2}
$\exists C>0 \quad \forall (x,y)\in(\reels^d)^2, (\noise,\noiseb)\in(\reels^n)^2, (\mu,\nu)\in(\mathcal{P}_2(\reels^d))^2, (u,v)\in(\reels^d)^2,$
    \begin{enumerate}
    \item[-] $|W_0(x,\noise,\mu)-W_0(y,\noiseb,\nu)|\leq C\left(|x-y|+|\noise-\noiseb|+\mathcal{W}_2(\mu,\nu)\right).$
        \item[-] $|F(x,\noise,\mu,u)-F(y,\noiseb,\nu,v)|+|G(x,\noise,\mu,u)-G(y,\noiseb,\nu,v)|\leq C\left(|x-y|+|\noise-\noiseb|+|u-v|+\mathcal{W}_2(\mu,\nu)\right).$
        \item[-]For any two continuous functions $f,g: \reels^d\to \reels^d$ and any coupling $\gamma\in \Gamma(\mu,\nu)$
        \[|b(\noise,\mu,f(\cdot))-b(\noiseb,\nu,g(\cdot))|\leq C\left(|\noise-\noiseb|+\mathcal{W}_2(\mu,\nu)+\sqrt{\int_{(\reels^d)^2}|f(x)-g(y)|^2\gamma(dx,dy)}\right).\]
    \end{enumerate}
\end{hyp}
Obviously, this restricts quite a bit how $b$ may depend on $W$ in general. From a probabilistic point of view, it is natural to assume that the dependency of $b$ on $W$ be of the form 
\[b(\theta,\mu,W(t,\cdot,\theta,\mu))=\tilde{b}(\theta,\mu, W(t,\cdot,\theta,\mu)_\# \mu).\]
If $\tilde{b}$ is Lipschitz for $\mathcal{W}_2$ in its last variable, then Hypothesis \ref{hyp: Lipschitz W2} holds. We do not require $b$ to be of this form and keep only this continuity assumption.
\begin{exemple}
Hypothesis \ref{hyp: Lipschitz W2} holds for coefficient $b$ of the form
\[b(\noise,\mu,W(t,\cdot,\noise,\mu))=\tilde{b}\left( \noise,\mu,\displaystyle\int_{\reels^d}h(y,\noise,\mu,W(t,y,\noise,\mu))\mu(dy)\right),\]
whenever $\tilde{b}$ and $h$ are Lipschitz functions in all variable. In particular, if we come back to the original problem this allows to take a $b$ that depends on the average control of the players 
\[\int_{\reels^d}\nabla_pH(y,\noise,\mu,\nabla_x U(t,y,\noise,\mu))\mu(dy).\]
\end{exemple}
\begin{remarque}
    Results obtained in this section still hold if $b$ depends on $x$ so long as it is in a Lipschitz fashion. We do not treat directly this case since this changes completely the interpretation of the noise process $\theta$ and such situation does not correspond to any mean field game no matter the coefficients. Nevertheless, it is a natural extension for mean field forward-backward systems. 
\end{remarque}
\subsubsection{Existence of solutions}
Just as the finite dimensional counterpart of this equation, some results were already given on the associated mean field forward backward stochastic differential equation under $G-$monotonicity \cite{mean-field_G-monotonicity,G-monotonicity} whenever $\sigma_x=0$. However, we here also treat situations in which $\sigma_x>0$, which raises some issues and does not allow us to use directly those results. We give new existence results for this equation, in dichotomy with G-monotonicity, in particular those results may be of interest for mean field forward backward stochastic differential equations outside mean field games theory. They consist in an extension to nonlinear systems of transport equations on the space of measures of a method we first developed in \cite{noise-add-variable} for nonlinear systems of transport equations. Whenever $b$ depends on $m$ (whether it be directly or through $W$), the evolution of the distribution of players is coupled with the noise process and its dynamic cannot be considered independently. This raise the question of what kind of monotonicity is natural in this case. If the solution $W$ of \eqref{eq: ME W2 Lip} were to take value in $\reels^{n+d}$ instead of $\reels^d$, then at least formally by rewriting the equation in function a new state variable $\tilde{x}=(x,\theta)$ we expect to fall back in the case of \eqref{eq: ME W2 Lip b(p)} that we treated. For mean field games we know that this is not true, however it is always possible to find a function $(t,x,\theta,\mu)\mapsto V(t,x,\theta,\mu)\in\reels^n$ satisfying 
\begin{equation}
\label{def: V for b(m,p)}
\left\{
\begin{array}{c}
\displaystyle\partial_t V+F(x,\noise,m,W)\cdot \nabla_x V+b[W](t,\noise,m)\cdot\nabla_\noise V-\sigma_x \Delta_x V-\sigma_\noise \Delta_\noise V\\
\displaystyle +\int_{\reels^d} F(y,\noise,m,W)\cdot D_m V(t,x,\noise,m)(y)m(dy)-\sigma_x\int_{\reels^d} \text{div}_y (D_mV(t,x,\noise,m)(y))m(dy)\\
\displaystyle =G_V(x,\noise,m,W,V) \text{ in } (0,T)\times\reels^d\times\reels^n\times\mathcal{P}_2(\reels^d),\\
\displaystyle V(0,x,\noise,m)=V_0(x,\noise,m) \text{ for } (x,\noise,m)\in\reels^d\times\reels^n\times\mathcal{P}_2(\reels^d),
 \end{array}
 \right.
\end{equation}
for some coefficients $(V_0,G_V)$, so as to complete in $\reels^{d+n}$ the equation satisfied by $W$ in $\reels^d$. The remaining question is to know under which condition does there exist a $V$ satisfying \eqref{def: V for b(m,p)} and such that the couple $(W,V)$ is jointly $L^2-$monotone in the sense that $\forall (X,Y)\in (L^2(\Omega,\reels^d))^2, (\noise,\noiseb)\in (\reels^n)^2$
\[\esp{(W(t,X,\theta,\mathcal{L}(X))-W(t,Y,\noiseb,\mathcal{L}(Y)))\cdot(X-Y)+(V(t,X,\noise,\mathcal{L}(X))-V(t,Y,\noiseb,\mathcal{L}(Y)))\cdot(\noise-\noiseb)}\geq 0.\]
Whenever such $V$ can be constructed, we expect some regularity can be obtained on solutions by working as we did in the autonomous case but on the couple $(W,V)$ instead of just $W$. This gives a different interpretation to the completion argument we presented in the previous section. We first focus on the simplest possible choice of such function $V\equiv Ap$ for some positive matrix $A$, as we did in section \ref{section LL monotone solutions}. Expending on the idea we presented on the autonomous case, we define
\begin{gather}
\nonumber Z^A_\beta(t,\gamma,\noise,\noiseb)=\\
\nonumber
\frac{1}{2}(\noise-\noiseb)\cdot A\cdot (\noise-\noiseb)
\nonumber +\int_{\reels^{2d}}\left((W(t,u,\noise,\pi_d\gamma)-W(t,v,\noiseb,\pi_{-d}\gamma))\cdot(u-v)\right)\gamma(du,dv)\\
\label{def: Z beta b(p,m)}
-\beta(t)\int_{\reels^{2d}}|W(t,u,\noise,\pi_d\gamma)-W(t,v,\noiseb,\pi_{-d}\gamma)|^2 \gamma(du,dv)
 \end{gather}
We observe that this new function also satisfy an equation in the viscosity sense.
\begin{lemma}
\label{lemma: Z supersol b(p,m) W2}
Consider $W$ a Lipschitz solution on $[0,T)$ for some $T>0$ and  let $Z^A_\beta$ be defined as in \eqref{def: Z beta b(p,m)} for some $A\in \mathcal{S}_n(\reels)$ and a $C^1$ function of time $\beta$. Then $Z^A_\beta$ is a viscosity supersolution of 
\begin{gather}
\label{eq: Z supersol b(p,m) W2}
\nonumber \partial_t Z^A_\beta(t,\noise,\noiseb,m)+b^{\noise}\cdot \nabla_\noise Z^A_\beta+b^{\noiseb}\cdot \nabla_{\noiseb} Z^A_\beta-\sigma_\noise \text{Tr}\left(B_nD^2_{(\noise,\noiseb)} Z^A_\beta\right)\\
\nonumber+\int_{\reels^{2d}}\left(\begin{array}{c}
     F^{u}  \\
     F^{v}
\end{array}
\right)
\cdot D_m Z^A_\beta(t,\noise,\noiseb,m,u,v) m(du,dv)-\sigma_x\int_{\reels^{2d}}\text{Tr}\left(B_d D_{(u,v)} D_mZ^A_\beta(t,\noise,\noiseb,m,u,v)\right)m(du,dv)\\
=\int_{\reels^{2d}} (F^{u}-F^{v})\cdot(W^{u}-W^{v})m(du,dv)+\int_{\reels^{2d}}(G^{u}-G^{v})\cdot (u-v)m(du,dv)\\
\nonumber +A(b^\noise-b^{\noiseb} )\cdot (\noise-\noiseb)\\
\nonumber-2\beta(t) \int_{\reels^{2d}} ( G^{u}-G^{v})\cdot(W^{u}-W^{v}) m(du,dv)-\frac{d\beta}{dt}\int_{\reels^{2d}} |W^{u}-W^{v}|^2m(du,dv)\\
\nonumber\text{ for }(t,\noise,\noiseb,m)\in (0,T)\times\reels^{2m}\times\mathcal{P}_2(\reels^{2d}).
\end{gather}
with $W^{u}=W(t,u,\noise,\pi_d m)$, $W^{v}=W(t,v,\noiseb,\pi_{-d} m)$, $F^{u}=F(u,\noise,\pi_d m,W^{u})$, $G^{v}=G(v,\noiseb,\pi_{-d}m,W^{v})$, $b^\noise=b[W](t,\noise,\pi_d m)$ and so on.
\end{lemma}
\begin{proof}
Because the proof is very similar to the one of Lemma \ref{lemma: Z supersol b(p) W2}, we focus on the new part only, the rest following from previously introduced arguments. For $t<T$, we define
\[
\left\{
\begin{array}{c}
d\noise_s=-b[W](t-s,\noise_s,\mu_s)ds+\sqrt{2\sigma_\noise}dB^\noise_s \quad \noise_{t=0}=\noise_0,\\
d\noiseb_s=-b[W](t-s,\noiseb_s,\nu_s)ds+\sqrt{2\sigma_\noise}dB^\noise_s \quad \noiseb_{t=0}=\noiseb_0,\\
dX_s=-F(X_s,\noise_s,\mu_s,W(t-s,X_s,\noise_s,\mu_s))ds+\sqrt{2\sigma_x} dB_s \quad X_{t=0}=x_0,\\
dY_s=-F(Y_s,\noiseb_s,\nu_s,W(t-s,Y_s,\noiseb_s,\nu_s))ds+\sqrt{2\sigma_x} dB_s \quad Y_{t=0}=y_0,\\
     dm_s=(-\text{div}\left(\left(\begin{array}{c}F(x,\noise_s,\mu_s,W(t-s,x,\noise_s,\mu_s))\\ F(y,\noiseb_s,\nu_s,W(t-s,y,\noiseb_s,\nu_s)) \end{array}\right)m_s\right)+\sigma_x \text{Tr}\left(BD^2_{(x,y)}m_s\right))ds \quad m_{t=0}=m_0.\\
\end{array}
\right.
\]
We still use the same notation $W_s^x=W(t-s,\noise_s,X_s,\mu_s)$, however now $W_s^y=W(t-s,\noiseb_s,Y_s,\nu_s)$ and $b^x_s=b(\noise_s,\mu_s,W(t-s,\cdot,\noise_s,\mu_s))$ even though it does not depend on $x$. We have to introduce a doubling of variable in $\noise$ as we did in the measure variable to account for the fact that $b$ is not autonomous anymore. 
By Ito's lemma, for any $s\leq t$
\[\frac{1}{2}(\noise_0-\noiseb_0)\cdot A\cdot (\noise_0-\noiseb_0)=\frac{1}{2}(\noise_{t-s}-\noiseb_{t-s})\cdot A\cdot (\noise_{t-s}-\noiseb_{t-s})+\int_0^{t-s}A(b^x_u-b^y_u)\cdot (\noise_{u}-\noiseb_{u})du,\]
as $(\theta_s)_{s\leq t}$,$(\noiseb_s)_{s\leq t}$ have been generated with the same Brownian motion. In this way, we treated the new term associated to $A$. There is no particular difficulty in extended the computations of Lemma \ref{lemma: Z supersol b(p) W2} to the two others terms involving $W$. When integrating with respect to $m_0$ there is no further difficulty as $(\noise_s,\noiseb_s)_{s\in[0,t]}$ does not depend on the initial conditions $(x_0,y_0)$ of $(X_s,Y_s)_{s\in[0,t]}$.
\end{proof}
To propagate the non-negativity of $Z^A_\beta$, we are going to need some monotonicity assumptions not just on $(G,F)$ but rather on $(G,F,Ab)$. At least formally, if we can find a $A$ for which furthermore $Z^A_\beta|_{t=0}\geq 0$, then we expect that the long time existence of a Lipschitz solution can be proved. This is exactly the monotonicity we ask in the following Hypothesis,
\begin{hyp}
\label{hyp: (F,G,W_0) b(p,m) W2}
There exists $A\in\mathcal{S}_n(\reels)$ and an $\alpha>0$, such that for any $(\noise,\noiseb,X,Y) \in (\reels^{n})^2\times (L^2(\Omega,\reels^{d}))^4$ and $f,g\in C_b(\reels^d,\reels^d)$
\begin{gather}
         \nonumber \esp{\frac{1}{2}(\noise-\noiseb)\cdot A\cdot (\noise-\noiseb)+\left(W_0(X,\noise,\mathcal{L}(X)-W_0(Y,\noiseb,\mathcal{L}(Y))\right)\cdot(X-Y)}\\
         \label{Hyp: W0 b(p,m) W2}\geq \alpha \left(\|W_0(X,\noise,\mathcal{L}(X))-W_0(Y,\noiseb,\mathcal{L}(Y))\|^2_{L^2}\right),
\end{gather}
\begin{gather}
    \nonumber \esp{(F(X,\noise,\mathcal{L}(X),f(X))-F(Y,\noiseb,\mathcal{L}(Y),g(Y))\cdot(f(X)-g(Y))}\\
    \label{hyp: (F,G) b(p,m) W2}+\esp{(G(X,\noise,\mathcal{L}(X),f(X))-G(Y,\noiseb,\mathcal{L}(Y),g(Y)))\cdot(X-Y)}\\
    \nonumber +(b(\noise,\mathcal{L}(X),f(\cdot))-b(\noiseb,\mathcal{L}(Y),g(\cdot))\cdot A\cdot  (\noise-\noiseb)\\
   \nonumber \geq  \alpha \underset{t\in[0,1]}{\min}\|G(X,\noise,\mathcal{L}(X),tf(X)+(1-t)g(Y)-G(Y,\noiseb,\mathcal{L}(Y),tf(X)+(1-t)g(Y))\|^2_{L^2}.
\end{gather}
\end{hyp}
\begin{remarque}
The second part of this assumption is stated slightly differently than its counterpart in Hypothesis \ref{hyp: (G,F,w0) b(p) W2}. Essentially, it is still the same idea, except we now require the random variables corresponding to the non-linearity to be functions of $X$ and $Y$ to account for the functional dependency of $b$. As in Hypothesis \ref{hyp: separated Hamiltonian +LL monotonicity b(p,m)} this joint monotonicity assumption is in general not going to be satisfied unless each of the coefficients is in some sense strongly monotone.
 \end{remarque}
 \begin{exemple}
    Let\[G(x,\theta,\mu,p)=\tilde{G}(x,\theta),\] for a function $\tilde{G}$ such that $\frac{1}{2}\left(D_x \tilde{G}+D_x\tilde{G}^T\right)\geq \alpha_G I_d$ and \[F(x,\theta,\mu,p)=\alpha_F p,\] for some $\alpha_F>0$ (which correspond to a Hamiltonian with a purely quadratic non-linearity). Take 
    \[b(\theta,\mu,f(\cdot)=\tilde{b}(\theta,f_\#\mu),\] for a function $\tilde{b}$ satisfying $\frac{1}{2}(D_\theta b+D_\theta b^T)\geq \alpha_b I_n$ and Lipschitz in $\mathcal{W}_2$ in the second variable. If
    \begin{equation}
    \label{eq: example joint monotonocity}
    4 \alpha_G\alpha_F (\alpha_b)^2> \|D_\theta G\|_\infty^2\|b\|_{Lip(\mathcal{W}_2)}^2,\end{equation}
    then there exists an $a>0$ and a matrix $A=aI_n$ such that \eqref{hyp: (F,G) b(p,m) W2} holds. 
 \end{exemple}
Indeed, we are looking for a constant $a>0$ such that the matrix 
\[
\left(\begin{array}{ccc}
    \alpha_G-\varepsilon & 0 & -\frac{1}{2}\|D_\theta G\|_\infty  \\
      0& \alpha_F & -\frac{a}{2}\|b\|_{Lip(\mathcal{W}_2)}\\
      -\frac{1}{2}\|D_\theta G\|_\infty & -\frac{a}{2}\|b\|_{Lip(\mathcal{W}_2)} &a\alpha_b-\varepsilon
\end{array}\right)\geq 0,\]
for some $\varepsilon>0$.
For such $a$ to exists, a sufficient condition is to ask for both matrix
\[
\left(\begin{array}{cc}
    \alpha_G& -\frac{1}{2}\|D_\theta G\|_\infty  \\
      -\frac{1}{2}\|D_\theta G\|_\infty &\frac{1}{2}a\alpha_b
\end{array}\right), 
\left(\begin{array}{ccc}
      \alpha_F & -\frac{a}{2}\|b\|_{Lip(\mathcal{W}_2)}\\
-\frac{a}{2}\|b\|_{Lip(\mathcal{W}_2)} &\frac{1}{2}a\alpha_b
\end{array}\right),\]
to be definite positive. This is true as soon as 
\[\frac{2\alpha_F\alpha_b}{\|b\|^2_{Lip(\mathcal{W}_2)}}>a>\frac{\|D_\theta G\|_\infty^2}{2\alpha_G\alpha_b}.\]
Let us insist on the fact that this not a smallness assumption on $b$. Replacing $b$ with $\eta b$ for some $\eta << 1$ yields the same condition \eqref{eq: example joint monotonocity} as $\eta$ cancels out in this equation.

 \begin{lemma}
 \label{Z b(p,m) non negative}
     Let $W$ be a Lipschitz solution to \eqref{eq: ME W2 Lip} on $[0,T)$. Under Hypothesis \ref{hyp: Lipschitz W2}, if there exists $A\in \mathcal{S}_n(\reels)$ such that Hypothesis \ref{hyp: (F,G,W_0) b(p,m) W2} holds then there exists a strictly positive function of time $\beta:[0,T)\to\reels^{+*}$ such that 
     \[\forall t<T, (\noise,\noiseb,\gamma)\in (\reels^{n})^2\times \mathcal{P}_2(\reels^{2d}) \quad Z^A_\beta(t,\noise,\noiseb,\gamma)\geq 0,\]
     for $Z^A_\beta$ as defined in \eqref{def: Z beta b(p,m)}.
 \end{lemma}
\begin{proof}
    Observe that we still have for any $\lambda \in[0,1]$
 \begin{gather*}
 \int_{\reels^{2d}} \langle G^{u}-G^{v},W^{u}-W^{v}\rangle m(du,dv)\\
 \leq\frac{1}{2} \int_{\reels^{2d}} |G(u,\pi_d m,\noise,\lambda W^{u}+(1-\lambda )W^{v})-G(v,\pi_{-d},\noiseb,\lambda W^{u}+(1-\lambda)W^{v})|^2m(du,dv)\\
 +(2\|G\|_{Lip}+\frac{1}{2})\int_{\reels^{2d}}|W^{u}-W^{v}|^2m(du,dv).
 \end{gather*}
As a consequence, under Hypothesis \ref{hyp: (F,G,W_0) b(p,m) W2}, for $\beta(t)=\beta_0e^{-\gamma t}$ with $\beta_0$ sufficiently small and $\gamma$ sufficiently big, $Z^A_\beta$ is a viscosity supersolution of 
 \begin{gather*}
 \nonumber \partial_t Z^A_\beta(t,\noise,\noiseb,m)+b^{u}\cdot \nabla_\noise Z^A_\beta+b^{v}\cdot \nabla_{\noiseb} Z^A_\beta-\sigma_\noise \text{Tr}\left(B_mD^2_{(\noise,\noiseb)} Z^A_\beta\right)\\
\nonumber+\int_{\reels^{2d}}\left(\begin{array}{c}
     F^{u}  \\
     F^{v}
\end{array}
\right)
\cdot D_m Z^A_\beta(t,\noise,\noiseb,m,u,v) m(du,dv)-\sigma_x\int_{\reels^{2d}}\text{Tr}\left(B_d D_{(u,v)} D_mZ^A_\beta(t,\noise,\noiseb,m,u,v)\right)m(du,dv)\geq 0,
 \end{gather*}
 satisfying $Z^A_\beta|_{t=0}\geq 0$ on $(0,T)$. By a comparison principle with $0$ as conducted in Lemma \ref{Z beta positive b(p)}, we deduce that 
 \[\forall t<T\quad  \forall (\noise,\noiseb,m)\in\reels^{2m}\times \mathcal{P}_2(\reels^{2d}) \quad Z^A_\beta(t,\noise,\noiseb,m)\geq 0.\]
\end{proof}
It now remains to show that this monotonicity estimate is sufficient to get Lipschitz estimates on solutions of \eqref{eq: ME W2 Lip}.
 \begin{thm}
 \label{thm: existence b(p,m) W2}
     Under Hypothesis \ref{hyp: Lipschitz W2} and Hypothesis \ref{hyp: (F,G,W_0) b(p,m) W2} there exists a unique Lipschitz solution $W$ to \eqref{eq: ME W2 Lip} on $[0,+\infty)$.
 \end{thm}
 \begin{proof}
Local existence of a Lipschitz solution $W$ on a time interval $[0,T_c)$ is a direct consequence of Hypothesis \ref{hyp: Lipschitz W2}. We now show that $W$ does not blow up in finite time.

 \noindent\textit{Step 1: Lipschitz estimate in $L^2$}\\
 From Lemma \ref{Z b(p,m) non negative}, we now that there exists a function $\beta:\reels^+\to \reels^{+*}$ such that $Z^A_\beta$ is positive for any $t<T_c$. For $\noise=\noiseb$ this gives that for any $t<T_c$, $(\noise,X,Y)\in \reels^n\times (L^2(\Omega,\reels^d))^2$
 \begin{gather*}
 \esp{(W(t,X,\noise,\mathcal{L}(X))-W(t,Y,\noise,\mathcal{L}(Y)))\cdot (X-Y)}\geq \beta(t)\|W(t,X,\noise,\mathcal{L}(X))-W(t,Y,\noise,\mathcal{L}(Y))\|_{L^2}.
 \end{gather*}
 By Lemma \ref{lemma: estimate in x b(p) W2}, this implies that 
 \[\|W(t,\cdot)\|_{Lip(x)}\leq \frac{1}{\beta(t)}.\]
It only remains to get a Lipschitz estimate in the noise variable $\noise$ and the measure argument.  Let us first introduce an alternate Lipschitz semi-norm in $L^2$, for a function $L^2(\Omega,\reels^d)\times\reels^n\ni(X,\noise)\mapsto U(X,\noise)$ we denote 
 \[\|U\|_{Lip(L^2)}:=\underset{(X,Y,\noise,\noiseb)}{\sup}\frac{\|U(X,\noise)-U(Y,\noiseb)\|_{L^2}}{|\noise-\noiseb|+\|X-Y\|_{L^2}}.\]
We also denote  
\[\|U\|_{Lip(L^2),x}:=\underset{(X,Y,\noise)}{\sup}\frac{\|U(X,\noise)-U(Y,\noise)\|_{L^2}}{\|X-Y\|_{L^2}},\]
and similarly we define $\|U\|_{Lip(L^2),\theta}$ by  
\[\|U\|_{Lip(L^2),\theta}=\underset{(X,\noise,\noiseb)}{\sup}\frac{\|U(X,\noise)-U(X,\noiseb)\|_{L^2}}{|\noise-\noiseb|}\]
We are now going to show that $W$ is Lipschitz with respect to this $L^2$ norm\footnote{Since this is the semi-norm of the lift of $W$, $(X,\noise)\mapsto W(t,X,\theta,\mathcal{L}(X))$, it is weaker than the usual Lipschitz semi-norm. See for exemple \[W(X,\theta,\mathcal{L}(X))=\frac{X-\esp{X}}{1+\sqrt{\esp{(X-\esp{X})^2}}} \theta,\] seen as a fonction $(x,\theta,\mu)\mapsto U(x,\theta,\mu)$ this function is only locally Lipschitz but $\|W\|_{Lip(L^2),\theta}=1$  }. Remember that the non negativity of $Z^A_\beta$ is equivalent to
\begin{gather}
\label{Z b beta L2} \forall (t,X,Y,\noise,\noiseb)\in (0,T_c)\times L^2(\Omega,\reels^{2d})\times \reels^{2m}\\
         \nonumber \frac{1}{2}(\noise-\noiseb)\cdot A\cdot (\noise-\noiseb)+\esp{\left(W(t,X,\noise,\mathcal{L}(X)-W(t,Y,\noiseb,\mathcal{L}(Y))\right)\cdot(X-Y)}\\
         \nonumber\geq \beta(t)\|W(t,X,\noise,\mathcal{L}(X))-W(t,Y,\noiseb,\mathcal{L}(Y))\|^2_{L^2}.
\end{gather}
Taking this expression for $\noise=\noiseb$ and $X=Y$ gives
\[\|W(t,\cdot)\|_{Lip(L^2),\noise},\|W(t,\cdot)\|_{Lip(L^2),x}\leq \frac{1+\|A\|}{\beta(t)}.\]
\noindent\textit{Step 2: Lipschitz estimate in $\reels^n\times \mathcal{P}_2(\reels^d)$}\\
It is a priori not trivial that the Lipschitz estimates in $L^2$ we just got translates into Lipschitz estimates for the norms of interest. To get estimates for the Euclidean and Wasserstein distance we proceed as we did in Theorem \ref{thm: b(p) W2}: by proving that this $L^2$ Lipschitz norm is sufficient to get Lipschitz continuity in $\mathcal{W}_2$ with respect to initial condition for the following SDE-SPDE system
\[\left\{\begin{array}{l}
d\mu^{\mu_0,\noise}_s=\left(-\text{div}\left(F(x,\noise_s,\mu_s^{\mu_0,\noise},W(t-s,x,\noise_s,\mu_s^{\mu_0,\noise}))\right)+\sigma_x\Delta_x \mu^{\mu_0,\noise}_s\right)ds \quad \mu_0\in\mathcal{P}_2(\reels^d),\\
d\noise^{\mu_0,\noise}_s=-b[W](t-s,\noise^{\mu_0,\noise}_s,\mu_s^{\mu_0,\noise})ds+\sqrt{2\sigma_x}dB^\noise_s \quad \noise_0=\noise\in\reels^n.\\
\end{array}\right.\]
There is no difficulty in extending the previous result by looking at the quantity 
\[\esp{\mathcal{W}_2^2(\mu^{\mu_0,\noise_0},\mu^{\nu_0,\noiseb_0})+|\noise_s^{\mu_0,\noise_0}-\noiseb_s^{\nu_0,\noiseb_0}|^2},\]
and applying Grönwall's Lemma. It is however important to remark that this is where the dynamics we imposed on $b$ in Hypothesis \ref{hyp: Lipschitz W2} comes into play. Indeed, this proof only works because 
\begin{gather*}
|b[W](t-s,\noise^{X_0,\noise_0}_s,\mu_s^{X_0,\noise_0})-b[W](t-s,\noise^{Y_0,\noiseb_0}_s,\mu_s^{Y_0,\noiseb_0})|\\
\leq C\left(|\noise^{X_0,\noise_0}-\noise^{Y_0,\noiseb_0}|+\|X^{X_0,\noise_0}-X^{Y_0,\noiseb_0}\|_{L^2}+\|W(t-s,X^{X_0,\noise_0},\noise^{X_0,\noise_0})-W(t-s,X^{Y_0,\noiseb_0},\noise^{Y_0,\noiseb_0})\|_{L^2}\right).
\end{gather*}
Once the Lipschitz continuity of this system with respect to initial conditions is established, an estimate on $\|W\|_{Lip(\mathcal{W}_2)}$ and $\|W\|_{Lip,\theta}$ is simply given by Grönwall's Lemma. 

\noindent\textit{Step 3: Conclusion}\\
Since we now have Lipschitz estimates on $W$ on any time interval, we get by a proof by contradiction that $T_c=+\infty$ must hold. As a consequence, there is indeed existence of a unique Lipschitz solution on $[0,+\infty)$.

 \end{proof}
 \subsubsection{Extensions and applications}
 We now come back on a possible application of those results outside of mean field games: the solvability of forward-backward mean field systems. Whenever $W$ is smooth the fact that it is a decoupling field for one such system is a direct consequence of Ito's lemma. We claim that it is still true for Lipschitz solutions. In particular, the existence of a Lipschitz solution to \eqref{eq: ME W2 Lip} gives an existence result for the associated FBSDE
\begin{lemma}
\label{lemma: lip sol FBSDE}
    Suppose there exists a Lipschitz solution $W$ to \eqref{eq: ME W2 Lip} on $[0,T]$ and that $b$ is of the form 
    \[b(\noise,\mathcal{L}(X),f(\cdot))=\tilde{b}(\noise,\mathcal{L}(X,f(X))).\]
    We define 
\begin{equation*}
\left\{
\begin{array}{l}
   \displaystyle X_t=x_0-\int_0^t F(X_s,\noise_s,\mathcal{L}(X_s|(\noise_u)_{u\leq s}),W(T-s,X_s,\noise_s,\mathcal{L}(X_s|(\noise_u)_{u\leq s})))ds+\sqrt{2\sigma_x}B_t,\\
   \displaystyle \noise_t=\noise_0-\int_0^t b(\noise_s,\mathcal{L}(X_s|(\noise_u)_{u\leq s}),W(T-s,\cdot,\noise_s,\mathcal{L}(X_s|(\noise_u)_{u\leq s})))ds+\sqrt{2\sigma_\noise}B^\noise_t, \\
\end{array}
\right.
\end{equation*}
for $(B_s,B^\noise_s)_{s\geq 0}$ a Brownian motion on $\reels^{d+n}$, $(x_0,\noise_0)\in\reels^{d+n}$.
Letting $W_t=W(T-t,X_t,\noise_t)$ and $\mathcal{F}$ be the natural filtration associated to $(B_s,B^\noise_s)_{s\geq 0}$, there exists a square integrable $\mathcal{F}$-predictable process $(Z_s)_{s\in[0,T]}$ such that $\forall t\in[0,T]$
    \begin{equation}
    \label{eq: mean field FBSDE}
\left\{
\begin{array}{l}
   \displaystyle X_t=x_0-\int_0^t F(X_s,\noise_s,\mathcal{L}(X_s|(\noise_u)_{u\leq s}),W_s)ds+\sqrt{2\sigma_x}B_t,\\
   \displaystyle \noise_t=\noise_0-\int_0^t \tilde{b}(\noise_s,\mathcal{L}(X_s,W_s)|(\noise_u)_{u\leq s})ds+\sqrt{2\sigma_\noise}B^\noise_t, \\
   \displaystyle W_t=W_0(X_T,\noise_T,\mathcal{L}(X_T|(\noise_u)_{u\leq T}))+\int_t^T G(X_s,\noise_s,\mathcal{L}(X_s|(\noise_u)_{u\leq s}),W_s)ds-\int_t^T Z_s\cdot d(B_s,B^\noise_s),
\end{array}
\right.
\end{equation}
\end{lemma}
\begin{proof}
    Since $W$ is Lipschitz on $[0,T]$, $(X_s,\noise_s)_{s\in[0,T]}$ is well-defined (see \cite{Carmona-major-minor}). Let us also recall that following the proof of Proposition \ref{prop: representation of lipschitz solutions} the stochastic process $(M_s)_{s\in[0,T]}$ defined by 
    \[M_t= W(T-t,X_t,\noise_t,\mathcal{L}(X_t|(\noise_u)_{u\leq t}))+\int_0^t G(X_s,\noise_s,\mathcal{L}(X_s|(\noise_u)_{u\leq s}),W(t-s,X_s,\noise_s,\mathcal{L}(X_s|(\noise_u)_{u\leq s})))ds,\]
    is a $\mathcal{F}-$martingale. Grönwall's Lemma yields estimates on the second moments of $(X_s,\noise_s)_{s\in[0,T]}$ and so $(M_s)_{s\in[0,T]}$ is a square integrable martingale. By the martingale representation theorem, there exists $(Z_s)_{s\in[0,T]}$, a square integrable, predictable, $\mathcal{F}-$adapted process such that 
    \[\forall t\leq T\quad  M_t=M_0+\int_0^t Z_s\cdot d(B_s,W_s).\]
    Using the equality 
    \[M_T=M_t+\int_t^TZ_sd(B_s,W_s),\]
    for this process $(Z_s)_{s\in[0,T]}$, we get exactly 
    \[W_t=W_0(X_T,\noise_T,\mathcal{L}(X_T|(\noise_u)_{u\leq T}))+\int_t^T G(X_s,\noise_s,\mathcal{L}(X_s|(\noise_u)_{u\leq s}),W_s)ds-\int_t^T Z_sd(B_s,W_s),\]
    which ends the proof. 
\end{proof}
\begin{remarque}
Local solvability of forward-backward stochastic systems with Lipschitz coefficients is already well known, Lipschitz solutions appear as another way to prove this result by showing directly the existence of a Lipschitz decoupling field. In the particular case of mean field games, the stochastic forward backward differential system \eqref{eq: mean field FBSDE} corresponds to the optimality conditions given by Pontryagin maximum principle.
\end{remarque}
 Many extensions and adaptations of Theorem \ref{thm: existence b(p,m) W2} are also possible. In particular, strong monotonicity in the state variable in \eqref{hyp: (F,G) b(p,m) W2} can always be traded for strong monotonicity in $W$ as in \cite{Lions-college,noise-add-variable,disp-monotone-2}. We focus rather on the choice of a function $V$ to complete \eqref{eq: ME W2 Lip}. In the above section, we explored the most simple case of a linear function of $\theta$. We now present how this idea can be extended to more general functions, leading to different monotonicity conditions. A simple general choice is to take 
 \[V(t,x,\noise,\mu)=A_\noise \noise+A_W W(t,x,\noise,\mu)+A_x x,\]
 for some matrix $(A_x,A_W)\in \left(\mathcal{M}_{d\times n}(\reels)\right)^2, A_\theta \in \mathcal{M}_n(\reels)$. In which case $V$ satisfy \eqref{def: V for b(m,p)} (in the sense of Lipschitz solution) for 
 \[G_V\equiv A_x F+A_W G+A_\noise b.\]
 
Since both $G_V$ and $V$ are explicit, Hypothesis \ref{hyp: (G,F,w0) b(p) W2} can be checked easily for the couple $(W,V)$.  A notable extension being the situation 
\[V=A_W W,\]
which corresponds, up to rescaling of $W$ by a $d\times d$ matrix $M$ (which is always possible as the general form of the equation is conserved by this linear shift), to $G-$monotonicity for mean field forward backward differential equations \cite{G-monotonicity,mean-field_G-monotonicity}.  
\begin{remarque}
    Different choice of functions $V$ are also possible in the results we presented in Section \ref{section LL monotone solutions}. The proof can be adapted to consider a more general function $V:\reels^+\times\reels^n\times \mathcal{P}_1(\tor^d)\to \reels^n$ instead of $\noise\mapsto A\noise$ for $V$ Lipschitz solution to  
\begin{gather*}\partial_t V(t,m,\noise)+b\cdot\nabla_\noise V-\sigma_\noise \Delta_\noise V\\
+\int_{\tor^d} \nabla_pH\cdot D_mV m(dy)-\sigma_x \int_{\tor^d}\text{div}_y\left( D_mV\right)m(dy)=R(\noise,m,\psi,U),
\end{gather*}
provided a monotonicity estimate \textit{à la} \eqref{LL monotinicity + in p U0} is satisfied for $(U_0,V)$ and $(b,R,H,f)$. 
\end{remarque}

\section{Presence of additive common noise}
 \label{section common noise}
 In this section we make the link between common noise through a noise process impacting the coefficients of the game as presented in the master equation \eqref{eq: general ME} and the more usual additive common noise that has already been widely studied. In particular, this will help us show how every result presented in this paper extend to the presence of a Brownian additive common noise $(B^c_t)_{t\geq 0}$ even when it is correlated with the Brownian motion driving the noise process $(B^\noise_t)_{t\geq 0}$.

 \subsection{From additive common noise to noise as an additional variable }
 Consider the following master equation associated to a mean field game with common noise 
 \begin{equation}
 \label{eq:ME with common noise}
\left\{
\begin{array}{c}
\displaystyle -\partial_t \mathcal{U}-(\sigma_x+\beta)\Delta_x \mathcal{U}+H(x,m,D_x\mathcal{U})\\
\displaystyle-(\sigma_x+\beta) \int \text{div}_{y}[D_m\mathcal{U}](t,x,m,y)]m(dy)+\int D_m \mathcal{U}\cdot D_p Hdm\\
\displaystyle -2\beta \int \text{div}_{x}[D_m\mathcal{U}](t,x,m,y)]m(dy)-\beta \int \text{Tr}\left[D^2_{mm}U(t,x,m,y,y')\right]m(dy)m(dy')\\
\text{ for } (t,x,m) \in (0,T)\times \reels^d\times  \mathcal{P}(\reels^d),\\
\text{ with }\mathcal{U}(T,x,m)=\mathcal{U}_0(x,m),
\end{array}
\right.
\end{equation}
Due to the presence of additive common noise, a second order term in the measure argument appears in the equation. Strong assumptions on the data are usually necessary to be able to define a solution to the problem twice differentiable in the measure argument \cite{convergence-problem}. However, we claim that by adding another variable to the equation it is possible to rewrite this equation without second order derivatives in the measure argument. For that, let us first come back to the game representation of a mean field game associated to \eqref{eq:ME with common noise}. Given the distribution of players at all times $(\mu_t)_{t\in[0,T]}$, each player solves the following optimal control problem 
\begin{gather*}
\underset{\alpha}{\sup} \quad \esp{\mathcal{U}_0(X^\alpha_T,\mu_T)-\int_0^T L(X_s^\alpha,\alpha_s,\mu_s)ds},\\
dX_t^\alpha=-h(X_t^\alpha,\alpha_t,\mu_t) dt+\sqrt{2\sigma_x} dW_t+\sqrt{2\beta} dB^c_t,
\end{gather*}
for $(W_s)_{s\geq 0}$ a Brownian motion independent of $(B^c_s)_{s\geq 0}$. By analogy with the previous sections let us define
\[\noise_t=\sqrt{2\beta} B^c_t.\]
Introducing $Y_t^\alpha=X^\alpha_t-\noise_t$, for any control $\alpha$ we observe that
\[
\begin{array}{c}
     dY^\alpha_t=-h(Y^\alpha_t+\noise_t,\alpha_t,\mu_t)dt+\sqrt{2\sigma_x} dW_t , \\
     d\noise_t=\sqrt{2\beta} dB^c_t.\\
\end{array}
\]
This simple computation allows to rewrite the optimal control problem in function of the controlled process $(Y_t)_{t\in[0,T]}$ and the noise process $(\noise_t)_{t\in[0,T]}$. By definition of $(\noise_t)_{t\in[0,T]}$ it is progressively measurable with respect to the filtration associated to the common noise, as such for any control $\alpha$
\[\mathcal{L}\left(X^\alpha_t|(B^c_s)_{s\in[0,t]}\right)=(id_{\reels^d}+\noise_t)_\#\mathcal{L}\left(Y^\alpha_t|(B^c_s)_{s\in[0,t]}\right).\]
Using the characterisation of the mean field equilibrium, this heuristic hints at the idea of studying 
\[V(t,y,\noise,m)=\mathcal{U}(t,y+\noise,(id_{\reels^d}+\noise)_\#m).\]
Indeed, we observe that if $\mathcal{U}$ is a smooth solution of \eqref{eq:ME with common noise} then $V$ solves 
\begin{equation}
\label{eq: ME common noise transform}
\left\{
\begin{array}{c}
\displaystyle -\partial_t V-\sigma_x\Delta_y V-\beta \Delta_\noise V+H(y+\noise,(id_{\reels^d}+\noise)_\#m,D_yV)\\
\displaystyle-\sigma_x \int \text{div}_{z}[D_mV](t,y,\noise,m,z)]m(dz)+\int D_m V(t,y,\noise,m,z)\cdot D_p H(z+\noise,(id_{\reels^d}+\noise)_\#m,\nabla_y V)m(dz)\\
\text{ for } (t,y,\noise,m) \in (0,T)\times \reels^d\times \reels^d\times \mathcal{P}(\reels^d),\\
\text{ with }V(T,y,\noise,m)=\mathcal{U}_0(y+\noise,(id_{\reels^d}+\noise)_\#m),
\end{array}
\right.
\end{equation}
which is exactly an equation of the same form as \eqref{eq: general ME}, with the new hamiltonian 
\[\tilde{H}(y,\noise,m,p)=H(y+\noise,(id_{\reels^d}+\noise)_\#m,p).\]
\begin{remarque}
Observe that monotone functions are left monotone through this transformation. Indeed, assume that $U:\reels^d\times\mathcal{P}_1(\reels^d)\to \reels$ is flat monotone, then for any $(\mu,\nu)\in \left(\mathcal{P}_1(\reels^d)\right)^2$
\begin{align*}
   &\langle U(\cdot+\theta,(id_{\reels^d}+\noise)_\#\mu)-U(\cdot+\theta,(id_{\reels^d}+\noise)_\#\nu),\mu-\nu\rangle\\
   =&\langle U(\cdot,(id_{\reels^d}+\noise)_\#\mu)-U(\cdot,(id_{\reels^d}+\noise)_\#\nu),(id_{\reels^d}+\noise)_\#\mu-(id_{\reels^d}+\noise)_\#\nu\rangle\geq 0.
\end{align*}
A simple calculation shows that this is also true for $L^2-$monotone functions. This is not the first time transformations of this kind are used in mean field games, in fact this has become a somewhat classic argument in the study of the stochastic forward backward mean field game system, see \cite{convergence-problem}) in which it was observed that a similar transformation leads to a system with a forward stochastic (Fokker–Planck) equation with finite variation. 
\end{remarque}
We believe such a transformation is particularly suited for the study of weak solutions to the master equation, in particular those based on comparison argument as in \cite{bertucci-monotone} or this work, as this bypass the need to tackle second order derivatives on the space of measures. Even for smooth solution, it is relevant.  Indeed, assuming that 
\[\noise\mapsto f((id_{\reels^d}+\noise)_\#m),\]
is a $C^2$ function is weaker than the existence of a second derivative in the measure argument, in this way we may define solutions of mean field games that are smooth in the sense that \eqref{eq: ME common noise transform} is satisfied without needing the existence of a smooth second derivative in the measure argument. For a solution $\mathcal{U}$ that is in fact smooth in the measure argument
\begin{gather*}
\Delta_\noise \mathcal{U}(t,x+\noise,(id_{\reels^d}+\noise)_\#m)|_{\noise=0}\\
=\Delta_x \mathcal{U}(t,y,m)+\int_{\reels^d}\text{Tr}\left[D^2_{mm}\mathcal{U}\right](t,x,m,y,y')m(dy)m(dy')+2\int_{\reels^d}\text{div}_x\left[D_m\mathcal{U}\right](t,x,m,y)m(dy).
\end{gather*}
Obviously if all derivatives are smooth then $\Delta_\noise \mathcal{U}$ is well-defined, the observation we are here making is just that for it to be well-defined we do not need to be able to make sense of each of the integrated term on their own.
\begin{exemple}
Consider the average distance between two players
    \[U(m)=\int_{\reels^{2d}}|y-y'|m(dy)m(dy').\]
For measures in $\mathcal{P}_1(\reels^d)$, $U$ is well-defined and continuous (it is even Lipschitz in $\mathcal{W}_2$). However, it does not admit a continuous Wasserstein derivative $y\mapsto D_m U(m,y)$ at every point (for example for $m=\delta_x$ for some $x\in \reels^d$), and much less a continuous second order derivative with respect to measures. Nevertheless, $\theta\mapsto U((id_{\reels^d}+\theta)_\#m)$ belongs to $C^\infty$ and
\[\Delta_\noise  U((id_{\reels^d}+\noise)_\#m)|_{\theta=0}\equiv 0.\]
\end{exemple}

\subsection{Application to Lipschitz solutions}
 Let us now detail how to adapt this idea to Lipschitz solutions. Firstly, let us mention that Lipschitz solutions are easily adapted to the presence of additive common noise. We refer to the original paper \cite{lipschitz-sol} for the presentation of such results in details. In the context of the nonlinear transport equation
 \begin{equation}
 \label{eq: Lip sol Pq common noise}
\left\{
\begin{array}{c}
\displaystyle\partial_t W+F(x,m,W)\cdot \nabla_x W-(\sigma_x+\beta) \Delta_x W\\
\displaystyle +\int_{\reels^d} F(y,m,W)\cdot D_m W(t,x,m)(y)m(dy)-\sigma_x\int_{\reels^d} \text{div}_y (D_mW(t,x,m)(y))m(dy)\\
\displaystyle -2\beta \int \text{div}_{x}[D_mW ](t,x,m,y)]m(dy)-\beta \int \text{Tr}\left[D^2_{mm}W(t,x,m,y,y')\right]m(dy)m(dy')\\
\displaystyle =G(x,m,W) \text{ in } (0,T)\times\reels^d\times\mathcal{P}_q(\reels^d),\\
\displaystyle W(0,x,m)=W_0(x,m) \text{ for } (x,m)\in\reels^d\times\mathcal{P}_q(\reels^d).
 \end{array}
 \right.
\end{equation}
the representation formula for Lipschitz solution becomes
\begin{gather}
\label{Def: Lipschitz solution common noise}
\nonumber \displaystyle W(t,x,\mu)=\esp{W_0(X_t,m_t)+\int_0^t G(X_s,m_s,W(t-s,X_s,m_s))ds},\\
\displaystyle dX_s=-F(X_s,m_s,W(t-s,X_s,m_s))ds+\sqrt{2\sigma_x}dB_s+\sqrt{2\beta}dB^c_s \quad X_0=x,\\
\nonumber dm_s=\left(-\text{div}\left(F(x,m_s,W(t-s,x,m_s))m_s\right)+(\sigma_x+\beta) \Delta_x m_s\right)ds-\text{div}\left(\sqrt{2\beta}m_sdB^c_s\right) \quad m_0=\mu,
\end{gather}
for two independent Brownian motions $(B_s)_{s\geq 0}$ and $(B^c_s)_{s\geq 0}$. We make the heuristic presented above rigorous directly at the level of the nonlinear transport equation \eqref{eq: Lip sol Pq common noise} which includes \eqref{eq:ME with common noise} as a particular case.
\begin{lemma}
\label{lemma: equiv additive/ add variable}
    Let $W$ be a Lipschitz solution of 
    \begin{equation*}
\left\{
\begin{array}{c}
\displaystyle \partial_t W-\sigma_x\Delta_x W-\beta \Delta_\noise W+F(x+\noise,(id_{\reels^d}+\noise)_\#m,W)\cdot \nabla_x W\\
\displaystyle-\sigma_x \int \text{div}_{y}[D_mW](t,x,\noise,m,y)]m(dy)+\int D_m W(t,x,\noise,m,y)\cdot F(y+\noise,(id_{\reels^d}+\noise)_\#m,W)m(dy)\\
=G(x+\noise,(id_{\reels^d}+\noise)_\#m,W)
\text{ in } (0,T)\times \reels^d\times \reels^d\times \mathcal{P}_q(\reels^d),\\
\text{ with }W(0,x,\noise,m)=W_0(x+\noise,(id_{\reels^d}+\noise)_\#m),
\end{array}
\right.
\end{equation*}
    on $[0,T)$, then
    \begin{enumerate}
        \item[-] $\forall (t,x,m)\in[0,T)\times\reels^d\times \mathcal{P}_q(\reels^d), $
        \[\noise\mapsto W(t,x-\noise,\noise,(id_{\reels^d}-\noise)_\#m)\]
        is a constant function.
        \item[-] $W_c:(t,x,m)\mapsto W(t,x,0,m)$ is a Lipschitz solution of \eqref{eq: Lip sol Pq common noise} on $[0,T)$.
    \end{enumerate}
\end{lemma}
\begin{proof}
The first claim is quite straightforward to prove, as it is just a simple matter of rewriting the system satisfied by $(X_s,\noise_s,m_s)_{s\in[0,t]}$. Indeed, for any $(t,\noise,y,\mu)$,
\[
\begin{array}{c}
\displaystyle W(t,y-\noise,\noise,(id_{\reels^d}-p)_\#\mu)=\esp{W_0(X_t+\noise_t,(id_{\reels^d}+\noise_t)_\#m_t)}\\
 \quad \quad\quad \quad \quad\quad \quad \quad\quad \quad \quad\quad \quad \quad\quad \quad \quad \displaystyle+\esp{\int_0^t G(X_s+\noise_s,(id_{\reels^d}+\noise_s)_\#m_s,W_s)ds},\\
\\
\displaystyle dX_s=-F(X_s+\noise_s,(id_{\reels^d}+\noise_s)_\#m_s,W_s)ds+\sqrt{2\sigma_x}dB_s \quad X_0=y-\noise,\\
\displaystyle d\noise_s=\sqrt{2\beta}dB^\noise_s \quad \noise_0=\noise,\\
dm_s=\left(-\text{div}\left(F(x+\noise_s,(id_{\reels^d}+\noise_s)_\#m_s,W(t-s,x,\noise_s,m_s))m_s\right)+\sigma_x \Delta_x m_s\right)ds \quad m_0=(id_{\reels^d}-\noise)_\#\mu,\\
W_s=W(t-s,X_s,\noise_s,m_s).
\end{array}
\]
By rewriting this system in function of $(X_s+\noise)_{s\geq 0}$ and $(m_s\#(id_{\reels^d}+\noise))_{s\geq 0}$ we get
\[
\begin{array}{c}
\displaystyle W(t,y-\noise,\noise,(id_{\reels^d}-\noise)_\#\mu)=\esp{ W_0(X_t+\noise_t,(id_{\reels^d}+\noise_t)_\#m_t)}\\
 \displaystyle\quad \quad\quad \quad \quad\quad \quad \quad\quad \quad \quad\quad \quad \quad\quad \quad \quad +\esp{\int_0^t G(X_s+\noise_s,(id_{\reels^d}+\noise_s)_\#m_s,W_s)ds},\\
\\
\displaystyle dX_s=-D_p H(X_s+\noise_s,(id_{\reels^d}+\noise_s)_\#m_s,W_s)ds+\sqrt{2\sigma_x}dB_s \quad X_0=y,\\
\displaystyle d\noise_s=\sqrt{2\beta}dB^\noise_s \quad \noise_0=0,\\
dm_s=\left(-\text{div}\left(D_p H(x+\noise_s,(id_{\reels^d}+\noise_s)_\#m_s,W(t-s,x-\noise,\noise_s+\noise,(id_{\reels^d}-\noise)_\#m_s))m_s\right)+\sigma_x \Delta_x m_s\right)ds,\\
m_0=\mu \quad W_s=W(t-s,X_s-\noise,\noise_s+\noise,(id_{\reels^d}-\noise)_\#m_s).
\end{array}
\]
Fixing $\noise$, let us observe that both $(t,y,\bar{\theta},m)\mapsto W(t,y-\noise,\bar{\noise}+\noise,(id_{\reels^d}-\noise)_\#m)$ and $(t,y,\bar{\noise},m)\mapsto W(t,y,\bar{\noise},m)$ are fixed points of this equation. By uniqueness of the fixed point for Lipschitz solution, we necessarily have 
\[\forall t<T, (y,m)\in \reels^d\times \mathcal{P}_q(\reels^d), \quad W(t,y,0,m)=W(t,y-\noise,\noise,(id_{\reels^d}-\noise)_\#m).\]
In particular, this implies that for any $t<T$, $(y,\noise,m)$
\[W(t,y,\noise,m)=W(t,y+\noise,0,(id_{\reels^d}+\noise)_\#m).\]
A result that we may apply directly to rewrite 
\[
\begin{array}{c}
\displaystyle W(t,y,0,\mu)=\esp{W_0(X_t,(id_{\reels^d}+\noise_t)_\#m_t)+\int_0^t G(X_s,(id_{\reels^d}+\noise_s)_\#m_s,W_s)ds},\\
\displaystyle dX_s=-F(X_s,(id_{\reels^d}+\noise_s)_\#m_s,W_s)ds+\sqrt{2\sigma_x}dB_s+\sqrt{2\beta}dB^\noise_s \quad X_0=y,\\
\displaystyle d\noise_s=\sqrt{2\beta}dB^\noise_s \quad \noise_0=0,\\
dm_s=\left(-\text{div}\left(F(x+\noise_s,(id_{\reels^d}+\noise_s)_\#m_s,W(t-s,x+\noise_s,0,(id_{\reels^d}+\noise_s)_\#m_s))m_s\right)+\sigma_x \Delta_x m_s\right)ds,\\
m_0=\mu \quad W_s=W(t-s,X_s,0,(id_{\reels^d}+\noise_s)_\#m_s).
\end{array}
\]
Since it is classical in MFG \cite{convergence-problem,lipschitz-sol} that under this definition,  $\tilde{m}_s=(id_{\reels^d}+\noise_s)_\#m_s$ is a solution of 
\[d\tilde{m}_s=\left(-\text{div}\left(D_p H(x,\tilde{m}_s,W(ft-s,x,\tilde{m}_s))m_s\right)+(\sigma_x+\beta) \Delta_x m_s\right)ds-\text{div}\left(\sqrt{2\beta}m_sdB^\noise_s\right) \quad \tilde{m}_0=\mu,\]
we conclude by uniqueness of Lipschitz solution that 
\[(t,x,m)\to W(t,x,0,m)\]
is indeed a Lipschitz solution of \eqref{eq: Lip sol Pq common noise}.
\end{proof}

We now explain how this Lemma can be used to show the addition of additive common noise in \eqref{eq: general ME} is non-consequential.
\begin{corol}
Under the same Hypothesis as in Theorem \ref{thm: existence b(p,m) W2} there exists a unique Lipschitz solution to the equation 
\begin{equation*}
\left\{
\begin{array}{c}
\displaystyle\partial_t W+F(x,\noise,m,W)\cdot \nabla_x W+b(\noise,m,W(t,\cdot,\noise,m))\cdot\nabla_\noise W-(\sigma_x+\beta) \Delta_x W-\sigma_\noise \Delta_\noise W\\
\displaystyle +\int_{\reels^d} F(y,\noise,m,W)\cdot D_m W(t,x,\noise,m)(y)m(dy)-\sigma_x\int_{\reels^d} \text{div}_y (D_mW(t,x,\noise,m)(y))m(dy)\\
\displaystyle -2\beta \int \text{div}_{x}[D_mW ](t,x,\noise,m,y)]m(dy)-\beta \int \text{Tr}\left[D^2_{mm}W(t,x,\noise,m,y,y')\right]m(dy)m(dy')\\
\displaystyle =G(x,\noise,m,W) \text{ in } (0,T)\times\reels^d\times\reels^m\times\mathcal{P}_2(\reels^d),\\
\displaystyle W(0,x,\noise,m)=W_0(x,\noise,m) \text{ for } (x,\noise,m)\in\reels^d\times\reels^m\times\mathcal{P}_2(\reels^d).
 \end{array}
 \right.
\end{equation*}
for any $T>0$.
\end{corol}
\begin{proof}
    Rather than working directly on the Lipschitz solution to this equation we may consider an equation with additional variables and without any additive common noise using Lemma \ref{lemma: equiv additive/ add variable}. The variable coming from common noise (say $\noise^\beta$) is treated like an autonomous noise process as in Section \ref{subsection: autonomous W2} and a doubling of variable ($\noise$,$\noiseb$) is made in the variable associated to the noise process $b$ only. From there on, there is no difficulty in extending the arguments of section \ref{subsection: b(p,m) W2}. 
\end{proof}
Obviously, this is also true for all other long time existence results presented in this paper, in particular Theorem \ref{thm: existence b(p,m) W1}.
\begin{remarque}
    This stays true even when the Brownian motion associated to the common noise $(B^c_s)_{s\geq 0}$ is correlated with the common noise from the noise process $(B^\noise_s)_{s\geq 0}$. Indeed, the only change in this case is that there is now a cross derivatives term 
    \[-\text{Tr}\left(\Gamma D^2_{(\noise^\beta,\noise,\noiseb)}W\right),\]
    with a matrix $\Gamma$ of the form 
    \[
    \left(
\begin{array}{ccc}
     \Gamma_\beta& \Gamma_{cross} & \Gamma_{cross}  \\
     \Gamma_{cross} & \Gamma_\noise &\Gamma_\noise\\
     \Gamma_{cross} & \Gamma_\noise &\Gamma_\noise\\
\end{array}
\right)
\geq 0.
\]
At a point of minimum this term with cross derivatives has a sign, and so there is no further difficulty in applying the maximum principle. 
\end{remarque}
\section*{acknowledgements}
\begin{center}
The two authors acknowledge a partial support from the chair FDD (Institut Louis Bachelier)
\end{center}
\nocite{*}
\printbibliography

\end{document}